\documentclass[11pt,twoside]{amsart}

\usepackage{amsmath, amsthm, amssymb}

\input xy
\xyoption{all}

\newcommand{\Hom}{\mathrm{Hom}}
\newcommand{\Ext}{\mathrm{Ext}}

\newcommand{\brokrarr}{\vphantom{\to}\mathrel{\smash{{-}{\rightarrow}}}}

\newcommand{\Ker}{\mathrm{Ker}}
\newcommand{\Aut}{\mathrm{Aut}}
\newcommand{\sdp}{\mathbin{{>}\!{\triangleleft}}} 
\newcommand{\Alt}{\mathrm{A}}   
\newcommand{\GL}{\mbox{\boldmath$\rm GL$}}
\newcommand{\PGL}{\mbox{\boldmath$\rm PGL$}}
\newcommand{\Sp}{\mbox{\boldmath$\rm Sp$}}
\newcommand{\bG}{\mbox{\boldmath$\rm G$}}
\newcommand{\SL}{\mbox{\boldmath$\rm SL$}}
\newcommand{\U}{\mbox{\boldmath$\rm U$}}

\newcommand{\rank}{\mathrm{rank}}

\newcommand{\Char}{\mathrm{\rm char\,}} 
\newcommand{\diag}{\mathrm{\rm diag}}
\newcommand{\Gal}{\mathrm{Gal}}
\newcommand{\galois}{\Gal}
\newcommand{\lra}{\longrightarrow}
\newcommand{\SO}{\mbox{\boldmath$\rm SO$}}
\newcommand{\Orth}{\mbox{\boldmath$\rm O$}}
\newcommand{\M}{\mathrm{M}}        
\newcommand{\ord}{\mathop{\rm ord}\nolimits}
\newcommand{\Sym}{{\mathrm{S}}}    
\newcommand{\tr}{\mathrm{\rm tr}}
\newcommand{\trace}{\tr}

\newcommand{\Res}{\mathrm{Res}}
\newcommand{\Sha}{\mbox{\rus{\fontsize{11}{11pt}\selectfont{SH}}}}
\renewcommand{\H}{\mathcal{H}}
\newcommand{\gen}[1]{\langle{#1}\rangle}
\newcommand{\C}{\mathcal{C}}

\newcommand{\Z}{\mathbb{Z}}
\newcommand{\Gm}{{\mathbb{G}_m}}
\newcommand{\Gmr}{\mathbb{G}_m^r}
\newcommand{\GG}{\mathbb{G}}
\newcommand{\rk}{\mathrm{rank}}
\newcommand{\PP}{\mathbb{P}}
\newcommand{\Cr}{\mathrm{Cr}}
\newcommand{\Bir}{\mathrm{Bir}}
\newcommand{\Q}{\mathbb{Q}}
\newcommand{\Quot}{\mathrm{Quot}}
\newcommand{\Spec}{\mathrm{Spec}}
\newcommand{\Pic}{\mathrm{Pic}}

\def\A{\mathbb{A}}

\def\cN{\mathcal{N}}
\def\cH{\mathcal{H}}
\def\cG{\mathcal{G}}
\def\Z{\mathbb{Z}}
\def\Q{\mathbb{Q}}
\def\R{\mathbb{R}}
\def\G{\mathcal{G}}

\def\SO{\rm{SO}}
\def\SL{\rm{SL}}
\def\e{\mathbf{e}}
\def\Fp{\mathbb{F}_p}
\def\Ind{\mathrm{Ind}}
\def\Cl{\mathrm{Cl}}
\def\F{\mathbb{F}}
\def\x{\mathbf{x}}
\def\y{\mathbf{y}}
\def\id{\mathrm{id}}
\def\calO{\mathcal{O}}
\def\B{\mathcal{B}}
\def\0{\mathbf{0}}
\def\v{\mathbf{v}}
\def\w{\mathbf{w}}
\def\aa{\mathbf{a}}
\def\bb{\mathbf{b}}
\def\bB{\mathbf{B}}
\def\bA{\mathbf{A}}
\def\bG{\mathbf{G}}
\def\bF{\mathbf{F}}
\def\wr{\mathrm{wr}}
\def\Syl{\mathrm{Syl}}

\newtheorem{theorem}{Theorem}[section]
\newtheorem{lemma}[theorem]{Lemma}

\newtheorem{prop}[theorem]{Proposition}

\newtheorem{cor}[theorem]{Corollary}
\theoremstyle{definition}
\newtheorem{defn}[theorem]{Definition}
\newtheorem{example}[theorem]{Example}
\newtheorem{remark}[theorem]{Remark}
\newtheorem{notation}[theorem]{Notation}

\title
{Four-dimensional Algebraic Tori}

\author{Nicole Lemire}

\begin{document}
\begin{abstract} 
The study of the birational properties of algebraic $k$-tori began in the
 sixties and seventies with work of Voskresenkii, Endo, Miyata, Colliot-Th\'el\`ene and Sansuc. There was particular interest in determining the rationality of a given algebraic $k$-tori. As rationality problems for algebraic varieties are in general difficult, it is natural to consider relaxed notions such as stable rationality, or even retract rationality.  Work of the above authors
and later Saltman in the eighties determined 
necessary and sufficient conditions to determine when an algebraic torus is 
stably rational, respectively retract rational in terms of the integral 
representations of its associated character lattice. An interesting question
is to ask whether a stably rational algebraic $k$-torus
is always rational. In the general case, there exist examples of non-rational stably rational $k$-varieties.  Algebraic $k$-tori of dimension $r$ are classified up to isomorphism by conjugacy classes of finite subgroups of $\GL_r(\Z)$.
This makes it natural to examine the rationality problem for algebraic $k$ 
tori of small dimensions.
In 1967, Voskresenskii~\cite{Vos67} proved that all algebraic tori of dimension 2 are rational.  In 1990, Kunyavskii~\cite{Kun90} determined which algebraic tori of dimension 3 were rational.  In 2012, Hoshi and Yamasaki~\cite{HY12} determined which algebraic tori
of dimensions 4 and 5 were stably (respectively retract) rational with the aid of GAP.  They did not address the rationality question in dimensions 4 and 5.
In this paper, we show that all stably rational algebraic $k$-tori of dimension 4 are rational, with the possible exception of 10 undetermined cases.
Hoshi and Yamasaki found 7 retract rational but not stably rational dimension 4 algebraic $k$-tori.  We give a non-computational proof of these results.
\end{abstract}

\maketitle

\section{Introduction}

Rationality problems in algebraic geometry are central but notoriously difficult questions, making it natural to consider relaxed notions.
Let $k$ be a field and  let $X$ be a $k$-variety.
$X$ is \emph{$k$-rational} if it is birationally isomorphic to projective $n$ space, $\PP^n_k$; $X$ is \emph{stably $k$-rational} if $X\times_k \A^r$ is rational for some $r\ge 0$.
$X$ is \emph{retract $k$-rational} if there exist rational maps 
$f:X\brokrarr \A^n$  and $g:\A^n\brokrarr X$ such that $g\circ f=\id_X$.
$X$ is \emph{$k$-unirational} if there exists a dominant rational map
$\PP^m_k\brokrarr X$, for some $m$.
Note that rationality implies stably rationality implies retract rationality
implies unirationality.
It is not too hard to find non-stably rational retract rational $k$-varieties
- we will discuss such examples later.  It is much more difficult to find non-rational stably rational $k$-varieties - although examples were found by
 Beauville, Colliot-Th\'el\`ene, Sansuc and Swinnerton-Dyer~\cite{BCTSSD85}
in 1985. 
In this paper, we examine the rationality problem for algebraic $k$-tori of dimension 4 where $k$ is a field of characteristic 0.  For algebraic $k$-tori, the question of whether
there exist non-rational stably rational algebraic $k$-tori remains open to the 
best of the author's knowledge.

For any fixed quasi-projective $k$-variety $X$, a $k$-form of $X$ 
is another $k$-variety $Y$ which is isomorphic to $X$ after extending to the 
separable closure $k_s$ of $k$.
Isomorphism classes of $k$-forms of $X$ are in bijection with elements
of non-abelian Galois cohomology set $H^1(\G_k,\Aut(X))$ where
$\G_k=\Gal(k_s/k)$ is the absolute Galois group.

An algebraic $k$-torus $T$ of dimension $r$ is a $k$-form of a split $k$-torus so that
$$T\times_k k_s\cong \GG^r_{m,k_s}$$
Algebraic $k$-tori of dimension $r$ are determined up to isomorphism
by continuous representations of $\G_k$ into $\GL_r(\Z)$.
Any algebraic $k$-tori of dimension $r$ is split by a finite Galois extension $L$ over $k$.
That is, there exists a finite Galois extension $L/k$ such that
$$T\times_k L\cong \GG^r_{m,L}$$
Algebraic $k$-tori of dimension $r$ which are split by a finite Galois extension $L/k$ are in bijection with conjugacy classes of finite subgroups of $\GL_r(\Z)$.  
More precisely, a finite subgroup $G$ of $\GL_r(\Z)$ up to conjugacy  determines a $G$-lattice $M_G$ up to isomorphism. 
Then $T_G=\Spec(L[M_G]^G)$ is an algebraic $k$-torus of dimension $r$  split by $L$.
Conversely, the character lattice of an algebraic $k$-torus of dimension $r$ split by $L$ is a $G$-lattice of rank $r$ and so determines a conjugacy class of finite subgroups of $\GL_r(\Z)$.

Given a finite subgroup $G$ of $\GL_r(\Z)$ and a Galois $G$-extension $L/k$,
the function field
$T_G=\Spec(L[M_G]^G)$ is the quotient field $K=L(M_G)^G$ of $L[M_G]^G$.
This concrete description allows us to take advantage of the fact that each of the rationality concepts defined above for a $k$-variety $X$ can be rephrased 
in terms of its function field $K=k(X)$.
$X$ is $k$-rational if and only if $K/k$ is a rational extension, i.e.
there exist  $z_1,\dots, z_n\in K$ which are algebraically independent over $k$
and such that $K=k(z_1,\dots,z_n)$.
$X$ is stably $k$-rational if and only if there exists a field $L$ containing $K$ which is rational over both $k$ and $K$.
$X$ is retract $k$-rational if and only if  $K$ contains a $k$-algebra $R$ such that $K$ is the quotient field of $R$ and the
identity map $1_R$ factors through the localisation of a polynomial ring over $k$. 
$X$ is $k$-unirational if and only if $k\subseteq K \subseteq k(x_1,\dots,x_m)$ for some $m$.

Voskresenskii, Endo, Miyata, Colliot-Th\'el\`ene, Sansuc, Saltman studied the rationality problem for algebraic $k$-tori from the sixties to the 
eighties~\cite{Vos67,Vos70,Vos74,Vos83,EM73,EM73p1,EM74,EM75,EM76,CTS77,CTS87,Sa84a}.  They determined conditions under which a algebraic $k$-torus is stably
rational, and retract rational respectively.  The conditions were phrased in terms of the character lattice of the algebraic $k$-torus as a lattice over its
splitting group.  

For algebraic tori of small dimensions, the question of retract/stable rationality has been addressed.  Since algebraic $k$-tori of dimension $r$ split by a Galois extension $L/k$ are determined up to isomorphism by conjugacy classes of finite subgroups of $\GL_r(\Z)$, this is a finite set by a theorem of Jordan.
For small values of $r$, the classification of conjugacy classes of finite
subgroups is known but the number of such conjugacy classes grows rapidly 
with $r$.  

For $r=2$, there are 13 conjugacy classes of finite subgroups of $\GL_2(\Z)$ and hence 13 possible isomorphism classes of algebraic $k$-tori.
Voskresenskii~\cite{Vos67} showed in 1967 that all 2-dimensional algebraic $k$-tori are
rational.

For $r=3$, there are 73 conjugacy classes of finite subgroups of $\GL_3(\Z)$,
a classification due to Tahara~\cite{Tah71} in 1971.
In 1990, Kunyavskii~\cite{Kun90} classified algebraic $k$-tori of dimension 3 up to birational equivalence.  He found that all but 15 were rational. The remaining 15 were
not even retract rational.

There are 710 conjugacy classes of finite subgroups of $\GL_4(\Z)$.  Dade~\cite{Dad65}
found the maximal conjugacy classes of finite subgroups of $\GL_4(\Z)$ in 1965
without the use of a computer by analyzing the quadratic forms stabilized by the subgroups.  In 1971, B\"ulow, Neub\"user~\cite{BN71} determined all the conjugacy classes of finite subgroups of $\GL_4(\Z)$ using Dade's classification
and computer techniques.  In 1978, together with H. Brown, H. Wondratschek, and H. Zassenhaus~\cite{BBNWZ78}, they wrote a  book
on Crystallographic Groups of Four-Dimensional Space.
This classification determined the library of crystallographic groups of 
dimensions 2,3,4 which is programmed into GAP~\cite{GAP4}.
In 2012, Hoshi and Yamasaki~\cite{HY12}  determined which algebraic tori of dimensions 4 and 5 are stably (respectively retract) rational using GAP.
For dimension 4, they found that 487 algebraic tori are stably rational, 7 are
retract but not stably rational and the remaining 216 are not retract rational.
They did not address the issue of rationality.
In this paper, I show that all but possibly 10 of the stably rational algebraic tori
of dimension 4 are rational.  

An important source of examples of algebraic tori are the norm one tori.
Given a separable field extension $K/k$ of degree $n$ and $L/k$ the Galois closure of $K/k$, the norm one torus $R^{(1)}_{K/k}(\Gm)$ is the kernel of the 
norm map $R_{K/k}(\Gm)\to \Gm$ where $R_{K/k}$ is the Weil restriction~\cite[Section 3.12]{Vos98}.
Note $R_{K/k}(\Gm)$ is split by $L/k$. Let $G=\Gal(L/k)$ and $H=\Gal(L/K)$.
The character lattice of $R^{(1)}_{K/k}(\Gm)$ is given by the $G$-lattice $J_{G/H}$, the dual of the kernel of the augmentation ideal of the permutation 
lattice $\Z[G/H]$.

The rationality problem for norm one tori was studied by many authors
~\cite{EM74,CTS77,CTS87,Leb95,CK00,LL00,Flo2,End11}.
The stable and retract rationality of norm one tori corresponding to Galois
field extensions is well understood by work of Endo, Miyata, Saltman, Colliot-
Th\'el\`ene, Sansuc.

\begin{theorem} Let $K/k$ be a finite Galois extension with Galois
group $G$.  Let $R^{(1)}_{K/k}(\Gm)$ be the corresponding norm 1 torus.
\begin{enumerate}
\item
$R^{(1)}_{K/k}(\Gm)$ is retract rational
if and only if the Sylow subgroups of $G$ are cyclic.~\cite{EM74,Sa84a}
\item 
$R^{(1)}_{K/k}(\Gm)$ is stably rational if and only if the 
Sylow subgroups of $G$ are cyclic and $\hat{H}^0(G,\Z)=\hat{H}^4(G,\Z)$
where $\hat{H}$ indicates Tate cohomology.~\cite{EM74,CTS77}

Note that the conditions above are equivalent to the following restrictions on the structure of the group $G$:
$$G=C_m\mbox{ or } G=C_n\times\langle \sigma,\tau: \sigma^k=\tau^{2^d}=1,
\tau\sigma\tau^{-1}=\sigma^{-1}\rangle$$
where $d\ge 1,k\ge 3, n,k$ odd and $\gcd(n,k)=1$.
\end{enumerate}
\end{theorem}

These theorems were proven in the 1970s.  Some cases of non-Galois separable
extensions were dealt with much more recently.

\begin{theorem}~\cite{End11}
Let $K/k$ be a finite non-Galois separable field extension
and let $L/k$ be the Galois closure of $K/k$. Let $H\le G$ be the Galois groups
of $L/K$ and $L/k$ respectively.  
\begin{enumerate}
\item If $G$
is nilpotent, then the norm one torus $R^{(1)}_{K/k}(\Gm)$ is not retract rational.
\item If all the Sylow subgroups of $G$ are cyclic, 
 then the norm one torus $R^{(1)}_{K/k}(\Gm)$ is retract $k$-rational and, we
have that 
$R^{(1)}_{K/k}(\Gm)$ is stably $k$-rational if and only if 
either $G=D_{2n}$ where $n$ is odd or $G=C_m\times D_{2n}$ where $n$ is odd and 
$\gcd(m,n)=1$ and $H\le D_{2n}$ is a cyclic group of order 2.
\end{enumerate}
\end{theorem}

\begin{theorem}~\cite{CTS87,Leb95,CK00,LL00,End11} 
Let $K/k$ be a finite non-Galois separable field extension and $L/k$  be the Galois closure of $K/k$ such that the Galois group of $L/k$ is $S_n$ and that of $L/K$ is $S_{n-1}$.
Then
\begin{enumerate} 
\item $R^{(1)}_{K/k}(\Gm)$ is retract $k$-rational if and only if $n$ is prime.
\item $R^{(1)}_{K/k}(\Gm)$ is (stably) $k$-rational if and only if $n=3$.
\end{enumerate}
\end{theorem}

\begin{theorem}~\cite{End11} Let $K/k$ be a non-Galois separable field extension of degree
$n$ and let $L/k$ be the Galois closure of $K/k$ such that the Galois 
group of $L/k$ is $A_n, n\ge 4$ and that of $L/K$ is $A_{n-1}$. 
Then
\begin{enumerate}
\item $R^{(1)}_{K/k}(\Gm)$ is retract $k$-rational if and only if $n$ is prime.
\item For some positive integer $t$, $[R^{(1)}_{K/k}(\Gm)]^{(t)}$ is stably $k$-rational
if and only if $n=5$.
\end{enumerate}
\end{theorem}

S. Endo asked in~\cite{End11} about the stable rationality of $R^{(1)}_{K/k}(\Gm)$
in the $A_5$ case.  Hoshi and Yamasaki~\cite{HY12} used GAP to show that in fact
$R^{(1)}_{K/k}(\Gm)$ is stably rational in this case.
In this paper, we present a non-computational proof of that fact.
The norm one torus $R^{(1)}_{K/k}(\Gm)$ corresponding to $A_5$ has 
character lattice $J_{A_5/A_4}$. 
It is one of the ten algebraic $k$-tori of dimension 4 which is stably
rational but whose rationality is unknown.
Another of these exceptional stably rational algebraic $k$-torus of dimension 
4 whose rationality is unknown is intimately 
related to this norm one torus. In fact its splitting group is $A_5\times C_2$ 
and its character lattice restricts to  $J_{A_5/A_4}$ on $A_5$.
The character lattices of the remaining 8 stably rational algebraic $k$-tori whose rationality is unknown are also intimately related.  
See Section 5 for more details.

This paper is organised as follows.
In Section 2, we will recall some facts about algebraic $k$-tori and 
conditions for stable and retract rationality.
In Section 3, we will determine some families of hereditarily rational algebraic $k$-tori.  We call an rational algebraic $k$-torus corresponding to a finite subgroup 
$G$ of $\GL_r(\Z)$ \emph{hereditarily rational} if an algebraic $k$-tori corresponding to any subgroup of $G$ is rational.  We do not know whether all rational algebraic $k$-tori are hereditarily rational but we show this to be the case for
a number of natural families of algebraic tori.
In Section 4,  we show that 8 of the 10  maximal conjugacy classes of 
$\GL(4,\Z)$ which correspond to stably rational tori from Hoshi and Yamasaki's list have corresponding algebraic tori which are hereditarily rational. 
 We then use GAP to show that the
set of rational $k$-tori obtained from subgroups of these groups  is of size 477
and matches the set of stably rational $k$-tori of dimension 4 obtained
by Hoshi and Yamasaki with 10 undetermined exceptions.  Note that the use of GAP
is very minimal: it is limited to finding  conjugacy classes of finite subgroups of $\GL_4(\Z)$ which are subgroups of the groups corresponding to the 8 hereditarily rational algebraic $k$-tori mentioned above.
In Sections 5 and 6, we give non-computational proofs that the remaining two maximal
algebraic $k$-tori of dimension 4 are stably rational, recovering the results of Hoshi and Yamasaki.
We also show that 7 algebraic $k$-tori of dimension 4 are retract but not stably rational in a non-computational way, also recovering the results of Hoshi
and Yamasaki. 
We remark that some rationality results for stably rational algebraic tori of dimension 5 have been found by Armin Jamshidpey in a thesis which is soon to be defended.

\section{Preliminaries}
 
We begin with some remarks on notation used in this paper.  
Throughout, $k$ will be a field of characteristic zero.  We will see that this
hypothesis is necessary to apply the criteria for stable and retract rationality.

For finite groups, we will denote by $C_n$
the cyclic group of order $n$; by $D_{2n}$ the dihedral group of order $2n$;
by $S_n$ the symmetric group on $n$ letters and by $A_n$ the alternating group 
on $n$ letters.

We will discuss below root system terminology, algebraic torus - lattice correspondence, lattice terminology and birational properties of algebraic tori.

\subsection{Root Systems}  
For more information on root systems, see, for example,
Humphreys~\cite[Chapter III]{Hum72}. 
Here we will review the notation that we will use on root systems.
Note that we use root systems on vector spaces over $\Q$, instead of $\R$,
but as our root systems are all crystallographic, this does not affect the theory.
Let $\Phi$ be a root system on a finite dimensional $\Q$ vector space $V$
equipped with a fixed symmetric bilinear form $(\cdot,\cdot)$.
Then, by definition, $V=\Q\Phi$ and 
the reflection $s_{\alpha}:V\to V$ in the root $\alpha\in \Phi$,
is given by 
$$s_{\alpha}(x)=x-\langle x,\alpha\rangle \alpha$$
where
$$\langle x,y\rangle=2\frac{(x,y)}{(y,y)},x,y\in V$$

The Weyl group of $\Phi$, or the group generated by these reflections, will be denoted by $W(\Phi)$.
Recall that for a basis $\Delta$ for the root system $\Phi$, we have
$$W(\Phi)=\langle s_{\alpha}:\alpha\in \Phi\rangle =\langle s_{\alpha}:\alpha\in \Delta\rangle$$ 
so that the Weyl group can be generated by just the simple reflections, those with respect to roots in $\Delta$.
Recall that $\Aut(\Phi)$, the group of automorphisms of the root system $\Phi$,
is 
$$\Aut(\Phi)=\{g\in \GL(\Q\Phi): \langle gv,gw\rangle =\langle v,w\rangle, v,w
\in \Q\Phi\}$$
The Weyl group $W(\Phi)$ is a subgroup of $\Aut(\Phi)$.
We will denote by $\Z\Phi$, the  root lattice of $\Phi$.  $\Z \Phi$ is the $\Z$-span of $\Phi$ 
and has $\Z$-basis $\Delta$, where $\Delta$ is a basis of $\Phi$.
 Its weight lattice is 
$$\Lambda(\Phi)=\{x\in \Q\Phi: \langle x,\alpha\rangle\in \Z\mbox{ for all }\alpha\in \Phi\}.$$
$\Z \Phi\subseteq \Lambda(\Phi)$ are lattices of the same rank
and both are stabilised by the finite subgroup $\Aut(\Phi)$ and so by its
subgroup $W(\Phi)$.  

A root system is irreducible if $\Phi$ cannot be partitioned into two proper 
orthogonal subsets.
Root systems can be decomposed into a disjoint union of irreducible root systems.  There is a classification of irreducible root systems.
There are 4 infinite families $A_n,B_n,C_n,D_n$ and exceptionals $E_6,E_7,E_8,
F_4,G_2$. 
For our applications, we will mainly be referring to the irreducible root systems of type 
$A_n$, $B_n$, $F_4$ and $G_2$.  We will refer to Humphreys~\cite[III,12,13]{Hum72} for explicit constructions of the irreducible root systems, their bases,
Weyl groups, weight lattices and automorphism groups. 

We make some particular remarks about automorphism groups of root systems that will be used subsequently.
The automorphism group of $A_n$ is $W(A_n)\times C_2=S_{n+1}\times C_2$.
As an $W(A_n)=S_{n+1}$ representation, the root lattice of $A_n$ is $I_{X_{n+1}}$, the augmentation ideal of the 
$S_{n+1}$-permutation lattice corresponding to the natural transitive $S_{n+1}$-set $X_{n+1}$ with stabilizer subgroup $S_n$.  As an $\Aut(A_n)$-representation, the root lattice of $A_n$ is 
$I_{X_{n+1}}\otimes I_{Y_2}$ where $X_{n+1}$ is the $\Aut(A_n)$-set on which $C_2$ acts trivially and $S_{n+1}$ acts as above; and $Y_2$ is the $\Aut(A_n)$-set of size 2 on which $C_2$ acts transitively and $S_{n+1}$ acts trivially.
Note that $\Aut(A_2)=W(G_2)$ and the root lattice of $G_2$ and $A_2$ coincide.
This implies that the root lattice of $G_2$ is $I_{X_3}\otimes I_{Y_2}$ as 
an $\Aut(A_2)=S_3\times C_2$-lattice.




\subsection{Lattice-Tori Correspondence}
We will discuss in more detail the correspondence between isomorphism 
classes of algebraic $k$-tori of dimension $r$ and conjugacy classes
of finite subgroups of $\GL_r(\Z)$ as well as criteria 
for determining whether a given algebraic $k$-torus is retract or stably 
rational.~\cite[Chapter 2]{Vos98}
 
An algebraic $k$-torus of dimension $r$ is a $k$-form of a split $k$-torus
$\Gmr$. As discussed in the introduction, this implies that 
isomorphism classes of algebraic $k$-tori are in bijection 
with elements of the non-abelian Galois cohomology
set $H^1(\cG_k,\Aut(\Gmr))$ where $\cG_k=\Gal(k_s/k)$ is the absolute Galois
group of $k$. Since $\cG_k$ acts trivially on $\Aut(\Gmr)=\GL_r(\Z)$,
elements of $H^1(\cG_k,\Aut(\Gmr))$ are actually continuous homomorphisms
of the compact profinite group $\cG_k$ into the discrete group $\GL_r(\Z)$
and as such have finite image.
This shows that algebraic $k$-tori of dimension $r$ are determined up to isomorphism
by continuous representations of $\cG_k$ into $\GL_r(\Z)$
or equivalently by  lattices of rank $r$ with a continuous action of $\cG_k$ up to isomorphism.

Given such a continuous representation $\rho:\cG_k\to \GL_r(\Z)$,
and the associated $\cG_k$ lattice $M$ of rank $r$, 
$\Spec(k_s[M]^{\cG_k})$ is an algebraic $k$-torus.
Conversely, an algebraic $k$-torus $T$ determines its character
lattice $\hat{T}=\Hom(T,\Gm)$ which is a lattice equipped with a continuous action of $\cG_k$.

Every algebraic $k$-torus is split by a finite Galois extension.
In fact, if the algebraic $k$-torus is determined by a continuous representation
$\rho:\cG_k\to \GL_r(\Z)$ with associated $\cG_k$ lattice $M$ of rank $r$,
then for $N=\ker(\rho)$, $L=k_s^N$ is a finite Galois extension of $k=k_s^{\cG_k}$
with finite Galois group $G\cong \cG_k/N\cong \rho(\cG_k)$.
Note that $L[M]^G=(k_s[M]^N)^{\cG_k/N}=k_s[M]^{\cG_k}$.
Isomorphism classes of algebraic tori of rank $r$ split by a finite Galois extension $L/k$
with Galois group $G$  are in bijection with isomorphism classes of $G$-lattices of rank $r$. Here the $G$-lattice $M$ determines the algebraic torus
split by $L/k$ as $\Spec(L[M]^G)$ and conversely the algebraic torus $T$ with 
splitting group $G$ determines its character lattice which is a $G$-lattice.

Criteria for determining the stable (respectively retract)
rationality of an algebraic $k$-torus split by a finite Galois extension $L/k$ 
with Galois group $G$ are phrased in terms of the integral representation theory of the character lattice as a $G$-lattice.
To describe these criteria, we need some definitions about $G$-lattices.

We introduce some notation: 
For a finite Galois extension $L/k$ with Galois group $G$, we will denote
by $C(L/k)$, the category of algebraic $k$-tori split by $L$
and by $C(G)$, the dual category of $G$-lattices (torsion-free $G$ modules)
of finite rank.

\subsection{Lattice terminology}
For more details, see for example, Lorenz~\cite{Lor05}.
\begin{defn}
\begin{enumerate}
Let $G$ be a finite group. Note that for a $G$-lattice 
$M$ and a subgroup $H$ of $G$, $\hat{H}^k(H,M)$ refers to the 
$k$th Tate cohomology group of $M$ as an $H$ module, where $k\in \Z$.
\item For a $G$-lattice  the dual lattice $M^*=\Hom(M,\Z)$ is a $G$-lattice
with $(g\cdot f)(m)=g\cdot f(g^{-1}\cdot m)$ for $f\in M^*$, $g\in G$, $m\in M$.
\item 
A  \emph{permutation $G$-lattice} is a $G$-lattice with a finite $\Z$
basis which is permuted by the action of $G$.  All transitive
permutation lattices with stabilizer subgroup $H$ are isomorphic to 
$\Z[G/H]$ where $G/H$ is the set of left cosets of $H$ in $G$.
Any permutation $G$-lattice is the direct sum of a finite number 
of transitive $G$ permutation lattices.  

\item A \emph{sign permutation $G$-lattice} is a $G$-lattice with 
$\Z$-basis which is permuted by the action of $G$ up to sign.

\item Given an $H$-lattice $N$ where $H\le G$, $\Ind^G_H(M)=\Z G\otimes_{\Z H}M$
is the $G$-lattice induced from the $H$ lattice $M$.
Note that $\Ind^G_H(\Z)=\Z[G/H]$.
\item An \emph{invertible} or \emph{permutation projective} $G$-lattice $M$
is a $G$-lattice which is a direct summand of a $G$ permutation lattice.
That is, there exists a $G$-lattice $M'$ such that $M\oplus M'=P$
for some $G$ permutation lattice $P$.
\item A $G$-lattice $M$ is \emph{quasipermutation} if there exists 
a short exact sequence $0\to M\to P\to Q\to 0$ of $G$-lattices
with $P,Q$ $G$-permutation lattices.

\end{enumerate}

\end{defn}

The following  are some natural $G$-lattices associated with each $G$-permutation lattice:
\begin{defn}
For a finite $G$-set $X$, let the associated permutation $G$-lattice be denoted as $\Z[X]$.  A natural $G$ sublattice is the augmentation ideal $I_X$ 
given by the kernel of the $G$-equivariant homomorphism
$\epsilon_X:\Z[X]\to \Z$, sending $x\to 1$ for all $x\in X$.
But then 
$$0\to I_X\to \Z[X]\to \Z\to 0$$ 
is a short exact sequence of $G$-lattices.
Let $J_X=(I_X)^*$ be its $\Z$ dual.
Then $J_X$ satisfies the short exact sequence of $G$-lattices
$$0\to \Z\to \Z[X]\to J_X\to 0$$
\end{defn}

\begin{defn}~\cite{CR90}
Facts and Theorems about Tate Cohomology:

Let $G$ be a finite group and $M$ be a $G$-lattice. Then 
\begin{enumerate}
\item Let $N_G:M\to M^G$ be the norm map $N_G(m)=\sum_{g\in G}g\cdot m$.
Then
$$
\begin{cases}\hat{H}^k(G,M)=H^k(G,M)&\mbox{ if }k\ge 1\\
\hat{H}^0(G,M)=\ker_M(N_G)/I_G(M)& \\
\hat{H}^{-1}(G,M)=M^G/\ker_M(N_G)&\\
\hat{H}^{-i-1}(G,M)=H_i(G,M)&\mbox{ if }i\ge 1
\end{cases} 
$$
\item (Duality)
$\hat{H}^{k}(G,M)\cong \hat{H}^{-k}(G,M^*)$ where $M^*$ is the $\Z$ dual lattice of $M$.
\item (Shapiro's Lemma): For an $H$-lattice $N$, where $H\le G$, we have
$$\hat{H}^k(G,\Ind^G_H(N))\cong \hat{H}^k(H,N)$$
where $\Ind^G_H(M)=\Z G\otimes_{\Z H}M$ is the $G$-lattice induced from the 
$H$-lattice $M$.
\end{enumerate}
\end{defn}

\begin{defn}
\begin{enumerate}
\item A $G$-lattice $M$ is \emph{flasque} if $\hat{H}^{-1}(H,M)=0$ for all subgroups $H\le G$.
\item A $G$-lattice $M$ is \emph{coflasque} if $\hat{H}^1(H,M)=0$ for all subgroups 
$H\le G$.
\end{enumerate}
\end{defn}

Following Voskresenskii, we denote by $S(G)$ the class of all permutation $G$-lattices, $D(G)$ the 
class of all invertible $G$-lattices, $\hat{H}^{-1}(G)$ the class of all
flasque $G$-lattices, $\hat{H}^1(G)$ the class of all coflasque $G$-lattices
and $\tilde{H}(G)=H^1(G)\cap H^{-1}(G)$.
Then 
$$S(G)\subset D(G)\subset \tilde{H}(G)\subset H^i(G)\subset C(G)$$
where each inclusion is proper.
Most inclusions are clear: permutation lattices
are flasque and coflasque due to  Shapiro's Lemma, so invertible lattices 
must also be flasque and coflasque as the direct summands of permutation lattices.
Note also permutation lattices are self-dual and that the dual of a flasque $G$-lattice is coflasque (and vice-versa).

The following definitions and results are due to Voskresenskii, Endo-Miyata, Colliot-Th\'el\`ene and Sansuc
\begin{defn} 
We say that two $G$-lattices $M,N$ are \emph{stably isomorphic} and write
$[M]=[N]$ if and only if there exist permutation lattices $P,Q$ such that 
$M\oplus P\cong N\oplus Q$. 
\end{defn}

\begin{theorem}(See ~\cite[Lemma 2.6.1]{Lor05})
A $G$-lattice $M$ has a flasque resolution.  That is, 
there exists a short exact sequence of $G$-lattices
$$0\to M\to P\to F\to 0$$
such that $P$ is $G$-permutation and $F$ is $G$-flasque.
Given 2 flasque resolutions of a $G$-lattice $M$
$$0\to M\to P_i\to F_i\to 0, i=1,2$$
we have that $[F_1]=[F_2]$.
\end{theorem}

\begin{defn}
We may then define the \emph{flasque class} of the $G$-lattice $M$
to be $\rho_G(M)=[F]$ where $0\to M\to P\to F\to 0$ is a flasque resolution
and $[F]$ is the stable isomorphism class of $F$.
\end{defn}

\begin{remark} The proof that every $G$-lattice has a flasque resolution (see eg. ~\cite[Lemma 2.6.1]{Lor05}) is straightforward and constructive
but relies strongly on knowledge of the conjugacy classes of subgroups of $G$
and the restrictions of the $G$-lattice $M$ to these subgroups.
Hoshi and Yamasaki~\cite{HY12} gave algorithms to construct these flasque 
resolutions for $G$-lattices of ranks up to 5 for which all of this information is  known.
\end{remark}

\begin{defn}
We say that $G$-lattices $M$ and $N$ are \emph{flasque equivalent} if 
there exist short exact sequences of $G$-lattices such that 
$$0\to M\to E\to P\to 0, 0\to N\to E\to Q\to 0$$
Note that $M\sim_G 0$ if and only if $M$ is a quasipermutation $G$-lattice.
Note that  $M\sim_G N$ if and only if $\rho_G(M)=\rho_G(N)$. (See~\cite[Lemma 2.7.1]{Lor05}).
\end{defn}

\subsection{Birational properties of algebraic $k$-tori}  

We will give a brief summary of how the concepts in the 
last section can be used to discuss birational properties of algebraic $k$-tori.
A more in depth discussion can be found in~\cite[Chapter 2]{Vos98}.

Given an algebraic $k$-torus $T\in C(L/k)$ where $L/k$ is a finite Galois
extension with group $G$, we may find by resolution of singularities (due 
to Hironaka in characteristic 0) 
a smooth projective $k$-variety $X$ which contains $T$ as an open subset.
[Note, this is the reason for our characteristic zero assumption.]
$X$ is called a projective model of $T$.  Then $X$ and $T$ are birationally 
isomorphic.  Then $X_L=L\otimes_k X$ is rational since $X_L$ and $T_L$ are 
birationally equivalent.  Then the Picard group of $X_L$, $\Pic(X_L)$, is a $G$-lattice. Given two such projective models $X,Y$ of $T$, $\Pic(X_L)$
is stably $G$ isomorphic to $\Pic(Y_L)$.  So the class of $[\Pic(X_L)]$
is a birational invariant.  
For any intermediate field extension $k\subset K\subset L$ with $H=\Gal(L/K)$,
$\hat{H}^{\pm 1}(H,\Pic(X_L))$ are also birational invariants of $T$
and $\hat{H}^{-1}(H,\Pic(X_L))=0$.
The inclusion $T_L$ into $X_L$ induces 
a short exact sequence of $G$-lattices
$$0\to \hat{T}\to \hat{S}\to \Pic(X_L)\to 0$$
where $\hat{S}$ is generated by the components of the closed subset $X_L-T_L$.
$\hat{S}$ is a $G$-permutation lattice and as observed above, $\Pic(X_L)$
is a flasque $G$-lattice.

This gives a geometric construction of a flasque resolution of the 
$G$-lattice $\hat{T}$.
If $T$ is $k$-rational, so is $X$  and then $[\Pic X_L]=0$.
In fact if $T$ is stably $k$-rational then so is $X$ and so 
$X\times_k \Gmr$ is rational for some $r\ge 0$ but $\Pic(X_L\times_k \GG^r_{m,L})=\Pic(X_L)$ and so we again have $[\Pic(X_L)]=0$.

This gives the basic setup behind the following results
connecting birational properties of algebraic $k$-tori to properties
about the integral representation of their character lattices over a splitting group.

\begin{theorem}
(Endo-Miyata, Voskresenskii, Saltman) Let $L/k$ be a finite
Galois extension with Galois group $G$ and let $T,T'\in C(L/k)$
with character lattices $\hat{T}, \hat{T'}$ in $C(G)$.
\begin{enumerate}
\item (\cite[Theorem 1.6]{EM73}) $T$ is stably $k$-rational if and only if
$\hat{T}\sim_G 0$ (equivalently $\hat{T}$ is a quasipermutation $G$-lattice). 
\item (\cite[Theorem 2]{Vos74})  $T$ and $T'$ are stably birational $k$ equivalent
if and only if $\hat{T}\sim_G \hat{T'}$.
\item (\cite[Theorem 3.14]{Sa84a}) $T$ is retract $k$-rational if and only 
if $\rho_G(\hat{T})$ is invertible.
\end{enumerate}
\end{theorem}

\begin{defn}
A torus $T\in C(L/k)$ with permutation character lattice 
$\oplus_{i=1}^k\Z[G/H_i]$ for subgroups $H_i, i=1,\dots,k$
corresponds to an algebraic $k$-torus
$T=\prod_{i=1}^kR_{K_i/k}(\Gm)$ where $K_i=L^{H_i}$ and 
$R_{K/k}$ refers to Weil restriction of scalars.
Such tori are called quasisplit.
\end{defn}
 
\begin{theorem}~\cite{EM73,Len74}. See also~\cite[Theorem 9.5.3]{Lor05}.
 Let $L/k$ be a Galois extension with Galois group $G$.
\begin{enumerate}
\item
A torus $T\in C(L/k)$ is rational if $\hat{T}\in C(G)$ is permutation.
\item
If there exists an exact sequence $0\to S\to T\to T'\to 0$ of tori in $C(L/k)$
where $S$ is quasisplit, then $T$ is birationally equivalent to $T'\times_k S$.
\end{enumerate}
\end{theorem}

Note that the last theorem is equivalent to the following in lattice theoretic terms:

\begin{theorem}
Let $L/k$ be a Galois extension of fields with $G=\Gal(L/k)$.
Then
\begin{enumerate}
\item $L(P)^G$ is rational over $L^G$ if $P$ is a permutation $G$-lattice.
\item If $0\to M\to N\to P\to 0$ is an exact sequence of $G$-lattices with 
$P$ $G$-permutation, then $L(N)\cong L(M)(P)$ as $G$ fields and so
$L(N)^G\cong L(M)(P)^G$ is rational over $L(M)^G$.
\end{enumerate}
\end{theorem}

\section{Families of Rational algebraic $k$-tori}

In this section, we will gather families of rational algebraic $k$-tori from
results in the literature.  

\begin{defn} 
Let $M$ be a $G$-lattice for a finite group $G$.  If all algebraic tori with character lattice $\Res_H^G(M)$ and splitting group $H$ are rational for any subgroup $H$ of $G$, we call $M$ \emph{hereditarily rational}.  Note that a $G$-lattice $M$ is 
 hereditarily  rational if and only if for  any subgroup $H$ of $G$
and for any Galois extension $K/k$ with Galois group $H$, 
the function field $K(M)^H$ is rational over $K^H$.

Let $T$ be an algebraic $k$-torus of dimension $r$, $h_T:\cG_k\to \GL_r(\Z)$
the associated continuous representation and $W_T=h_T(\cG_k)\le \GL_r(\Z)$.
We call $T$ \emph{hereditarily  rational} if the associated $W_T$-lattice $\hat{T}$
is hereditarily rational.
\end{defn}

\begin{remark} In the definition of hereditarily rational algebraic $k$-torus, rationality is completely determined by the structure of its character lattice as a $\cG_k$ representation.
Note that an algebraic $k$-torus $T$ is determined uniquely
by the $G$-lattice $\hat{T}$ and the Galois extension $L/k$ with Galois group
$G$.  Every $G$ module $\hat{T}$ determines many algebraic $k$-tori
corresponding to different Galois extensions $L/k$ with Galois group $G$
and these tori may be non-isomorphic over $k$.  
Stable rationality and retract rationality of an algebraic $k$-torus only 
depend on the $G$-lattice $\hat{T}$.  However it is not clear whether this
is true for rationality, although the author does not know of a counterexample.
A reason for caution is given by the example given in~[p. 54]\cite{Vos98} of 2 norm one 
tori $R^{(1)}_{L_i/\Q}(\Gm)$, $i=1,2$ corresponding to distinct biquadratic
extensions which are not even stably birationally equivalent over $\Q$.   

\end{remark}

We will now produce families of examples of hereditarily rational algebraic $k$-tori.  These are gathered from known examples in the literature, which will be cited accordingly.

\begin{prop}
Let $T$ be an algebraic $k$-torus with Galois splitting field $K/k$, 
$G=\Gal(K/k)$ and character lattice $M$.  Then $T$ is hereditarily rational in the following cases:
\begin{enumerate}
\item~\cite[1.4]{Len74} $M$ is a permutation $G$-lattice.  ($T$ is then a quasi-split torus).
\item~\cite[4.8]{Vos98} $M=I_X$ is an augmentation ideal for a $G$ set $X$.
\item~\cite[8.2, Example 5]{Vos98} $M$ is a sign-permutation $G$-lattice.
\item~\cite[1.3]{Len74} $T_M$ is $G$-equivariantly linearisable.    
\item~\cite{Kly88,Flo1} $M=I_X\otimes I_Y$ is a tensor product of augmentation ideals
for finite $G$-sets $X,Y$ with $\gcd(|X|,|Y|)=1$.
\item~\cite[4.9, Example 7]{Vos98} $M$ is the root lattice for the $G_2$ root system for the group $G=W(G_2)$.
\end{enumerate}
\end{prop}
\begin{proof}
\begin{enumerate}
\item
Any permutation $G$-lattice $P$ is hereditarily rational since 
$K(P)^G$ is rational over $k$ for any Galois $G$ extension $K/k$~\cite[1.4]{Len74} and permutation
lattices are preserved under restriction.
Since quasi-split tori are precisely those with permutation character
lattices, these tori are hereditarily rational.
More explicitly, if $P=\oplus_{i=1}^r\Z[G/H_i]$ is a permutation $G$-lattice
where $H_1,\dots, H_r$ are subgroups of $G$, and $L/k$ is a Galois extension
with Galois group $G$, the corresponding quasi-split torus
$\prod_{i=1}^rR_{K_i/k}(\Gm)\in C(L/k)$ is hereditarily rational.
\item
Let $P=\Z X$ be a permutation $G$-lattice corresponding to any $G$ set $X$.
Then $\epsilon_X:\Z X\to \Z, x\to 1$ is a $G$ invariant surjection.
The augmentation ideal $I_X=\ker(\epsilon_X)$ is an hereditarily rational $G$ 
lattice. Suppose $\Z X=\oplus_{i=1}^s \Z[G/H_i]$.
 Given a Galois $G$ extension $K/k$, the exact sequence
$0\to I_X\to \Z X\to \Z\to 0$
corresponds to the exact sequence of $k$-tori
$$1\to \GG_{m,k}\to R_{F_1/k}(\Gm)\times \cdots \times R_{F_s/k}(\Gm)\to T\to 1$$
where $K_i/k$, $i=1,\dots,s$ are intermediate field extensions of $L/k$
such that $L/K_i$ is Galois with group $H_i$ for each $i$.
An algebraic $k$-torus with character lattice $I_X$ is given by 
$T=\prod_{i=1}^sR_{K_i/k}(\Gm)/\Gm$.
Since each $R_{F_i/k}(\Gm)$ can be identified with an open set of 
$R_{F_i/k}(\A^1)=\A^{n_i}_k$, $T=\prod_{i=1}^sR_{F_i/k}(\Gm)/\Gm$ admits an open embedding into 
projective space and hence is rational.
Note that for any subgroup $H$ of $G$, $(I_X)_H$ is also such an augmentation
ideal and so any algebraic torus with this character lattice would be 
rational by the same argument.
\item A geometric proof that a torus with sign permutation character lattice is rational is given  in 
~\cite[8.2, Example 5]{Vos98}.  Note here that an orthogonal integral representation corresponds to a sign permutation lattice in our language. The restriction of a sign permutation lattice to any subgroup is again a sign permutation lattice, so that tori with sign permutation character lattices are hereditarily rational.
\item $T_M$ is $G$-equivariantly linearisable if for any Galois $G$-extension
$K/k$, we have that $K(M)$
is $G$-equivariantly birational to $K(V)$ where $G$ acts $k$ linearly on $V$.
Then as $G$-fields, $K(M)\cong K(V)$.  Since $G$ acts faithfully on $K$ and linearly on $V$, $K(M)^G\cong K(V)^G$ is rational over $K^G$.
If $K(M)$ is $G$-equivariantly birationally linearisable, then  $K(\Res^G_H(M))$ is $H$ equivariantly birationally linearisable for any subgroup $H$ of $G$.
So a torus with Galois splitting field $K$ and character lattice $M$ is hereditarily rational.
Note that this gives an alternative proof that an  algebraic torus with a (sign) permutation or augmentation ideal character lattice is hereditarily rational.
For algebraic tori with permutation character lattices, this is clear.
For algebraic tori with sign permutation character lattices, 
$K(X_1,\dots,X_n)=K(Y_1,\dots,Y_n)$ where $Y_i=\frac{1-X_i}{1+X_i}$ for all $i=1,\dots,n$.  Then $\sigma(Y_i)=\pm Y_j$ for all $i=1,\dots, n$ and so the action of $G$ on $V=\sum_{i=1}^nkY_i$ is linear.
\item 
A result of Klyachko~\cite{Kly88} with a simpler proof given by Florence~\cite{Flo1} shows that for any
finite $\cG_k$ sets $X$ and $Y$ which are relatively prime, an algebraic $k$-torus with character lattice $I_X\otimes I_Y$ is rational.
This shows that for any finite group $G$, and any 2 relatively prime $G$ sets
$X,Y$, the $G$-lattice  $I_X\otimes I_Y$ is hereditarily rational.
Note that any $G$ set $Y$ of size 2 gives a rank 1 sign lattice $I_Y$.
Then for any other $G$ set $Y$ of odd order, $I_X\otimes I_Y$ is hereditarily 
rational.  In particular, this recovers the result for the root lattice of $G_2$
since it may be expressed as a tensor product 
$I_{X_3}\otimes I_{Y_2}$ as described earlier.  More generally, it shows that a algebraic $k$-torus
with character lattice $(\Z \bA_{2k},\Aut(\bA_{2k}))$~\cite[p.102]{Vos98} is hereditarily rational.

\item Note that the root lattice for the root system $\bG_2$ can also be described as the action of the automorphism group of the root system $\bA_2$ on the 
root lattice of $\bA_2$. 
The root lattice for $G_2$ restricted to $S_3$ is $I_{S_3/S_2}$ and $C_2$ acts as $-1$ on the lattice.  For an algebraic $k$-torus $T$ with character lattice $G_2$,
the group $W(G_2)$ acts regularly on $\GG_m^2$.  Voskresenskii shows that the action
of $W(G_2)$ on $\GG_m^2$ can be extended to a birational action on $(\PP^1)^2$
and so $T$ is an open part of the $k$-variety $((\PP^1)^2\otimes_k k_s)/W(G_2)$
which is a $k$-form of $\PP^1\times \PP^1$.  Since any $k$-form of $\PP^1\times \PP^1$
is rational if it has a point, $T$ is rational.
\end{enumerate}
\end{proof}

\begin{remark} Let $M=\oplus_{i=1}^rM_i$ be the direct sum of $G$-lattices 
$M_i$, $i=1,\dots,k$.  If $T_i,i=1,\dots,r$ are algebraic $k$-tori with character lattices $M_i$, then $T=T_1\times_k \cdots \times_k T_r$ is the algebraic $k$-torus with character lattice $\oplus_{i=1}^rM_i$.  If $K/k$ is a $G$-Galois extension, $K(M)$ is the composite of the $G$-fields $K(M_i)$  and so $K(M)^G$ is the composite of the fields $K(M_i)^G$.  So if each $T_i$ is (hereditarily) rational, so is $T$. 
So the direct sum of hereditarily rational $G$-lattices is hereditarily rational.
\end{remark}

\begin{defn} Let $M$ be a $G$-lattice of rank $r$.
Let $P_n=\oplus_{i=1}^n\Z\e_i$ be a permutation lattice for the wreath product $G^n\rtimes S_n$ on which $G^n$ acts trivially and $S_n$ acts by permuting the basis. Then $M\otimes P_n$ is the natural $G^n\rtimes S_n$ lattice of rank $rn$ for the wreath product.  Let $M_i=M\otimes \Z\e_i$.
Then $M\otimes P_n=\oplus_{i=1}^nM_i$.  Let $\pi_i:G^n\to G$ be the $i$th projection.  The action of $G^n\rtimes S_n$ on $M\otimes P_n$ is then given by
$$g\cdot (m\otimes \e_i)=(\pi_i(g)\cdot m)\otimes \e_i, g\in G^n$$
$$\sigma\cdot (m\otimes \e_i)=m\otimes \e_{\sigma(i)}, \sigma\in S_n$$

Suppose $T$ is an algebraic $F$-torus with character lattice $M$ and Galois
splitting field $K/F$ with Galois group $G$. 
If $F/k$ is a separable extension of degree $n$, then $R_{F/k}(T)$ is an algebraic $k$-torus with character lattice $M\otimes P_n$.
\end{defn}

\begin{prop}~\label{prop:wreath} Let $M$ be a $G$-lattice of rank $r$.
Using the notation from the definition,
let $M\otimes P_n$ be the natural $G^n\rtimes S_n$ lattice of rank $rn$.

Let $T^{n,\wr}_M$
be an algebraic torus with character lattice $M\otimes P_n$.
Then if any algebraic torus with character lattice $M$  is (stably) rational, so is $T^{n,\wr}_M$.
If any algebraic torus with character lattice $M$ is linearisable, so is $T^{n,\wr}_M$.
If any algebraic torus with character lattice $M$ is hereditarily rational, so is $T^{2,\wr}_{M}$.
\end{prop}

\begin{proof} Let $L/k$ be a Galois $G^n\rtimes S_n$-extension.
For preciseness, we will write $G^n=G_1\times \cdots \times G_n$ where $G_i=G$
and $M\otimes P_n=\oplus_{i=1}^nM_i$ where $M_i=M$.  Here $G_i$ acts on $M_i$ as $G$ acts on $M$ and $G_j$ acts trivially on $M_i,i\ne j$
so that $M\otimes P_n$ restricted to $G^n$ is $M^n$.
Now $L(M^n)$ is a composite of $G^n=\prod_{i=1}^nG_i$ fields $L(M_i)$.
Since any algebraic torus with character lattice $M$ is rational,
we have that $L(M_i)^{G^n}=L^{\prod_{j\ne i}G_j}(M_i)^{G_i}=L^{G^n}(f_1^i,\dots,f_r^i)$, for all
$i=1,\dots,n$.  
Then $L(M)^{G^n}=\prod_{i=1}^n(L^{\prod_{j\ne i}G_j}(M_i))^G_i=L^{G^n}(f_1^1,\dots,f_r^1,\dots,f_1^n,\dots,f_r^n)$
and $L(M\otimes P_n)^{G^n\rtimes S_n}=(L(M^n)^{G^n})^{S_n}$.
Since $S_n$ permutes the $n$ copies of $M$ and correspondingly permutes the 
fields $L^{\prod_{j\ne i}G_j}$, it permutes the corresponding invariants, and so $L(M\otimes P_n)^{G^n\rtimes S_n}$ is rational over $L^{G^n\rtimes S_n}$. 

If an algebraic torus with character lattice $M$ is $G$-linearisable,
then $L(M)\cong L(V)$ as $G$ fields where $G$ acts linearly on $V$.
  Then $L(M\otimes P_n)\cong L(V\otimes P_n)$
as $G^n\rtimes S_n$ fields.  

If an algebraic torus with character lattice $M$ is stably rational then 
$M$ is quasi-permutation and satisfies an exact sequence 
$$0\to M\to Q_1\to Q_2\to 0$$
where $Q_1,Q_2$ are permutation $G$-lattices.
Then 
$$0\to M\otimes P_n\to Q_1\otimes P_n\to Q_2\otimes P_n\to 0$$
is an exact sequence of $G^n\rtimes S_n$ lattices with $Q_i\otimes P_n$ 
$G^n\rtimes S_n$-permutation for $i=1,2$.

Let $n=2$ and let $H$ be any subgroup of $G^2\rtimes S_2$.
Then either $H\le G^2$ or $H/H\cap G^2\cong C_2$.
Let $N=H\cap G^2$.  Then $L(M_1\oplus M_2)=L(M_1)L(M_2)$ is a composite of $N$ fields.
Let $\pi_i$ be the restriction of the 
$i$th projection map $G^2\to G$ to $N$. Then $L(M_i)^N=L^{\ker(\pi_i)}(M_i)^{\pi_i(N)}$.  Since $\pi_i(N)\le G$, we see by assumption that $L(M_i)^N$ is rational over $L^N$, $i=1,2$ and so $L(M_1\oplus M_2)^N=L(M_1)^NL(M_2)^N$
is rational over $L^N$ as the composite of rational extensions.
So we need only consider the case when $H/N\cong C_2$.  Let $h\in H\setminus N$.  Then $h(M_1)=M_2$ and $h(L)=L$ shows that $L(M_2)=h(L(M_1))$.
We have that $L(M_1\oplus M_2)^N=L(M_1)^N(h(L(M_1))^N=L(M_1)^Nh(L(M_1)^N)$
since $N$ is normal in $H$.  Since $L(M_1)^N=L^N(f_1,\dots,f_r)$,
$h(L(M_1)^N)=L^N(hf_1,\dots,hf_r)$ and so 
$L(M_1\oplus M_2)^N=L^N(f_1,\dots,f_r,hf_1,\dots,hf_r)$.  Since $h^2\in N$, 
$h(hf_i)=f_i$ for all $i=1,\dots,r$.  Then $H/N$ acts by permutations on 
$L(M_1\oplus M_2)^N$ and so $L(M\otimes P_2)^H=L^N(f_1,\dots,f_r,hf_1,\dots,hf_r)^{H/N}$ is rational over $L^H$ as required.
\end{proof}

\begin{remark} If all algebraic tori with character lattice $M$ are hereditarily rational, an algebraic torus, $T^{n,\wr}_M$ with character lattice $M\otimes P_n$ is rational and hence stably rational.
This means that an algebraic torus with character lattice $M\otimes P_n$ corresponding to a subgroup of $G^n\rtimes S_n$ must also be stably rational.
However it is not clear whether such an algebraic torus must be hereditarily rational if $n\ge 3$.  
However for the case of 4 dimensional algebraic tori, the above result will be sufficient.
\end{remark}


\begin{prop} Let $M$ be an hereditarily  rational faithful $G$-lattice
and let $L$ be a $G$-lattice such that there exists an exact sequence
$$0\to M\to L\to P\to 0$$ 
for some permutation $G$-lattice $P$.
Then an algebraic $k$-torus with character lattice $(L,G)$ is hereditarily rational.
More precisely, if $T$ and $T'$ are algebraic $k$-tori fitting into an exact
sequence  
$$1\to S\to T\to T'\to 1$$
where $S$ is a quasi-split $k$-torus, then $T$ is birationally equivalent
to $T'\times_k S$.  If $T'$ is (hereditarily) $k$-rational, so is $T$.   
\end{prop}

\begin{proof}
Let $K/k$ be a Galois extension with Galois group $G$. Then 
$K(L)\cong K(M)(P)$ as $G$-fields. Then $K(L)^G\cong K(M)(P)^G$ is rational over
$K(M)^G$ which is in turn rational over $K^G$~\cite{Len74}.
The same result would hold for any subgroup $H$ and the $H$ lattice $L_H$.
\end{proof}



\section{Hereditarily rational $k$-tori in dimensions 2,3,4}

We now have enough information to show that all but 10 of the stably rational
algebraic $k$-tori of dimension 4 are hereditarily rational. 

We will first quickly illustrate our approach for the rank 2 and 3 cases
due to Voskresenskii and Kunyavskii respectively.
Note in order to do this, we will need to identify certain $G$-lattices
corresponding to conjugacy classes of finite subgroups of $\GL_r(\Z)$.
We remark that for a $G$-lattice $M$, we could use the character to 
identify the corresponding $\Q G$ module $\Q M$, but that 
is not sufficient to identify its isomorphism type as a $G$-lattice.
We identify $G$-lattices up to isomorphism using explicit isomorphisms. 

Given a finite subgroup $G\in \GL_r(\Z)$, we will denote by $M_G$ the associated $G$-lattice of rank $r$.
For the standard basis $\e_1,\dots,\e_r$ of $M_G$, and given $g=(a_{ij})_{i,j=1}^r\in \GL_r(\Z)$, the action of $G$ on $M_G$ will be given on \textbf{rows}
to agree with the notation in GAP and in Hoshi and Yamasaki's paper.
That is, $g\cdot \e_i=\sum_{j=1}^ra_{ij}\e_j$ for all $i=1,\dots,r$.

\begin{remark} We will frequently wish to show that a  $G$-lattice $L$ \
is isomorphic to  $I_X$ or its dual $J_X$ for some transitive $G$-set $X$ with stabilizer subgroup $H$.
Note that if there exists $\y\in L$ such that $\Z G\cdot \y=L$, $\sum_{g\in G}g\y=\0$ and $\rk(L)=[G:G_{\y}]-1$ then 
we see that for $H=G_{\y}$,
$\pi:\Z[G/H]\to L$ is a surjective $G$-homomorphism with 
$\ker(\pi)=\Z(\sum_{gH\in G/H}gH)$ so that $L\cong J_{G/H}$. 
If instead we show that $L^*\cong J_{G/H}$ in this way, then, we have
$L\cong I_{G/H}$.  
\end{remark}

\begin{notation}
A rank 1 $G$-lattice is determined by a group homomorphism $\chi:G\to \GL_1(\Z)=\{\pm 1\}$. If $\chi$ is trivial, we will write the lattice as $\Z$.
If $\chi:G\to \{\pm 1\}$ is non-trivial, it is completely determined by its kernel which is necessarily a normal subgroup $N$ of index 2 in $G$.
We will then write the rank 1 $G$ lattice corresponding to $\chi$ as $\Z^{-}_{N}$ where $N=\ker(\chi)$.
\end{notation}

\begin{notation}
(Recognising augmentation ideals and their duals).

Let $X_{n+1}=\{x_i:i=1,\dots,n+1\}$ be the standard $S_{n+1}$ set on which $S_{n+1}$ acts transitively with stabilizer subgroup $S_n$.
Let $\pi:\Z[X_{n+1}]\to J_{X_{n+1}}$ be the natural surjection.
Note that $\ker(\pi)=\Z(\sum_{i=1}^{n+1}x_i)$.
Then $\{\overline{x_i}:i=1,\dots,n\}$ forms a $Z$-basis of $J_{X_{n+1}}$
where $\overline{x_i}=\pi(x_i)$.
With respect to this basis,
$\sigma(\overline{x_i}))=\overline{x_{\sigma(i)}}, \sigma(i)\ne n+1$
and $\sigma(\overline{x_i})=-\sum_{i=1}^n\overline{x_{\sigma(i)}}, \sigma(i)=n+1$.
We will write $\rho_n:S_{n+1}\to \GL_n(\Z)$ for the representation of $S_{n+1}$ associated with $J_{X_{n+1}}$ with respect to the 
basis $\overline{X_{n+1}}$ determined on rows.
We will write $\rho_n^*:S_{n+1}\to \GL_n(\Z)$ for the dual of this representation of $S_{n+1}$ associated with $I_{X_{n+1}}\cong J_{X_{n+1}}^*$. 
Note that the matrices in the image of $\rho_n$ (resp. $\rho_n$) are either 
permutation matrices or permutation matrices with one row (resp. column)
replaced by $[-1,\dots,-1]$.  It is then easy to determine which 
permutation determined them.

We will extend the action of $S_{n+1}$ on $X_{n+1}$ to an action of 
$S_{n+1}\times C_2$ on $X_{n+1}$ by inflation.
That is, $S_{n+1}$ acts on $X_{n+1}$ as before and $C_2$ acts trivially on
$X_{n+1}$.

We will denote by $\rho_n^-:S_{n+1}\times C_2\to \GL_n(\Z)$ the 
representation associated to the $S_{n+1}\times C_2$-lattice 
$J_{X_{n+1}}\otimes \Z^-_{S_{n+1}}$ with respect to the $\Z$-basis 
$\overline{x_i}\otimes 1$. 
Then if $C_2=\langle \gamma\rangle$, note that $\rho_n^-(\sigma,1)=\rho_n(\sigma)$
and $\rho_n(\sigma,\gamma)=-\rho_n(\sigma)$ for all $\sigma\in S_{n+1}$.
The matrices in the image of $\rho_n^-$ are then also easy
to recognise.

We will denote by 
$(\rho_n^-)^*:S_{n+1}\times C_2\to \GL_n(\Z)$ the representation
corresponding to the dual lattice 
$I_{X_{n+1}}\otimes \Z_-^{S_{n+1}}=J_{Y_{n+1}}\otimes \Z_-^{S_{n+1}}$.  
Note that the root lattice $\Z A_n$ as a $W(A_n)=S_{n+1}$-lattice is $(I_{X_{n+1}})$
so  that the weight lattice $\Lambda(A_n)$ is
$J_{X_{n+1}}$ as a $W(A_n)=S_{n+1}$ lattice.
Note that $\Z A_n$ as an   $\Aut(A_n)=S_{n+1}\times C_2$ lattice is
$I_{X_{n+1}}\otimes \Z^-_{S_{n+1}}$ 
so that $\Lambda(A_n)$ is its dual
$J_{X_{n+1}}\otimes \Z^-_{S_{n+1}}$
as an $\Aut(A_n)$-lattice.

\end{notation}

\begin{notation}
With respect to the standard basis $\e_1,\dots,\e_n$, the Weyl group 
of $B_n$, has reflections $\tau_i=s_{\e_i}$, $i=1,\dots,n$ and 
$\sigma_{ij}=s_{\e_i-\e_j}$. On the root lattice $\Z B_n$, with $\Z$-basis $\e_1,\dots,\e_n$, the $\tau_i$ fix $\e_j,j\ne i$ and 
$\tau_i(\e_i)=-\tau_i(\e_i)$
and $\sigma_{ij}$ acts by swapping $\e_i$ and $\e_j$ and fixing the 
other basis elements.
So $W(B_n)=(C_2)^n\rtimes S_n$
where $C_2^n=\langle \tau_i:i=1,\dots,n\rangle$ and $S_n=\langle \sigma_{ij}:i\ne j\rangle$.
We denote
$\eta_n:W(B_n)\to \GL_n(\Z)$ 
by the representation of $W(B_n)$ corresponding to its lattice $\Z B_n$
with respect to the standard basis $\e_1,\dots,\e_n$.
Note that the images of $\eta_n$ are sign permutation matrices.
We will write elements of $W(B_n)$ as $\tau \sigma$ where $\tau\in C_2^n$
is a product of $\tau_i$ and $\sigma\in S_n$.
\end{notation}

\begin{prop}~\label{prop:d2n} Let $n$ be an odd integer 
and let $D_{2n}$ be the dihedral group 
of size $2n$ given by the presentation
$$D_{2n}=\langle \sigma,\tau: \sigma^n=1=\tau^2, \tau\sigma\tau^{-1}=\sigma^{-1}\rangle$$
Via the injective group homomorphism $\varphi:D_{2n}\to S_n$
given on generators by $\sigma\mapsto (1,2,\dots,n)$ and $\tau\mapsto \prod_{i=1}^{\frac{n-1}{2}}(i,n-i)$, $D_{2n}$ acts by restriction on the 
$S_n$-set $X_n$. 
Then 
$$J_{X_n}\cong I_{X_n}\otimes \Z^{-}_{C_n}$$
as $D_{2n}$-lattices.
Considering $X_n$ as a $D_{2n}\times C_2$-set by inflation,
then
$$J_{X_n}\otimes \Z^{-}_{D_{2n}}\cong I_{D_{2n}\times C_2/C_2\times C_2}\otimes \Z^{-}_{D_{2n}}$$
as $D_{2n}\times C_2$-lattices.
\end{prop}

\begin{proof}
Restricting the standard $S_n$-set $X_n$ to $D_{2n}$ via $\varphi$,
we see that $\sigma(x_i)=x_{i+1\bmod n}$ and $\tau(x_i)=x_{n-i}$
for all $i=1,\dots,n$.
In particular, $D_{2n}\cdot x_{\frac{n+1}{2}}=X_n$ and the stabilizer subgroup of $x_{\frac{n+1}{2}}$ is $\langle \tau\rangle\cong C_2$.
So $X_n\cong D_{2n}/C_2$ as a $D_{2n}$-set.

 Note that all elements of order 2 are conjugate in $D_{2n}$
so that there is a unique conjugacy class of subgroups isomorphic to $C_2$ in
$D_{2n}$.  
This implies that $J_{D_{2n}/C_2}$ and $I_{D_{2n}/C_2}$ are well-defined,
since $\Z[G/H]\cong \Z[G/gHg^{-1}]$.

By the remark, it will suffice to find an element $z$ of 
$I_{X_n}\otimes \Z^-_{C_n}$ such that the distinct elements of the orbit of $z$ under $D_{2n}$
form a $\Z$-basis for $I_{X_n}\otimes \Z^{-}_{C_n}$ and the stabilizer
subgroup of $z$ is a cyclic subgroup of order 2.

Take $z=(x_1-x_n)\otimes 1\in I_{X_n}\otimes \Z^-_{C_n}$.
We will show that $(D_{2n})_{z}=\langle \tau\rangle$, $\sum_{g\in D_{2n}}\cdot z=\0$ and $\Z D_{2n}\cdot z$ has rank $n-1$.  By our previous observations, this will imply that
$$I_{X_n}\otimes \Z^-_{C_n}=\Z D_{2n}\cdot z\cong J_{X_n}$$
as $D_{2n}$-lattices.

Note that $\tau(z)=z$.
Since 
$$\sigma^i((x_1-x_n)\otimes 1)=(x_{1+i}-x_{n+i\bmod n})\otimes 1,$$
it is clear that the stabilizer subgroup of $z$ in $D_{2n}$ is $\langle \tau\rangle$.
It is also clear that the orbit sums to zero.  
We need only check that $\Z D_{2n}\cdot z=I_{X_n}\otimes \Z^-_{C_n}$.

Let $\sum_{i=1}^nb_i(x_i\otimes 1)\in I_{X_n}\otimes \Z^-_{C_n}$.
Then $\sum_{i=1}^nb_i=0$. We need to show that we can solve
$$\sum_{i=1}^{n-1}a_i(x_i-x_{n+i\bmod n})\otimes 1 
=\sum_{i=1}^{n}b_i(x_i\otimes 1)$$ 
for some unique $a_i$ if $\sum_{i=1}^{n}b_i=0$.
We obtain the following equations: 
$$a_i-a_{i+1}=b_i, i=1,\dots,n-2; a_{n-1}=b_{n-1}, -a_1=b_n$$
One can easily see that these equations correspond to a matrix system
of the form $C\aa=\bb$ where 
$$C=\left[\begin{array}{rrrrr}-1&0&0&\dots&0\\1&-1&0&\dots&0\\0&1&-1&0&\dots\\
\dots&\dots&\dots&\dots&\dots\\0&\dots&0&1&-1\\0&0&\dots&\dots&1\end{array}\right]$$
is an $n\times n$ matrix.
Since the rows of this matrix add to 0, and the last $n-1$ rows form a triangular
system with ones on the diagonal, it is clear that one could solve this
system uniquely for $\bb=[b_1,\dots,b_{n}]^T\in \Z^n$ where $\sum_{i=1}^{n}b_i=0$.

So we have proved that $J_{X_n}\cong I_{X_n}\otimes \Z^-_{C_n}$.

For the last statement, 
we will write 
$$D_{2n}\times C_2=\langle \sigma,\tau,\gamma: \sigma^n=\tau^2=\gamma^2=1,
\tau\sigma\tau^{-1}=\sigma^{-1},\gamma\tau=\tau\gamma, \gamma\sigma=
\sigma\gamma\rangle$$
Note that $X_n$ as an $(D_{2n}\times C_2)$-set is isomorphic to 
$(D_{2n}\times C_2)/(C_2\times C_2)$.
Inflation preserves exactness, so in particular isomorphisms and commutes 
with tensor products.
So inflating the $D_{2n}$-lattice isomorphism
$$J_{X_n}\cong I_{X_n}\otimes \Z^-_{C_n}$$
we obtain a $D_{2n}\times C_2$-lattice isomorphism
\begin{equation}
J_{X_n}\cong I_{X_n}\otimes \Z^-_{C_n\times C_2}\label{eq:jinfl}
\end{equation}
Note that $$\Z^{-}_{C_n\times C_2}\otimes \Z^{-}_{D_{2n}}\cong \Z^{-}_{\langle \sigma,\tau\gamma\rangle}$$
Note also that 
$\langle \sigma,\tau\gamma\rangle\cong D_{2n}$
and $\langle \sigma,\tau\gamma,\gamma\rangle=\langle \sigma,\tau,\gamma\rangle=D_{2n}\times C_2$, $\langle \tau\gamma,\gamma\rangle=\langle \tau,\gamma\rangle=C_2\times C_2$.
So we see that after tensoring (\ref{eq:jinfl}) by $\Z^-_{D_{2n}}$ we obtain
$$J_{X_n}\otimes \Z^-_{D_{2n}}\cong 
I_{X_n}\otimes \Z^{-}_{\langle \sigma,\tau\gamma\rangle}\cong I_{X_n}\otimes \Z^-_{D_{2n}}$$
Note that although $H=\langle \sigma,\tau\gamma\rangle$ is a different non-conjugate subgroup isomorphic to $D_{2n}$, $I_{X_n}=I_{\langle \sigma,\tau\gamma,\gamma\rangle}/{\langle \tau\gamma,\gamma\rangle}$. 
\end{proof}

\begin{remark} This technical result has some very interesting consequences.
Note that $\Lambda (A_2)\cong J_{X_3}$ as an $S_3$-lattice.  The symmetric group $S_3$ is also 
the dihedral group $D_6$.  We then see that $\Lambda (A_2)\cong J_{X_2}\cong 
I_{X_2}\otimes \Z^-_{C_3}=I_{X_2}\otimes I_{Y_2}$
where $Y_2$ is the $S_3$ set of size 2 permuted by $S_3/C_3$ and fixed by $C_3$.
Using Proposition~\ref{prop:d2n}, we easily recover the fact that an algebraic torus
with character lattice $(\Lambda (A_2),S_3)$ is hereditarily rational.
Note also that we can recover the fact that an algebraic torus with character lattice $(\Z G_2,W(G_2))$
or equivalently $(\Z A_2,\Aut(A_2))$ is hereditarily rational.
More generally, it shows that an algebraic torus with character lattice
$J_{X_n}\otimes \Z^-_{D_{2n}}$, $n$ odd,  is hereditarily rational.
\end{remark}

Given a finite subgroup $G$ of $\GL_r(\Z)$ up to conjugacy, the lattice determined by $G$, $M_G$ is determined by the action of $G$ by multiplying elements 
of $\Z^r$ (considered as rows) by elements of $G$ on the right.
So for the standard basis $\e_1,\dots, \e_r$ of $\Z^r$,
$\e_i\cdot g=\sum_{j=1}^ra_{ij}\e_j$
where $g=[a_{ij}]_{i,j=1}^r\in \GL_r(\Z)$.

There is a library of conjugacy class representatives of finite
subgroups of $\GL_r(\Z)$ for $r=2,3,4$ in GAP.  
The maximal finite subgroups of $\GL_r(\Z)$ for $r=2,3,4$ are encoded in 
GAP as DadeGroup(r,k), in honour of Dade who determined the maximal
finite subgroups of $\GL_4(\Z)$ without the use of a computer.
I will use the GAP labelling to refer to conjugacy class representatives.
I will identify $M_G$ for each maximal finite subgroup of $\GL_r(\Z)$, $r=2,3,4$.  These results are probably folklore (at least for $r=2,3$) but are not phrased in these terms in the literature.

\begin{remark}
Note, in GAP, the command 
\begin{verbatim}DadeGroup(r,k);
\end{verbatim}
 where $r=2,3,4$ and $k$ is in the 
correct range for the rank gives 
\begin{verbatim}MatGroupZClass(r,i,j,k);
\end{verbatim}
The command 
\begin{verbatim}
GeneratorsOfGroup(M);
\end{verbatim}
 gives the group generators of the 
group M.  
For better legibility, one could use the command
\begin{verbatim}
for g in GeneratorsOfGroup(M) do PrintArray(g); od;
\end{verbatim}
The use of GAP to list generators is a convenience.  This
information could also be found in~\cite{BBNWZ78}.
\end{remark}    

\subsection{Algebraic $k$-tori of dimension 2}.
For rank 2, there are 13 conjugacy classes of finite subgroups.
There are 2 maximal such classes.

\begin{prop} (Voskresenskii)
The maximal finite subgroups of $\GL_2(\Z)$ up to conjugacy
are $G_i$=\textup{DadeGroup(2,i)}, $i=1,2$.
Then 
\begin{enumerate}
\item $M_{G_1}=(\Z B_2,W(B_2))=(\Z B_2,\Aut(B_2)$ and
\item $M_{G_2}=(\Z A_2,\Aut(A_2)=(\Z G_2,W(G_2))$.
\end{enumerate}
The corresponding algebraic $k$-tori are hereditarily rational.
\end{prop}

\begin{proof}
The generators of $G_1=$ DadeGroup(2,1) with GAP ID [2,3,2,1] given by 
$$\left[\begin{array}{rr}1&0\\0&-1\end{array}\right],\left[\begin{array}{rr}0&-1\\1&0\end{array}\right],-I_2$$
are the images of 
$$\tau_2,\tau_2(12),\tau_1\tau_2$$
under  the faithful representation $\eta_2:W(B_2)\to \GL(2,\Z)$ determined by 
$\Z B_2$ with respect to the standard basis determined earlier.
Since $\langle \tau_2,\tau_2(12),\tau_1\tau_2\rangle =\langle \tau_1,\tau_2,(12)\rangle=W(B_2)$, we may claim that $M_{G_1}=(\Z B_2,W(B_2))$ as required.
Since this is a  sign permutation lattice, the corresponding torus is 
hereditarily rational.

The generators of  $G_2=$ DadeGroup(2,2) with  GAP code [2,4,4,1] given by
$$\left[\begin{array}{rr}0&1\\1&0\end{array}\right],-I_2,
\left[\begin{array}{rr}0&-1\\1&-1\end{array}\right]$$
are the images of
$$((12),1),(1,\gamma),((132),1)$$
under the faithful representation $(\rho_2^-)^*:S_3\times C_2\to \GL_2(\Z)$ 
determined by the $S_3\times C_2$-lattice 
$I_{X_3}\otimes \Z^-_{S_3}$ described earlier.
We have already noticed that this coincides with $(\Z A_2,\Aut(A_2))$.
We need only check that $\langle ((12),1),((132),1),(1,\gamma)\rangle=S_3\times C_2$, which is clear.

So for rank 2, the lattices corresponding to maximal finite subgroups of $\GL_2(\Z)$  are $(\Z G_2,W(G_2))=(\Z A_2,\Aut(A_2))$ and $(\Z B_2,W(B_2))$.  As explained above,
these are hereditarily  rational and so the corresponding tori and those
corresponding to their subgroups are rational.  This is effectively
a rephrasing of how Voskresenskii proves that all algebraic k-tori of dimension 2 are rational.
\end{proof}

\subsection{Algebraic $k$-tori of dimension 3}.

For rank 3, the 73 conjugacy classes of finite subgroups of $\GL_3(\Z)$
were determined by
Tahara~\cite{Tah71}  There are 4 maximal such classes.
Kunyavskii~\cite{Kun90} classified the algebraic $k$-tori of dimension 3 up to birational
equivalence. For each of the tori corresponding to maximal subgroups, he 
constructed a nonsingular projective toric model of the algebraic $k$ 
tori and used the geometric construction of the flasque resolution of the 
character lattice of each to understand birational properties of the 
algebraic $k$-tori corresponding to maximal subgroups.
See Kunyavskii~\cite{Kun07} and the description of his work in Voskresenskii~\cite{Vos98}.

\begin{prop} 
Let $G_k$=\textup{DadeGroup(3,k)} for $k=1,\dots,4$.
Then the corresponding lattices are:
\begin{enumerate}
\item $(M_{G_1},G_1)=(\Z A_2\oplus \Z A_1,\Aut(A_2)\times \Aut(A_1))$.
\item $(M_{G_2},G_2)=(\Z B_3,W(B_3))=(\Z B_3,\Aut(B_3))$.
\item $(M_{G_3},G_3)=(\Lambda(A_3),\Aut(A_3))$.
\item $(M_{G_4},G_4)=(\Z A_3,\Aut(A_3))$.
\end{enumerate}
\end{prop}

\begin{proof}    
We claim DadeGroup(3,1) with GAP ID [3,6,7,1] corresponds to 
$(\Z A_2\oplus \Z A_1,\Aut(A_2\times A_1))$.
$\Aut(A_2)$ acts as above on $\Z A_2$ and trivially on $\Z A_1$
and $\Aut(A_1)$ acts as $-1$ on $\Z A_1$ and trivially on $\Z A_2$.
 
The generators given by GAP are
$$-I_3,\diag(A,-1),\diag(-I_2,1),\diag(B,1)$$
where 
$$ A=(\rho_2^-)^*((12),1), B=(\rho_2^-)^*((132),1)$$
One may replace these generators by 
$$\diag(A,1),\diag(B,1),\diag(-I_2,1),\diag(I_2,-1)$$
Then by the above argument, we have already seen that $\langle A,B,-I_2\rangle$
determines $(\Z A_2,\Aut(A_2))$ and so we clearly 
have that the full group determines $(\Z A_2\oplus \Z A_1,\Aut(A_2)\times \Aut(A_1))$.


We claim that DadeGroup(3,2) with GAP ID [3,7,5,1] corresponds to $(\Z B_3,W(B_3))$.
It is clear from the matrix generators that this corresponds to a sign
permutation lattice.
Recall the faithful representation $\eta_3:W(B_3)\to \GL_3(\Z)$ determined by $\Z B_3$ 
with respect to the standard basis. 
The generators given by GAP are the images under $\eta_3$ of the following elements of $W(B_3)$:
$$\tau_1\tau_2\tau_3,\tau_3(12),(132),\tau_1\tau_3,\tau_1\tau_2$$
so it is not hard to see that the generators
can be replaced by
$$\langle \tau_1,\tau_2,\tau_3,(12),(132)\rangle\cong C_2^3\rtimes S_3$$

DadeGroup(3,3) with GAP ID [3,7,5,2] corresponds to $(\Lambda(A_3),\Aut(A_3))$
where $G=\Aut(A_3)=S_4\times C_2$.
The generating set given by GAP are the images under $\rho^-_3:S_4\times C_2\to \GL(3,\Z)$ of the elements:
$$(\id,\gamma),((3,4),\gamma),((1,3,2),1),
((13)(24),1),((12)(34),1)$$

where  $\rho_3^-:S_4\times C_2\to \GL_3(\Z)$
is the representation determined by $J_{X_4}\otimes \Z^-_{S_4}$
with respect to a  natural basis described earlier.
It suffices to check that $\langle (3,4),(1,3,2),(1,3)(2,4),(1,2)(3,4)\rangle=S_4$ which is straightforward.
Note that $\Aut(A_3)$ acts on $\Lambda(A_3)$ as
$J_{X_4}\otimes \Z^-_{S_4}$, where $X_4$ is the natural $S_4\times C_2$-set.


DadeGroup(3,4) with GAP code [3,7,5,3] corresponds to $(\Z A_3,\Aut(A_3))$.
The matrix generators of this group given by GAP are the transposes of 
those for DadeGroup(3,4).  So the corresponding lattice is accordingly the 
$\Z$ dual $S_4\times C_2$-lattice $I_{X_4}\otimes \Z^-_{S_4}$ which indeed is the 
representation of $\Aut(A_3)$ on $\Z A_3$.

Note that the last 3 Dade groups for dimension 3 are all $\Z$-forms
of the root lattice of $A_3$.  [Since $\SL_4/C_2\cong \SO_3$, the character lattice
of $\SL_4/C_2$ would be a $\Z$-form of the character lattice of $\SL_4/C_4$
which is $\Z A_3$.] 
\end{proof}

Although this is not how Kunyavskii determined the rational algebraic $k$ 
tori of dimension 3, the following  argument is more or less equivalent.
He did not need to determine which ones were maximal rational as the numbers
were relatively small.

\begin{prop} (Rational algebraic tori of dimension 3)
The maximal finite subgroups of $\GL(3,\Z)$ corresponding to hereditarily rational tori of dimension 3 have the following lattices:
\begin{enumerate}
\item $(\Z A_2\oplus \Z A_1,\Aut(A_2)\times \Aut(A_1))$
\item $(\Z B_3,W(B_3))=(\Z B_3,\Aut(B_3))$.
\item $(\Z A_3,W(A_3))$
\item $(L,W(B_2))$ where $L$ fits into a short exact sequence of 
$W(B_2)$ lattices
$$0\to \Z B_2\to L\to \Z\to 0$$
\end{enumerate}
There are 58 conjugacy classes of finite subgroups of $\GL(3,\Z)$ which are conjugate to a subgroup of one of the above 4 groups.  These correspond to the list of 58 rational algebraic tori of dimension 3. 
\end{prop}

\begin{proof} 
We note that the lattice $(\Z A_2\oplus \Z A_1,\Aut(A_2)\times \Aut(A_1))$
is hereditarily rational as it is a direct sum of hereditarily rational
lattices.  This corresponds to DadeGroup(3,1) with GAP code [3,6,7,1].

The lattice $(\Z B_3,W(B_3))$ is hereditarily rational  as it is a sign permutation lattice.  It corresponds to DadeGroup(3,2) with GAP code [3,7,5,1].

The lattice $(\Z A_3,W(A_3))=(I_{X_4},S_4)$ is hereditarily rational.
 From our identification of DadeGroup(3,4) with lattice $(\Z A_3,\Aut(A_3))$ we see that this should correspond to  a maximal subgroup.
The subgroup with GAP code [3,7,4,3]
has generators 
$$\rho_3^*((3,4)), \rho_3^*((1,3,2)), \rho_3^*((13)(24)), \rho_3^*((12)(34))$$
where $\rho_3^*:S_4\to \GL_3(\Z)$ is the dual representation of $\rho_3$ 
and hence corresponds to the lattice $I_{X_4}$.

The group with  GAP code [3,4,5,2] is abstractly isomorphic to $D_8$.
The following are a set of generators given by GAP:
$$B_1=\left[\begin{array}{rrr}0&1&-1\\0&1&0\\-1&1&0\end{array}\right],
B_2=\left[\begin{array}{rrr}0&1&0\\0&1&-1\\-1&1&0\end{array}\right],B_3=\left[\begin{array}{rrr}0&1&-1\\1&0&-1\\0&0&-1\end{array}\right]$$
We note that the sublattice spanned by $\{\e_1-\e_2,\e_3\}$ is stable under the 
action of the group.  We then recompute the action of the 
generators on the $\Z$-basis $\{\e_1-\e_2,\e_3,\e_2\}$ or equivalently
conjugate the generators by the change of basis matrix.
Then the conjugate generators are 
$$C_1=\left[\begin{array}{rrr}0&-1&0\\-1&0&0\\0&0&1\end{array}\right],
C_2=\left[\begin{array}{rrr}0&1&0\\-1&0&0\\0&-1&1\end{array}\right],C_3=\left[\begin{array}{rrr}-1&0&0\\0&-1&0\\1&-1&1\end{array}\right]$$
With respect to this new basis, it is clear that
the corresponding lattice $L$ fits into the short exact sequence of $W(B_2)\cong D_8$
lattices
$$0\to \Z B_2\to L\to \Z\to 0$$
where we recall that $W(B_2)=C_2^2\rtimes C_2\cong D_8$.
Then we see that this lattice is also hereditarily  rational.

We can then check using GAP that the union of conjugacy classes of subgroups of
the groups with the above GAP IDs corresponds to  the complete list of 58 
rational algebraic $k$-tori given by Kunyavskii. We will give more
details of our minimal GAP calculations after the dimension 4 case.
\end{proof}

\subsection{Algebraic $k$-tori of dimension 4}
We now examine the dimension 4 case.
The classification of maximal finite subgroups of $\GL_4(\Z)$ up to conjugacy
is due to Dade.  There are 9 maximal finite subgroups.
There are 710 conjugacy classes of finite subgroups of $\GL_4(\Z)$.

\begin{prop} 
Let $G_k=$\textup{DadeGroup(4,k)} for $k=1,\dots,9$.
Then the corresponding lattices are:
\begin{enumerate}
\item $(M_{G_1},G_1)=(\Z B_2\oplus \Z A_2,\Aut(B_2)\times \Aut(A_2))$.
\item $(M_{G_2},G_2)=(\Lambda A_3\oplus \Z A_1,\Aut(A_3)\times \Aut(A_1))$.
\item $(M_{G_3},G_3)=(\Z A_3\oplus \Z A_1,\Aut(A_3)\times \Aut(A_1))$
\item $(M_{G_4},G_4)=(L,((W(A_2)\times W(A_2))\rtimes C_2)\times C_2)$
where $L$ is the non-trivial intermediate lattice between $\Z A_2\oplus \Z A_2$
and $\Lambda(A_2)\oplus \Lambda(A_2)$.
\item $(M_{G_5},G_5)=(\Z A_2\oplus \Z A_2, (\Aut(A_2)\times \Aut(A_2))\rtimes C_2)$.
\item $(M_{G_6},G_6)=(\Z A_4,\Aut(A_4))$.
\item $(M_{G_7},G_7)=(\Lambda(A_4),\Aut(A_4))$.
\item $(M_{G_8},G_8)=(\Z B_4,W(B_4))=(\Z B_4,\Aut(B_4))$.
\item $(M_{G_9},G_9)=(\Z F_4,W(F_4))=(\Z F_4,\Aut(F_4))$.
\end{enumerate}
\end{prop}

\begin{proof}
$G_1=$ DadeGroup(4,1) with GAP ID [4,20,22,1]
corresponds to $(\Z B_2\oplus \Lambda(A_2),W(B_2)\times \Aut(A_2))$
which is hereditarily rational as the direct sum of hereditarily rational 
lattices.
The generators of this group are given as 
$$\{\diag(I_2,B_1),\diag(A_1,B_2),\diag(A_2,I_2),-I_4\}$$
where 
$$A_1=\eta_2(\tau_2),A_2=\eta_2(\tau_1 (12)),$$
and 
$$B_1=\rho^-_3((23),1), B_2=\rho^-_3((13),1)$$

Since $B_1B_2=\rho^-_3((132),1)$ has order 3, and $A_1$ has order 2,
$\diag(A_1,I_2)=\diag(A_1,B_1B_2)^3\in G_1$.  Since also $A_2^2=-I_2$, we may 
replace the above set of generators by 
$$\{\diag(A_1,I_2),\diag(A_2,I_2),\diag(-I_2,I_2),\diag(I_2,B_1),\diag(I_2,B_2),
\diag(I_2,-I_2)\}$$
Then the associated lattice is $M=M_1\oplus M_2$ where
where $M_1=\Z\e_1\oplus \Z\e_2$ and $M_2=\Z\e_3\oplus \Z\e_4$
are both $G$ invariant.
Then since 
$$A_1=\eta_2(\tau_2),A_2=\eta_2(\tau_2 (12)),-I_2=\eta_2(\tau_1\tau_2\tau_3)$$
and 
$$B_1=\rho^-_3((23),1), B_2=\rho^-_3((13),1),-I_2=\rho^-_3(1,\gamma)$$
we may see that the lattice $M_1$ is $(\Z B_2,W(B_2))$ and the 
lattice $M_2$ is $(J_{X_3}\otimes \Z^-_{S_3},S_3\times C_2)$, where 
$X_3$ is the natural $S_3\times C_2$-set.
$M_2$ corresponds to the natural action of $\Aut(A_2)$ on $\Lambda(A_2)$.
So $(M_{G},G)=(\Z B_2\oplus \Lambda(A_2),\Aut(B_2)\times \Aut(A_2))$.
The lattice is hereditarily  rational as the direct sum of 2 hereditarily  
rational lattices. 

DadeGroup(4,2) has GAP ID [4,25,11,2]. 
The generators are 
$$\{\diag(1,A_1),\diag(1,A_2),\diag(1,A_3),\diag(-1,A_4),-I_4\}$$
where
$$A_1=\left[\begin{array}{rrr}0&1&0\\1&0&0\\0&0&1\end{array}\right],
A_2=\left[\begin{array}{rrr}1&0&-1\\1&-1&0\\1&0&0\end{array}\right],
A_3=\left[\begin{array}{rrr}-1&0&0\\-1&0&1\\-1&1&0\end{array}\right],
A_4=\left[\begin{array}{rrr}0&-1&1\\-1&0&1\\0&0&1\end{array}\right]$$
We then clearly see that the lattice decomposes as a direct sum of  
$M_1=\Z\e_1$ and $M_2=\oplus_{i=2}^4\Z \e_i$.
To determine $M_2$, we
note that for the faithful representation of $S_4\times C_2$ given
by $(\rho_3^-)^*:S_4\times C_2\to \GL(3,\Z)$ corresponding to the 
$S_4\times C_2$ lattice $I_{X_4}\otimes \Z^-_{S_4}$, we observe that
$$A_1=(\rho_3^-)^*((12),1),-A_2=(\rho_3^-)^*((1,3,4),1),A_3=(\rho_3^-)^*((14)(23),1),-A_4=(\rho_3^-)^*((12)(34),1)$$
Since $A_2^3=-I_3$, we see that the elements
$$\diag(1,A_1),\diag(1,-A_2),\diag(1,A_3),\diag(1,-I_3)$$
are contained in this group. 
Note that $\langle A_1,-A_2,A_3,-I_3\rangle=(\rho_3^-)^*(S_4\times C_2)$
since 
$$\langle ((12),1),(134),1),((14)(23),1),(1,\gamma)\rangle=S_4\times C_2.$$
So $H=(1,(\rho_3^-)^*(S_4\times C_2))$ is a subgroup of $G_2$.
Since $(1,A_4)\in H\le G_2$, then we also have $(-1,I_3)\in G_2$ and 
so $G_2=\langle (-1,I_3)\rangle \times (\rho_3^-)^*(S_4\times C_2)$. 
This shows that $M_2$ corresponds to the $S_4\times C_2$-lattice
$I_{X_4}\otimes \Z^-_{S_4}$.
This is the natural action of $\Aut(A_3)$ on $\Z A_3$.
So the lattice corresponding to DadeGroup(4,2) is $(\Z A_1\oplus \Z A_3,\Aut(A_1)\times \Aut(A_3))$.

$G_3=$DadeGroup(4,3) has GAP ID [4,25,11,4]. 

The generators are 
$$\{\diag(1,B_1),\diag(1,B_2),\diag(1,B_3),\diag(-1,B_4),-I_4\}$$
where
$$B_1=\left[\begin{array}{rrr}0&1&0\\1&0&0\\0&0&1\end{array}\right],
B_2=\left[\begin{array}{rrr}0&0&-1\\0&-1&0\\1&1&1\end{array}\right],
B_3=\left[\begin{array}{rrr}-1&-1&-1\\0&0&1\\0&1&0\end{array}\right],
B_4=\left[\begin{array}{rrr}0&-1&0\\-1&0&0\\1&1&1\end{array}\right]$$
We then clearly see that the lattice decomposes as a direct sum of  
$M_1=\Z\e_1$ and $M_2=\oplus_{i=2}^4\Z \e_i$.
To determine $M_2$, we
note that 
$$\rho_3^-((12),1)=B_1,\rho_3^-((134),1)=-B_2,\rho_3^-((14)(23),1)=B_3,
\rho_3^-((12)(34),1)=-B_4$$
where $\rho_3^-:S_4\times C_2\to \GL_3(\Z)$ is the representation
corresponding to $J_{X_4}\otimes \Z^-_{S_4}$
which in turn corresponds to $\Aut(A_3)$ acting on $\Lambda(A_3)$.
A similar argument to the one for $G_2$ shows that
$G_3=\langle (-1,I_3)\rangle \times \rho_3^-(S_4\times C_2)$
and the group determines the lattice
 $(\Z A_1\oplus \Lambda(A_3),\Aut(A_1)\times \Aut(A_3))$.

$G_4$=DadeGroup(4,4) with GAP ID [4,29,9,1]
is abstractly isomorphic to $((S_3\times S_3)\rtimes C_2)\times C_2$
and has generators given by 
$$\left[\begin{array}{rrrr}1&-1&0&0\\0&-1&0&0\\0&0&1&-1\\0&0&0&-1\end{array}\right], \left[\begin{array}{rrrr}1&-1&0&0\\0&0&1&-1\\0&-1&0&0\\0&0&0&-1\end{array}\right],\left[\begin{array}{rrrr}1&-1&0&0\\1&0&0&0\\0&0&1&-1\\0&0&1&0\end{array}\right],\left[\begin{array}{rrrr}0&0&0&1\\0&0&-1&1\\0&-1&0&1\\1&-1&-1&1\end{array}\right]$$
We claim that the lattice determined by this group is 
the proper intermediate lattice $L$ between $\Z A_2\oplus \Z A_2$
and $\Lambda(A_2)\oplus\Lambda(A_2)$ and the group action is that 
induced by the index 2 subgroup of the automorphism group of the 
root system $A_2\times A_2$ given by 
$(W(A_2)\times W(A_2))\rtimes C_2)\times C_2$. 
Let $\omega_1,\omega_2$ be the fundamental dominant weights of $A_2$
with respect to a basis $\alpha_1,\alpha_2$ of the $A_2$ root system.
We will write the basis of $\Lambda(A_2)\oplus \Lambda(A_2)$ 
as $\{\omega_1,\omega_2,\omega_1',\omega_2'\}$
and that of $\Z A_2\oplus \Z A_2$ as
$\{\alpha_1,\alpha_2,\alpha_1',\alpha_2'\}$.
Then the claimed sublattice $L$ of $\Lambda(A_2)\oplus \Lambda(A_2)$ 
satisfies $L=\langle \omega_1+\omega_1',\Z A_2\oplus \Z A_2\rangle$
Let $s_1=s_{\alpha_1},s_1'=s_{\alpha_1'},s_2=s_{\alpha_2},s_2'=s_{\alpha_2'}$
be the generators of $W(A_2)\times W(A_2)$.  Let $\tau$ be the element of 
order 2 which swaps the 2 copies of $\Lambda(A_2)$.
Then our group is $(\langle s_1,s_2,s_1',s_2'\rangle \rtimes \tau)\times \langle -\id \rangle$.
We claim that 
$$\beta=\{\omega_1+\omega_1',s_2s_1(\omega_1+\omega_1'),\alpha_1',s_2's_1'(\alpha_1')\}$$
is a $\Z$-basis of our lattice $L$.
Note 
$$s_i(\omega_j)=\omega_j-\delta_{ij}\alpha_j,i,j=1,2$$
We also recall that $\alpha_1=2\omega_1-\omega_2$ and $\alpha_2=-\omega_1+2\omega_2$.  The same results hold for the prime copies.
$\langle s_1,s_2\rangle$ acts trivially on $\oplus_{i=1}^2\omega_i'$ 
and similarly for the prime copies. 
So $s_2s_1(\omega_1+\omega_1')=\omega_1-\alpha_1-\alpha_2+\omega_1'=-\omega_2+\omega_1'$
and $s_2's_1'(\alpha_1')=-\alpha_1'-\alpha_2'$.
Note that $s_2s_1(\omega_1)=\omega_1+s_2s_1(\alpha_1)$.
Then 
$s_1(\omega_1+\omega_1')=\omega_1-\alpha_1+\omega_1'=-\omega_1+\omega_2+\omega_1'=-(\omega_1+\omega_1')-(-\omega_2+\omega_1')+3\omega_1'=
-(\omega_1+\omega_1')-(-\omega_2+\omega_1')+\alpha_1'-(-\alpha_1'-\alpha_2')
=-(\omega_1+\omega_1')-s_2s_1(\omega_1+\omega_1')+\alpha_1'-(s_2's_1'(\alpha_1')$.
From these calculations, one can show that $\beta$ is a basis of $L$.
They also allow us to find the matrices of the generators
$s_1,s_2s_1,s_1',s_2's_1',\tau,-\id$ on the basis $\beta$.
[We omit the details but note that the  worst calculation is  $s_1(\omega_1+\omega_1')$ made above.]
We obtain:



$$s_1=\left[\begin{array}{rrrr}-1&-1&1&0\\0&1&0&0\\0&0&1&0\\0&0&1&0\end{array}\right], s_1'=\left[\begin{array}{rrrr}1&0&-1&0\\0&1&-1&0\\0&0&-1&0\\0&0&1&1\end{array}\right],s_2s_1=\left[\begin{array}{rrrr}0&1&0&0\\-1&-1&1&-1\\0&0&1&0\\0&0&0&1\end{array}\right]$$
$$s_2's_1'=\left[\begin{array}{rrrr}1&0&0&1\\0&1&0&1\\0&0&0&1\\0&0&-1&-1\end{array}\right], \tau=\left[\begin{array}{rrrr}1&0&0&0\\1&0&0&1\\2&1&-1&1\\-1&1&0&0\end{array}\right], -I_4$$ 
The group generated by these generators is conjugate to the
group DadeGroup(4,4). 
This is determined by GAP by checking that the 
CrystCatZClass of the two matrix groups given by generators are in fact the 
same.  See the remark below on GAP calculations. 
[In fact, the index 2 subgroup generated by the generators
$\langle s_1,s_2s_1,s_1',s_2's_1',-\id\rangle$
is precisely equal to the group with GAP ID [4,22,11,1].
which is conjugate to a subgroup of the group DadeGroup(4,4).]
So the lattice 
determined by DadeGroup(4,4) is indeed the intermediate lattice
$L$ between $\Z A_2\oplus \Z A_2$ and $\Lambda(A_2)\oplus \Lambda(A_2)$
as a $(W(A_2)\times W(A_2))\rtimes C_2 \times C_2$ lattice.

For G=DadeGroup(4,5) with GAP ID [4,30,13,1], we claim that the 
corresponding lattice is $(\Z A_2\oplus \Z A_2,\Aut(A_2\times A_2))$
which is hereditarily rational as a wreath product of 2 hereditarily 
rational lattices.
$\Aut(A_2\times A_2)=(\Aut(A_2)\times \Aut(A_2))\rtimes S_2$ acts naturally 
on the root lattice for $A_2\times A_2$ where $\Aut(A_2)\times \Aut(A_2)$
acts diagonally on $\Z A_2\oplus A_2$ and the subgroup $S_2$ permutes the 
two copies of $A_2$.
The generators given by GAP are 
$$\left[\begin{array}{rr}I_2&0\\0&P\end{array}\right],
\left[\begin{array}{rr}0&I_2\\P&0\end{array}\right],
\left[\begin{array}{rr}I_2&0\\0&Y\end{array}\right],
\left[\begin{array}{rr}Z&0\\0&I_2\end{array}\right]$$
where 
$$P=\left[\begin{array}{rr}0&1\\1&0\end{array}\right],
Y=\left[\begin{array}{rr}0&1\\-1&1\end{array}\right],
Z=\left[\begin{array}{rr}1&-1\\1&0\end{array}\right]$$
Since  $Y,Z$ have order 6 and $Y^3=Z^3=-I_2$,
we may replace these generators by 
$$\left[\begin{array}{rr}I_2&0\\0&P\end{array}\right],
\left[\begin{array}{rr}0&I_2\\I_2&0\end{array}\right],
\left[\begin{array}{rr}I_2&0\\0&Y^2\end{array}\right],
\left[\begin{array}{rr}Z^2&0\\0&I_2\end{array}\right],
X_5=\diag(I_2,-I_2), X_6=\diag(-I_2,I_2).$$

Since
$$Y^2=\left[\begin{array}{rr}-1&1\\-1&0\end{array}\right]=(\rho_2^-)^*((123))\qquad Z^2=\left[\begin{array}{rr}0&-1\\1&-1\end{array}\right]=(\rho_2^-)^*((132)), P=(\rho_2)^*((12))$$
We see that the  lattice defined on $\Z\e_1\oplus \Z\e_2$ defined by $\langle P,Y^2\rangle$ is the $S_3$-lattice $I_{X_3}$.
We also see that the  lattice defined on $\Z\e_3\oplus \Z\e_4$
defined by $\langle P, Z^2\rangle$ is $I_{X_3}$.
This shows that the lattices defined by both
$\langle P,Y^2,-I_2\rangle$ and $\langle P,Z^2,-I_2\rangle$ are 
isomorphic to  the lattice $(I_{X_3}\otimes \Z^-_{S_3},S_3\times C_2)$
or equivalently $(\Z A_2,\Aut(A_2))$.
So the lattice restricted to 
$$\langle \diag(I_2,P),\diag(I_2,Y^2),\diag(I_2,-I_2),\diag(P,I_2),\diag(Z^2,I_2),\diag(-I_2,I_2)\rangle$$
is $(\Z A_2\oplus \Z A_2,\Aut(A_2)\times \Aut(A_2))$.
  Since 
$$\left[\begin{array}{rr}0&I_2\\I_2&0\end{array}\right]$$
swaps the 2 copies of $\Z A_2$, we see that the full lattice structure is 
given by 
$$(\Z A_2\oplus \Z A_2,(\Aut(A_2)\times \Aut(A_2))\rtimes C_2)$$
as required.

DadeGroup(4,6) has GAP ID [4,31,7,1].
We will show that it determines the lattice $(\Z A_4, \Aut(A_4))$
which is hereditarily rational since it is the tensor product of
2 augmentation ideals of relatively prime ranks.
This lattice is $(I_{X_5}\otimes \Z^-_{S_5},S_5\times C_2)$. 
Recalling our representation $((\rho^-_4)^*:S_5\times C_2\to \GL_4(\Z)$
determined by $(I_{X_5}\otimes \Z^-_{S_5})$, 
we see that the generators of DadeGroup(4,6) given by GAP
are 
$$(\rho^-_4)^*(((15)(234),\gamma),(\rho^-_4)^*((14532)),\gamma), (\rho_4)^*((1423),\gamma)$$
Since $(14532)$ has odd order, we see that $((14523),\gamma)^5=(\id,\gamma)$,
and we know that $S_5$ is generated by any 5 cycle and any transposition,
so it suffices to note that $[(15)(234)]^3=(15)$, in order to conclude that 
the preimages generate $S_5\times C_2$.
So, as required,  DadeGroup(4,6) determines  the 
lattice $(I_{X_5}\otimes \Z^-_{S_5},S_5\times C_2)=(\Z A_4,\Aut(A_4))$. 

DadeGroup(4,7) has GAP ID [4,31,7,2].
We will show that it determines the lattice $(\Lambda(A_4), \Aut(A_4))$.
This lattice is $(J_{X_5}\otimes \Z^-_{S_5},S_5\times C_2)$. 
In terms of our representation $\rho^-_4:S_5\times C_2\to \GL_4(\Z)$
determined by $(J_{X_5}\otimes \Z^-_{S_5})$, 
we see that the generators of DadeGroup(4,7) given by GAP
are 
$$\rho^-_4((132)(45),\gamma),\rho^-_4((15234),\gamma), \rho^-_4((1324),\gamma)$$
Since $(15234)$ has odd order, we see that $((15234),\gamma)^5=(\id,\gamma)$,
and we know that $S_5$ is generated by any 5 cycle and any transposition,
so it suffices to note that $[(132)(45]^3=(45)$, in order to conclude that 
the preimages generate $S_5\times C_2$.
So, as required,  DadeGroup(4,7) determines  the  lattice  
$(J_{X_5}\otimes \Z^-_{S_5},S_5\times C_2)=(\Lambda(A_4),\Aut(A_4))$.

DadeGroup(4,8) has GAP ID [4,32,21,1].  
We claim that it determines the lattice 
$(\Z B_4,W(B_4))$.
In terms of our representation $\eta_4:W(B_4)\to \GL_4(\Z)$, the generators
given by GAP are
$$\eta_4(\tau_2\tau_4(24)), \eta_4(\tau_1\tau_3\tau_4(234)), 
\eta_4(\tau_1\tau_2),
\eta_4(\tau_3\tau_4(12)(34)),\eta_4(\tau_1\tau_2(13)(24)),
\eta_4(\tau_1\tau_4(14)(23))$$
We need only check that the subgroup $H$ of $W(B_4)$ generated by the 
preimages of the generators under $\eta_4$ is $W(B_4)$.
(Note that GAP gives a structure description for the group as the wreath 
product of $C_2$ by $S_4$ so this is just a check).
Recall that $(\tau\sigma)^2=\tau\tau^{\sigma}\sigma^2$ and $\tau_i^{\tau\sigma}=\tau_{\sigma(i)}$ for any $\tau\in C_2^4$ and $\sigma\in S_4$.
Noting that $(\tau_1\tau_2)$ is in $H$ and conjugating this element by 
each of the other above generators, one can see that
$\tau_i\tau_j$ is in the group for all $1\le i\ne j\le 4$.
Looking again at the generators, we see that $(24),(12)(34),(13)(24),(14)(23)$
are also in the group.
But then since also $[\tau_1\tau_3\tau_4(234)]^2=\tau_2\tau_3(243)$
is in $H$, we have that $(243)$ is in $H$ too.
We now have $\langle (24),(12)(34),(13)(24),(14)(23),(243)\rangle=S_4$
and $\langle \tau_i\tau_j: 1\le i\ne j\le 4\rangle$ as subgroups.
But looking at the original generator $\tau_1\tau_3\tau_4(234)$, we see that
$\tau_1$ and hence all its conjugates under $S_4$ are in the group.
This shows that the group generated by its preimages contains
$\langle \tau_i:i=1,\dots,4\rangle$ and $S_4$ and then must be $W(B_4)$. 
Note that $(\Z B_4,W(B_4))$ is a hereditarily rational lattice as it is 
sign permutation.

DadeGroup(4,9) has GAP ID [4,33,16,1].
Its generators are
$$X_1=\left[\begin{array}{rrrr}1&0&0&0\\0&0&1&0\\0&1&0&0\\0&0&0&1\end{array}\right], X_2=\left[\begin{array}{rrrr}1&0&0&0\\0&0&1&0\\0&1&0&0\\1&1&1&-1\end{array}\right],
X_3=\left[\begin{array}{rrrr}0&1&0&-1\\0&0&1&-1\\1&0&0&-1\\1&1&1&-2\end{array}\right],$$
$$X_4=\left[\begin{array}{rrrr}-1&0&0&1\\0&0&-1&1\\-1&-1&-1&1\\-1&-1&-1&2\end{array}\right], X_5=\left[\begin{array}{rrrr}0&0&-1&0\\1&1&1&-2\\-1&0&0&0\\0&0&0&-1\end{array}\right],$$
$$X_6=\left[\begin{array}{rrrr}-1&-1&-1&2\\0&0&1&0\\0&-1&0&0\\-1&-1&0&1\end{array}\right],X_7=\left[\begin{array}{rrrr}1&0&0&0\\0&-1&0&0\\0&0&-1&0\\0&-1&-1&1\end{array}\right]$$

We wish to show that the associated lattice is $(\Z F_4,W(F_4))$.
We recall a standard basis of the root system $F_4$ is given by 
$$\Delta=\{\alpha_1=\e_2-\e_3,\alpha_2=\e_3-\e_4,\alpha_3=\e_4,\alpha_4=\frac{\e_1-\e_2-\e_3-\e_4}{2}\}$$
The roots in the root system $F_4$ are of the 
form 
$$\pm \e_i\pm \e_j,1\le i<j\le 4; \pm \e_i, 1\le i\le 4; \frac{\pm \e_1\pm \e_2\pm \e_3\pm \e_4}{2}$$
A $\Z$-basis for $\Z F_4$ is given by 
$$\beta=\{\e_1,\e_2,\e_3,\frac{\e_1+\e_2+\e_3+\e_4}{2}\}$$
For each of the simple roots in $\Delta$ 
we may compute the matrices of $s_{\alpha_i}$ with respect to the basis $\beta$
where 
$$s_{\alpha}(x)=x-2\frac{x\cdot \alpha}{\alpha\cdot \alpha}\alpha$$
is the simple reflection corresponding to $\alpha$.
We find that
$$s_{\alpha_1}=\left[\begin{array}{rrrr}1&0&0&0\\0&0&1&0\\0&1&0&0\\0&0&0&1\end{array}\right], s_{\alpha_2}=\left[\begin{array}{rrrr}1&0&0&0\\0&1&0&0\\-1&-1&-1&2\\
0&0&0&1\end{array}\right]$$
$$s_{\alpha_3}=\left[\begin{array}{rrrr}1&0&0&0\\0&1&0&0\\0&0&1&0\\1&1&1&-1\end{array}\right], s_{\alpha_4}=\left[\begin{array}{rrrr}0&0&0&1\\1&1&0&-1\\1&0&1&-1\\
1&0&0&0\end{array}\right]$$
We may show that  $s_{\alpha_1},\dots,s_{\alpha_4}$ are all contained in DadeGroup(4,9).  That is, we may express them as products of the generators given by GAP. Explicitly,  $s_{\alpha_1}=X_1$, $s_{\alpha_2}=X_5X_3X_4^{-1}X_1$,
$s_{\alpha_3}=X_2X_1$, $s_{\alpha_4}=X_6^{-1}X_4^{-1}X_3^{-1}X_2X_1$.
Since DadeGroup(4,9) has order $1152=|W(F_4)|$ we see that they coincide.
\end{proof} 

\begin{remark}
(Stably Rational Tori of Dimension 4 (Hoshi Yamasaki))
The 487 finite subgroups of $\GL(4,\Z)$  which correspond to stably rational tori from Hoshi and Yamasaki's list
are conjugate to a subgroup of one of the groups with  
GAP IDs on the following list:
\begin{itemize}
\item (Irreducible maximal finite subgroups)
[4,30,13,1],[4,31,7,1],[4,32,21,1]
\item (Irreducible non-maximal finite subgroup)
[4,31,6,2]
\item (Decomposable finite subgroups) 
[4,20,22,1], [4,25,9,2],[4,25,11,1]
\item (Reducible finite subgroups)
[4,24,3,4],[4,25,7,5],[4,25,8,5]
\end{itemize}
That is, these are the maximal conjugacy classes of subgroups corresponding to 
stably rational tori.
In the list above, the maximal indecomposable subgroups corresponding to stably rational tori are listed in their
paper in the proof of the result for dimension 4.  The maximal decomposable subgroups are not listed in their paper but could be 
easily derived from their results or directly from GAP.

The 7 finite subgroups of $\GL(4,\Z)$ which correspond to retract rational
tori which are not stably rational from Hoshi and Yamasaki's list are
\begin{itemize}
\item 
(6 Subgroups of DadeGroup(4,7) of GAP ID [4,31,7,2]):

[4,31,1,3],[4,31,1,4],[4,31,2,2],[4,31,3,2],[4,31,5,2],[4,31,7,2].
\item
(1 exceptional subgroup)
[4,33,2,1].
\end{itemize}
\end{remark}

\begin{theorem}\label{th:herrat}
(Hereditarily rational tori of dimension 4)
Among the 10 finite conjugacy classes of subgroups of $\GL(4,\Z)$ which are maximal among those corresponding to  stably rational tori of dimension 4, the following 8 correspond to 
hereditarily rational algebraic tori.  We list the GAP ID of each group $G$ together with the 
lattice structure of $M_G$.  
\begin{enumerate}
\item Dade Groups which correspond to hereditarily rational tori:
\begin{itemize}
\item \textup{[4,20,22,1]:} $(\Z B_2\oplus \Z A_2,\Aut(B_2)\times \Aut(A_2))$ 
\item \textup{[4,30,13,1]:} $(\Z A_2\oplus \Z A_2,(\Aut(A_2)\times \Aut(A_2))\rtimes C_2)$.
\item \textup{[4,31,7,1]:} $(\Z A_4,\Aut(A_4))$.
\item \textup{[4,32,21,1]:} $(\Z B_4,W(B_4))$.
\end{itemize}
\item Direct Products $G=H\times C_2$ of a finite matrix group H of rank 3 and
a finite matrix group $C_2$ of rank 1 where $H$ corresponds to a maximal hereditarily rational torus of rank 3.
In this case the lattice is $M_G=\inf_H^G(M_H)\oplus \Z^-_H$
\begin{itemize}
\item \textup{[4,25,9,2]:} $(\inf_{S_4}^{S_4\times C_2}\Z A_3 \oplus \Z^-_{S_3},W(A_3)\times C_2)$.
\item \textup{[4,13,6,4]:} $(\inf_{D_8}^{D_8\times C_2}L\oplus \Z^-_{D_8},D_8\times C_2)$.
\end{itemize}
\item Groups whose corresponding lattice is indecomposable but reducible with a rank 3 invariant maximal hereditarily rational sublattice and a fixed rank 1 quotient lattice.
\begin{itemize}
\item \textup{[4,25,7,5]:} $(L_1,W(B_3))$ where $L_1$ is a non-split extension of $W(B_3)$-lattices
$$0\to \Z B_3\to L_1\to \Z\to 0$$
\item \textup{[4,24,3,4]:} $(L_2,W(A_3))$ where $L_2$ is a non-split extension of 
$W(A_3)$-lattices
$$0\to \Z A_3\to L_2\to \Z\to 0$$ 
\end{itemize}
\end{enumerate}
The union of the conjugacy classes of these 8 finite subgroups of $\GL(4,\Z)$ 
produces 477 of the 487 stably rational algebraic $k$-tori 
determined by Hoshi and Yamasaki.
\end{theorem}

\begin{proof}
We have already determined  the Dade groups which correspond to hereditarily
rational tori.

We first determine the GAP IDs of the groups corresponding to a direct
sum of a maximal hereditarily rational lattice of rank 3 and a sign lattice.

We claim that the group with GAP ID [4,25,9,2] has lattice given by 
$$(\Z^-_{W(A_3)}\oplus \inf^{W(A_3)\times C_2}_{W(A_3)}\Z A_3,W(A_3)\times C_2)$$
The generators given by GAP are
$$\diag(1,A_1),\diag(-1,A_2),\diag(1,A_3),\diag(1,A_4)$$
where $A_i\in \GL_3(\Z)$ are given 
by
$$A_1=(\rho_3)^*((34)), A_2=(\rho_3)^*((124)), A_3=(\rho_3)^*((14)(23)),
A_4=(\rho_3)^*((12)(34))$$
where $(\rho_3)^*:S_4\to \GL_3(\Z)$ is the representation associated to the root lattice $(I_{X_4},S_4)=(\Z A_3,W(A_3))$ described earlier.
Since $(124)$ is odd, so is $A_2$ and so the generator $(-1,A_2)$
can be replaced by $(1,A_2)$ and $(-1,I_3)$.
Then it is clear that the group is a direct product of $\langle (-1,I_3)\rangle$ and $\langle (1,A_i):i=1,\dots,4\rangle$.
It suffices to show that $\langle (34),(124),(14)(23),(12)(34)\rangle =S_4$,
which is easily checked.
So $\langle (1,A_i\rangle=(\rho_3)^*(S_4)$
which implies that the corresponding lattice is the $C_2\times S_4$-lattice inflated from the $S_4$-lattice $I_{X_4}$ as required.


The group with GAP ID [4,13,6,4] has  
 generators 
$(I_3,-1),(A,-1),(B,-1)$
where 
$$A=\left[\begin{array}{rrr}1&0&0\\1&0&-1\\1&-1&0\end{array}\right],
B=\left[\begin{array}{rrr}0&0&1\\-1&0&1\\0&-1&1\end{array}\right]$$
Since we may replace the generators by $(I_3,-1),(A,1),(B,1)$,
the group is a direct product  $\langle (I_3,-1)\rangle\times \langle (A,1),(B,1)\rangle$ and so we may just determine the group $\langle A,B\rangle$.
We note that $\Z(\e_1-\e_3)\oplus \Z\e_2$ is stable under $\langle A,B\rangle$.
Computing the matrices with respect to the new basis
$\{\e_1-\e_3,\e_2,\e_3\}$ (or equivalently conjugating by the change of basis matrix)
we obtain
$$A'=\left[\begin{array}{rrr}0&1&0\\1&0&0\\1&-1&1\end{array}\right],
B'=\left[\begin{array}{rrr}0&1&0\\-1&0&0\\0&-1&1\end{array}\right]$$
It is then clear that the $W(B_2)$ lattice determined by $\langle A',B'\rangle$ 
satisfies 
$$0\to \Z B_2\to L\to \Z\to 0$$
Then our lattice is $(\inf_{W(B_2)}^{W(B_2)\times C_2}L\oplus \Z^-_{W(B_2)},W(B_2)\times C_2)$.

 
We next look at the groups which correspond to a reducible
lattice with a 3 dimensional invariant sublattice:

For the group with GAP ID [4,25,7,5], the generators are
$$C_1=\left[\begin{array}{rrrr}1&0&1&1\\0&1&0&0\\0&0&0&-1\\0&0&-1&0\end{array}\right],
C_2=\left[\begin{array}{rrrr}1&0&1&0\\0&0&1&0\\0&0&0&1\\0&-1&0&0\end{array}\right],$$
$$C_3=\left[\begin{array}{rrrr}1&-1&1&0\\0&-1&0&0\\0&0&-1&0\\0&0&0&1\end{array}\right],
C_4=\left[\begin{array}{rrrr}1&0&1&1\\0&1&0&0\\0&0&-1&0\\0&0&0&-1\end{array}\right]$$
Note that this determines a lattice $M$ which contains a sublattice
with basis $\{\e_2,\e_3,\e_4\}$ which is stable under the action of the 
group.  Note also that $M/(\oplus_{i=2}^4\Z\e_i)\cong \Z$.
The action on $\oplus_{i=2}^4\Z\e_i$ is determined by
the group generated by 
$$\eta_3(\tau_2\tau_3(23)),\eta_3(\tau_1(123)),
\eta_3(\tau_1\tau_2),  \eta_3(\tau_2\tau_3)$$
Since $(\tau_1(123))^2=\tau_1\tau_2(132)$, $(132)$ is in the preimage,
and then so is $\tau_1$. Then one can easily
show that $\tau_2,\tau_3$ and $(23)$ are in the preimage too.  
This shows that the group determines a $W(B_3)$-lattice $M$
which satisfies
$$0\to \Z B_3\to M\to \Z\to 0$$
So the corresponding algebraic torus is hereditarily rational.

For the group with GAP ID [4,24,3,4], the generators are
$$X_1=\left[\begin{array}{rrrr}1&-1&0&0\\0&-1&0&0\\0&-1&1&0\\0&-1&0&1\end{array}\right],
X_2=\left[\begin{array}{rrrr}1&-1&0&1\\0&-1&0&1\\0&-1&1&0\\0&-1&0&0\end{array}\right],$$
$$X_3=\left[\begin{array}{rrrr}1&-1&-1&1\\0&0&-1&1\\0&0&-1&0\\0&1&-1&0\end{array}\right],
X_4=\left[\begin{array}{rrrr}1&-1&-1&1\\0&-1&0&0\\0&-1&0&1\\0&-1&1&0\end{array}\right]$$
Note that this determines a lattice $M_4$ which contains a sublattice
$M_3=\oplus_{i=2}^4\Z\e_i$ which is stable under the action of the 
group.  Note also that $M_4/M_3\cong \Z$.
The action restricted to the sublattice  $M_3$ is determined by
the group generated by 
$$\rho_3^*((14)),\rho_3^*((134)),\rho_3^*((13)(24)),\rho_3^*((14)(23))$$
for the representation $\rho_3^*:S_4\to \GL(3,\Z)$ 
associated to the $S_4$-lattice $I_{X_4}$.  It is easily checked that the restriction of the group action on $M_3$ is a faithful action.
Since the group generated by preimages is $\langle (14),(134),(13)(24),(14)(23)\rangle$, we see that it contains the normal subgroup $\langle (13)(24),(14)(23)\rangle$ of $S_4$. Then clearly $\langle (13)(24),(14)(23),(143)\rangle =A_4$
and $\langle (13)(24),(14)(23),(143),(14)\rangle =S_4$.
So the lattice determined by this group satisfies a short exact sequence of 
$W(A_3)=S_4$ lattices given by  
$$0\to \Z A_3\to M_3\to \Z\to 0$$
This again shows that the associated group is hereditarily rational.

We then use GAP to take the union of the conjugacy classes of subgroups
corresponding to these 10 hereditarily rational lattices.  (See below for our simple use of GAP to obtain this information.)
We find that we obtain 477 hereditarily rational tori, all but 10 of the 
stably rational tori obtained by
Hoshi and Yamasaki.  
Of these 10, there are 2 maximal groups having GAP IDs [4,25,8,5] and [4,31,6,2].  In the next proposition we will describe the lattice structure of these
two groups and list the 10 exceptional subgroups.
In a subsequent section, we will give non-computational
proofs that the tori corresponding to these two groups are stably rational.
\end{proof}

\begin{remark}
Let $H_i$, $i=1,2$ be finite matrix groups  of rank $r_i$, $i=1,2$  where each is a maximal subgroup corresponding to a  hereditarily rational torus of the appropriate rank.  Then $H_1\times H_2$ is a finite matrix group corresponding to  a hereditarily rational torus of rank $r+s$ but it may not be maximal among the finite matrix groups of rank $r+s$ whose corresponding algebraic torus is stably rational.
\end{remark}

\begin{prop}
The following are the GAP IDs and lattices 
corresponding to the two finite subgroups of $\GL(4,\Z)$ which are maximal among those corresponding to stably rational tori of dimension 4 but are not known to be hereditarily rational:
\begin{itemize}
\item \textup{[4,31,6,2]:} $(J_{X_5}\otimes \Z_{A_5}^-,A_5\times C_2)$
\item \textup{[4,25,8,5]:} The corresponding lattice $L$ is a non-split extension
of $W(B_3)$-lattices 
$$0\to \Z B_3\to L\to \Z^-_{C_2^3\rtimes A_3}\to 0$$
\end{itemize}
There are 10 conjugacy classes of subgroups of these 2 groups which are not
conjugate to a subgroup of one of the 8 groups from Theorem \ref{th:herrat}.
The rationality of the corresponding algebraic tori is hence unknown.
The full list is
\begin{itemize}
\item Subgroups of \textup{[4,31,6,2]:} \textup{[4,31,3,2], [4,31,6,2]}
\item Subgroups of \textup{[4,25,8,5]:} \textup{[4,6,2,11], [4,12,4,13], [4,13,2,6], [4,13,3,6], [4,13,7,12], [4,24,4,6], [4,25,4,5], [4,25,8,5]}.
\end{itemize}
\end{prop}
Note that the lattice corresponding to \textup{[4,31,3,2]} 
is $(J_{X_5},A_5)$ which corresponds to a norm one torus.
\begin{proof}

From the proof of the last theorem, we see that the tori corresponding to all
of the maximal groups corresponding to stably rational tori except for those with  GAP IDs [4,25,8,5] and [4,31,6,2] are hereditarily rational.

For the group with GAP ID [4,25,8,5], the generators are
$$D_1=\left[\begin{array}{rrrr}-1&0&0&0\\0&-1&0&0\\0&0&0&-1\\0&0&-1&0\end{array}\right],
C_2=\left[\begin{array}{rrrr}1&0&1&0\\0&0&1&0\\0&0&0&1\\0&-1&0&0\end{array}\right],$$
$$
C_3=\left[\begin{array}{rrrr}1&-1&1&0\\0&-1&0&0\\0&0&-1&0\\0&0&0&1\end{array}\right],
C_4=\left[\begin{array}{rrrr}1&0&1&1\\0&1&0&0\\0&0&-1&0\\0&0&0&-1\end{array}\right]$$
Note that these are in fact the same generators as in the [4,25,7,5]
case except for the first one.

Note that this determines a lattice $M$ which contains a sublattice $M_0$
with basis $\{\e_2,\e_3,\e_4\}$ which is stable under the action of the 
group.  Note also that $M_1=M/M_0$ is a rank 1 lattice with 
non-trivial action.

The action on $M_0=\oplus_{i=2}^4\Z\e_i$ is determined by
the group generated by 
$$\eta_3(\tau_1\tau_2\tau_3(23)),\eta_3(\tau_1(123)),
\eta_3(\tau_1\tau_2),  \eta_3(\tau_2\tau_3)$$
Since $(\tau_1(123))^2=\tau_1\tau_2(132)$, then $(132)$ is in the preimage
and hence so is $\tau_1$.  
Then one can easily
show that $\tau_2,\tau_3$ and $(23)$ are too.  Then it is clear that the group
acts on the lattice $M_0$ as $(\Z B_3,W(B_3))$.
It is easy to check that the restriction of the group action to $M_0$
is faithful, so that we may identify the group elements with the elements of 
$W(B_3)$. 
Note that  $N=\langle \tau_1(123),\tau_1\tau_2,\tau_2\tau_3\rangle$ acts trivially on $M_1=M/M_0$.  Since $N$ contains $(\tau_1(123))^2= \tau_1,\tau_2,\tau_3$, it can be shown to contain $\langle \tau_1,\tau_2,\tau_3,(123)\rangle=C_2^3\rtimes A_3$.   
This shows that the group $G$ determines a lattice $M$
which satisfies
$$0\to \Z B_3\to M\to \Z^-_{C_2^3\rtimes A_3}\to 0$$
Note also that  $W(B_3)=C_2^3\rtimes S_3\cong C_2\times S_4$ and $N=C_2^3\rtimes A_3\cong C_2\times A_4$.
We will present a non-computational proof that the lattice corresponding to 
[4,25,8,5] is quasi-permutation.

The lattice determined by [4,31,6,2] is $
(J_{X_5}\otimes \Z_{A_5}^-,A_5\times C_2)$.
This is because the generators given by GAP are
$$\rho_4^-((15243),\gamma),\rho_4^-((132),1),\rho_4^-((12)(34),1),\rho_4^-((13)(24),1)$$
Since $(15243)$ has odd order, we see that $(1,\gamma)\in C_2$ is in the preimage.
Then it is not hard to see that $\langle (13)(24),(12)(34),(132)\rangle=A_4$
and so $\langle (15243),(132),(12)(34),(13)(24)\rangle=A_5$.

The lattice determined by [4,31,3,2] is $(J_{X_5},A_5)$.
This is because the generators given by GAP are
$$\rho_4((14352)),\rho_4((123)),\rho_4((13)(25)),\rho_4((12)(35))$$
Again, it's easy to see that $\langle (13)(25),(12)(35)\rangle=C_2\times C_2$
and $\langle (13)(25),(12)(35),(123)\rangle=A_4$
and finally $\langle (13)(25),(12)(35),(123)\rangle=A_5$.
  We will show in the next section that the
lattices $(J_{A_5\times C_2/A_4\times C_2}\otimes \Z_{A_5}^-,A_5\times C_2)$
and $(J_{A_5/A_4},A_5)$ are quasi-permutation so that the corresponding tori are stably
rational. We are not able to show that these tori are rational.
Note that that determined by the $A_5$-lattice $J_{A_5/A_4}$ is a norm one torus and we
will see that the other torus is closely related.

To explain why there are only 2 missing subgroups of [4,31,6,2] not known to correspond to stably rational tori, note that the
restriction of $(J_{X_5}\otimes \Z_{A_5}^-,A_5\times C_2)$ to the maximal subgroup $D_5\times C_2$
is $(J_{X_5}\otimes \Z_{D_{10}}^-,D_{10}\times C_2)$
which corresponds to a  hereditarily rational torus.
Note that all subgroups of $A_5\times C_2$ except $A_5$ are subgroups of $D_{10}\times C_2$.
\end{proof}

\begin{remark}
We explain our very basic use of GAP.
We mainly used the generating sets (which could have been found in~\cite{BBNWZ78}) and as a calculation tool to check our hypotheses.
All the calculations of the lattices corresponding to the groups  could be done by hand as explained above directly from the 
generating sets, with only 2 exceptions.  In the case of the lattice for the  Weyl group of $F_4$, we used GAP to find the simple reflections in the generators. In the case of DadeGroup(4,4) we used GAP to check that our proposed group
was conjugate to DadeGroup(4,4).  We hope to find simpler proofs in those 2 cases.  Note that they do not come into play in checking for rational tori.

To check in the dimension 3 and 4 cases that the conjugacy classes of subgroups of the groups corresponding to our hereditarily rational algebraic tori are
rational give all (respectively all but 10) stably rational algebraic tori
we mainly use the following function:

\begin{verbatim}
SubConjClass:=function(r,m,n,k)
 local g,sub,l,setsub;
 g:=MatGroupZClass(r,m,n,k);
 sub:=Subgroups(g);
 l:=List(sub,x->CrystCatZClass(x));
 setsub:=Set(l);
return setsub;
end;
\end{verbatim}

This function returns the conjugacy classes of subgroups of the 
group MatGroupZClass(r,m,n,k) given as a list of GAP IDs.
It depends on the GAP script written by Hoshi and Yamasaki
crystcat.gap.  This script which is available on the 
second author's website,
at \begin{verbatim} http://www.math.h.kyoto-u.ac.jp/~yamasaki/Algorithm/\end{verbatim}
 determines the GAP ID of a finite
subgroup $G$  of $\GL_n(\Z)$ where $n=2,3,4$ by the function 
\begin{verbatim}CrystCatZClass(G).\end{verbatim}  It uses the data of the 
book~\cite{BBNWZ78} to find the crystal class, Q class and Z class of a finite
subgroup of $\GL_n(\Z)$.
It is invoked using 
\begin{verbatim}
Read(``crystcat.gap''); 
\end{verbatim}
With that tool, for the proposed maximal hereditarily rational subgroups, one can find the 
union of all the conjugacy classes of subgroups in terms of their GAP IDs.
One can also find the union of all conjugacy classes of subgroups of the
Dade Groups.  By taking the difference of these two sets, we find the 
list of all GAP IDs correponding to non-rational tori.  One can then
 check them against the lists in Hoshi and Yamasaki~\cite{HY12}  which I do not
reproduce here. 
 \end{remark}

\section{Stable rationality of Exceptional tori}
In this section, we will show that algebraic $k$-tori whose
character lattices are given by $(J_{X_5},A_5)$ 
or $(J_{X_5}\otimes \Z^-_{A_5},A_5\times C_2)$ are stably rational
recovering results of Hoshi and Yamasaki in a non-computational way.
We will also give new non-computational proofs showing that the 7 algebraic
$k$-tori of dimension 4 which are retract but not stably rational.
 
Note that, for a prime $p$, Beneish~\cite{Ben98} proved that the $S_p$-lattice $J_{X_p}$ is flasque equivalent to 
$$\Ind^{S_p}_{N_{S_p}(C_p)}\Res^{S_p}_{N_{S_p}(C_p)}J_{X_p}$$
as $S_p$-lattices where $N_{S_p}(C_p)=C_p\rtimes C_{p-1}$ is the normaliser
of the cyclic $p$ Sylow subgroup of $S_p$.
[Recall however, that it is known that the $N_{S_p}(C_p)$-lattice $J_{X_p}$ is not
$C_p\rtimes C_{p-1}$-quasi permutation for primes  $p\ge 5$.] 
We intend to prove a similar result for $A_5$ and $N_{A_5}(C_5)=C_5\rtimes C_2=D_{10}$.  The arguments are similar at the start but diverge at a critical point.
This result and the fact that $J_{X_5}$ is $D_{10}$-quasi-permutation,
will allow us to show that $J_{X_5}$ is also $A_5$-quasi-permutation.
Note that this is equivalent to  the result that for a separable extension $K/k$ of degree 5
with Galois closure $L/k$ such that $\Gal(L/k)=A_5$ and $\Gal(L/K)=A_4$,
the norm one torus $R^{(1)}_{K/k}(\Gm)$ is stably rational.

We intend also to show that $J_{X_5}\otimes \Z^-_{A_5}$ is $A_5\times C_2$-quasi-permutation (and the corresponding torus stably rational)
using the result for the corresponding norm one torus and a useful Lemma
due to  Florence (see below).

The following lemma was observed by Bessenrodt-Lebruyn~\cite{BL91}.
\begin{lemma} For the transitive $S_n$-set $X_n$ with stabilizer subgroup $S_{n-1}$,
$$I_{X_n}\otimes \Z[X_n]\cong \Z[S_n/S_{n-2}]$$
\end{lemma}

\begin{proof}
Let $\{\e_i:i=1,\dots,n\}$ be the $\Z$-basis of the $S_n$-set $X_n$ permuted by $S_n$
via $\sigma(\e_i)=\e_{\sigma(i)}$. 
It suffices to show that $I_{X_n}\otimes \Z[X_n]$
has $\Z$-basis 
$$\{(\e_i-\e_j)\otimes \e_i: i\ne j\}$$
as then this basis is clearly transitively permuted by the action of $S_n$
with stabilizer subgroup $S_{n-2}$.

Since a $\Z$-basis of $I_{X_n}\otimes \Z[X_n]$ is 
given by 
$$\{(\e_i-\e_{i+1})\otimes \e_j: 1\le i\le n-1, 1\le j\le n\}$$
we need only show the $\Z$-span of each set contains the other.
Then
$$(\e_i-\e_j)\otimes \e_j=\sum_{k=i}^{j-1}(\e_k-\e_{k+1})\otimes \e_j, i<j$$
and
$$(\e_i-\e_j)\otimes \e_j=-\sum_{k=j}^{i-1}(\e_k-\e_{k+1})\otimes \e_j, i>j.$$
Conversely,
$$(\e_i-\e_{i+1})\otimes \e_j=(\e_j-\e_{i+1})\otimes \e_j-(\e_j-\e_i)\otimes \e_j.$$
\end{proof}

The following lemma was proved by Bessenrodt and Lebruyn but unpublished.
It was proved in Beneish~\cite{Ben98}.  Here is a simpler proof.

\begin{lemma} For $p$ prime, 
let $B_p$ be the $S_p$-lattice 
$$B_p=J_{X_p}\otimes I_{X_p}$$
Then
$$B_p\oplus \Z[X_p]\cong \Z[S_p/S_{p-2}]\oplus \Z$$ 
So $B_p$ is $S_p$-stably permutation.
\end{lemma}

\begin{proof} 
Tensoring the exact sequence
$$0\to I_{X_p}\to \Z [X_p]\to \Z\to 0 \qquad (*)$$
by $J_{X_p}=(I_{X_p})^*$, and noting that 
$$(I_{X_p})^*\otimes (\Z [S_p/S_{p-1}])^*\cong (I_{X_p}\otimes \Z [S_p/S_{p-1}])^*\cong \Z [S_p/S_{p-2}]$$
as well as the fact that permutation lattices are self-dual, 
we see that 
$$
\xymatrix{  
&\Z \ar[r]^{=}\ar@{>->}[d] &\Z\ar@{>->}[d]  \\
  B_p \ar@{>->}[r]\ar[d]^{=} & \mbox{pull-back} \ar@{->>}[r]\ar@{->>}[d] 
&\ar@{->>}[d]\Z[X_p]\\  
  B_p \ar@{>->}[r]& \Z[S_p/S_{p-2}] \ar@{->>}[r] &I_{X_p}^* 
}
$$
Since extensions of permutation lattices by permutation lattices 
are always split,
we have the exact sequence 
$$0\to B_p\to \Z\oplus \Z[S_p/S_{p-2}]\to \Z[X_p]\to 0$$
To prove the result, we need only show that 
$B_p$ is invertible since extensions of permutation lattices 
by invertible lattices are split.

To show that $B_p$ is $S_p$-invertible,
it suffices to show that $B_p$ is $Q$-invertible for each Sylow $q$-subgroup $Q$ of $S_p$.

Let $Q$ be a Sylow $q$-subgroup of $S_p$ where $q\ne p$.
Then $Q$ must fix some $\e_i\in \Z[X_p]$.
Then $I_{X_p}\vert_Q\oplus \Z=\Z[X_p]\vert_Q$. 
In fact $I_{X_p}\vert_Q$ is then $Q$-permutation with 
$\Z$-basis $\{\e_i-\e_j: j\ne i\}$.
Dualising we get
$(I_{X_p})^{*}\vert_Q\oplus \Z=\Z[X_p]\vert_Q$
and tensoring with $I_{X_p}\vert_Q$
we obtain
$$(B_p)\vert_Q\oplus I_{X_p}\vert_Q\cong \Z[S_p/S_{p-2}]\vert_Q$$
So $B_p\vert_Q$ is $Q$-stably permutation and hence $Q$-invertible.

It suffices to show that $B_p\vert P$ is $P$-invertible
where $P$ is a Sylow $p$-subgroup of $S_p$. Note that $P\cong C_p$
and  $\Z[X_p]_{C_p}\cong \Z[C_p]$. 
So 
$$0\to (B_p)_P\to (\Z\oplus \Z[S_p/S_{p-2}])_P\to \Z[P]\to 0$$
splits since $\Z[P]$ is free.

So $B_p$ is $S_p$-invertible and then 
the sequence 
$$0\to B_p\to \Z\oplus \Z[S_p/S_{p-2}]\to \Z[X_p]\to 0$$
splits to give us the result.
\end{proof}

\begin{defn}
For a group $G$, a $G$ module $M$ is cohomologically trivial if $\hat{H}^k(H,M)=0$ for all $k$
and for all subgroups $H$ of $G$. Projective $G$ modules are cohomologically
trivial.  A result in Brown~\cite[Theorem 8.10]{Bro82} shows that faithful cohomologically trivial $G$-lattices (torsion-free $G$ modules) are $G$-projective.
\end{defn}

\begin{lemma}\label{lem:cohtriv}  $\Fp I_{S_p/N_p}:=\Fp\otimes_{\Z}I_{S_p/N_p}$
is $S_p$-cohomologically trivial where $N_p=N_{S_p}(C_p)$ is the normaliser of a cyclic Sylow $p$-subgroup $C_p$ of $S_p$.
\end{lemma}

\begin{proof}
Note that tensoring the augmentation sequence for $S_p/N_p$
by $\Fp$ is exact.
Since $p$ does not divide $[S_p:N_p]$,
the $\Fp S_p$-exact  sequence
$$0\to \Fp I_{S_p/N_p}\to \Fp[S_p/N_p]\to \Fp\to 0$$
splits
and so 
$$\Fp[S_p/N_p]\cong \Fp I_{S_p/N_p}\oplus \Fp$$
It suffices to check whether the restrictions to Sylow subgroups 
are cohomologically trivial.
For any Sylow $q$-subgroup for $q\ne p$, representations of $\Fp Q$
are completely reducible and so all are projective and hence 
cohomologically trivial.
So it suffices to check whether $\Fp I_{S_p/N_p}\vert P$ is cohomologically trivial for  a cyclic Sylow $p$-subgroup $P=C_p$.
By Mackey's Theorem,
$$\Res^{S_p}_P\Ind^{S_p}_{N_p}\Fp=\oplus_{x\in P\backslash S_p/N_p}\Fp[P/P\cap N_p^x]$$
Since $P\cong C_p$, $P\cap N_p^x$ is either $P$ or $\{1\}$.
We claim that the unique double coset with $P\cap N_p^x=P$ is $PxN_p=N_p$.
Suppose $P\cap N_p^x=P$. Then $P\le N_p^x$ and so $P^{x^{-1}}\le N_p$.
But $P$ is the unique $p$-Sylow subgroup of $N_p$ and so $P^{x^{-1}}=P$
which means $x^{-1}\in N_p$.  Then we have $x\in N_p$ and $PxN_p=N_p$.
So for all non-trivial double cosets $PxN_p\ne N_p$, we have
$P\cap N_p^x=1$.
This means that 
$$\Res^{S_p}_P\Ind^{S_p}_{N_p}\Fp=\Fp\oplus (\Fp P)^k$$
for some $k$.
Since $\Fp[S_p]$ satisfies Krull Schmidt, we have that 
$$\Res^{S_p}_P\Fp I_{S_p/N_p}\oplus \Fp\cong \Fp \oplus (\Fp P)^k$$
implies that 
$\Res^{S_p}_P\Fp I_{S_p/N_p}\cong (\Fp P)^k$ is free
and so cohomologically trivial.
\end{proof}

\begin{prop}~\label{prop:clA5}
A projective  $A_5$-lattice is 
$A_5$-stably permutation.
\end{prop}

\begin{proof}
By Endo and Miyata~\cite{EM76}, the projective class group of the group ring of $A_5$, $\Z [A_5]$, is $\Cl(\Z [A_5])=0$.  But the projective class group is the 
kernel of the rank homomorphism $K_0(\Z [A_5])\to \Z$ which sends any finitely
generated projective $A_5$ lattice to its rank.
Since this class group is zero, it shows that any projective $A_5$-lattice 
has the same class in $K_0(\Z A_5)$ as a free $A_5$-lattice of the same rank.
But then by ~\cite[38.22]{CR87}, we see that a projective $A_5$-lattice is $A_5$-stably free.
This shows that all cohomologically trivial faithful $A_5$-lattices are stably free and hence
stably permutation.
\end{proof}

\begin{remark} Projective $G$-lattices are stably permutation  for any finite group $G$
with splitting field $\Q$ (e.g. $S_n$).~\cite{EM76}, see also~\cite[Lemma 2.3.1]{Lor05}. This was used in~\cite{BL91,Ben98}.
Note that the splitting field for $A_5$ is $\Q(\sqrt{5})$.
\end{remark}

\begin{lemma} An $\Fp G$-module $M$ is projective if and only if
$\Res^G_PM$ is projective for a Sylow $p$-subgroup $P$ of $G$.
In particular, if $G$ is a transitive subgroup of $S_p$, then for the $G$-set $X_p\cong G/G\cap S_{p-1}$, we have that $\Fp[X_p]$ is projective as an $\Fp G$-module.\label{lem:fpproj}
\end{lemma}

\begin{proof}
The natural surjection 
$\pi:\Ind^G_P\Res^G_PM\to M$
has a section
$$s:M\to \Ind^G_P\Res^G_PM, m\to \frac{1}{[G:P]}\sum_{gP\in G/P}g\otimes g^{-1}m$$
since $[G:P]$ is invertible in $\Fp$.
Then $M$ is an $\Fp G$ direct summand of $\Ind^G_P\Res^G_PM$.

If $\Res^G_PM$ is projective, so is $\Ind^G_P\Res^G_PM$ and hence so 
is $M$ by the previous remark.

The converse is clear.

Since $G$ is a transitive subgroup of $S_p$, it has a cyclic $p$-Sylow subgroup
$C_p$. Then $\Fp[X_p]$ restricted to the cyclic $p$-Sylow subgroup $C_p$ is isomorphic to the 
free $\Fp[C_p]$-module  $\Fp[C_p]$.
\end{proof}

\begin{lemma}~\label{lem:flasque} If $0\to M\to P\to L\to 0$ and $0\to M'\to Q\to L\to 0$
are 2 short exact sequences of $G$ modules with $P,Q$ $G$-permutation,
then $M\sim M'$ where $\sim$ denotes flasque equivalence.
\end{lemma} 

\begin{proof} The pullback diagram gives two exact 
sequences $0\to M\to E\to Q\to 0$ and $0\to M'\to E\to P\to 0$.
Note that the pullback module $E$ is a $G$-lattice (i.e. is $\Z$-torsion
free).  So $M\sim M'$ as required.
\end{proof}

\begin{prop} Suppose $G$ is a transitive subgroup of $S_p$
for which all $G$-projective lattices are $G$-stably permutation.
Let $N=N_G(C_p)$ be the normaliser subgroup of a (cyclic) Sylow $p$ subgroup $C_p$.
Let $M$ be a $G$-lattice such that
there exists a short exact sequence of $G$ modules
$$0\to M\to P\to X\to 0$$
where
\begin{itemize}
\item $P$ is $G$-permutation.
\item $X$ $p$-torsion 
\item  $X\otimes \Fp I[G/N]$ is cohomologically trivial
\end{itemize}
then
$$M\sim \Ind^G_N\Res^G_N(M)$$ 
where  $\sim$ denotes flasque equivalence.
\end{prop}

\begin{proof}
Note that the hypothesis of transitivity implies that a Sylow $p$ subgroup 
$C_p$ of $G$ is cyclic of order $p$.
So $N=N_G(C_p)\le N_{S_p}(C_p)=C_p\rtimes C_{p-1}$
and so $N=N_G(C_p)=C_p\rtimes (C_{p-1}\cap G)$.

Applying $\Ind^G_N\Res^G_N$,
we get 
$$0\to \Ind^G_N\Res^G_NM\to \Ind^G_N\Res^G_NP\to \Ind^G_N\Res^G_NX\to 0$$
Since $X$ is $p$-torsion,
$\Ind^G_N\Res^G_NX= \Fp[G/N]\otimes_{\Fp}X$.
We have already noted that $\Fp[G/N]=\Fp\oplus \Fp I_{G/N}$.
So 
\begin{equation}
\Ind^G_N\Res^G_NX=X\oplus (\Fp I_{G/N}\otimes X)\label{eq:indresx}
\end{equation}

By hypothesis, $\Fp I_{G/N}\otimes X$ is cohomologically trivial.
Let 
$$0\to K\to F\to \Fp I_{G/N}\otimes X\to 0$$
be an exact sequence of $G$-modules with $F$ a free $G$-module.
Then $K$ is also cohomologically trivial and $\Z$-free.
By~\cite[Theorem 8.2]{Bro82}, this implies that $K$ is $G$-projective and hence $G$-stably permutation
by hypothesis.
Adding this sequence to the original and recalling (\ref{eq:indresx}), we obtain
$$0\to M\oplus K\to P\oplus F\to \Ind^G_N\Res^G_NX\to 0$$
But then by Lemma~\ref{lem:flasque}, we see that $M\oplus K\sim \Ind^G_N\Res^G_NM$.
Since $K$ is $G$-stably permutation, we see that $M\sim M\oplus K$
and so $M\sim \Ind^G_N\Res^G_NM$.

\end{proof}

\begin{cor} For an odd prime $p$, the transitive $S_p$-set $X_p$, 
 and $N=N_{S_p}(C_p)$, we have that
that $$J_{X_p}\sim \Ind^{S_p}_{N}J_{N/C_{p-1}}$$
as $S_p$-lattices.

Also, for the transitive $A_5$-set $X_5$,
we have that
$$J_{X_5}\sim \Ind^{A_5}_{D_{10}}J_{D_{10}/C_2}$$
as $A_5$-lattices. It follows that $J_{X_5}$ is $A_5$-quasi-permutation.
\end{cor}

\begin{proof}
We will apply the previous proposition.  Note that $G=S_p$, for an odd prime $p$ and $G=A_5$
satisfy the hypothesis that all projective $G$-lattices are $G$-stably permutation.
Note also that $A_5$ is a transitive subgroup of $S_5$ such that the normaliser of a cyclic
subgroup of order 5 is $D_{10}$, the dihedral group of order 10.
We need to construct an appropriate $G$-exact sequence to apply the proposition.

For any $G$ set $Y$ of size $n$,
there is an inclusion of $G$-lattices
$\alpha: I_{Y}\oplus \Z\to \Z[Y]$
where $\alpha\vert_{I_{Y}}$ is the inclusion and for $n=|Y|$ and $\{\e_i:i=1,\dots,n\}$ a $\Z$-basis of $\Z[Y]$, 
$\alpha:\Z\to \Z[Y], 1\to \sum_{i=1}^n\e_i$.
Since $\{\e_1-\e_2,\dots,\e_{n-1}-\e_n,\e_n\}$
is a basis for $\Z[Y]$
and 
$\{\e_1-\e_2,\dots,\e_{n-1}-\e_n,n\e_n\}$
is a basis for $I_Y\oplus \Z$,
it is easily checked that 
$$0\to I_Y\oplus \Z\to \Z[Y]\to \Z/n\Z\to 0$$
is a short exact sequence of $G$-lattices
with $\Z/n\Z$ having trivial action.

Letting $Y=X_p$, we set $B_p=J_{X_p}\otimes I_{X_p}=(I_{X_p})^*\otimes I_{X_p}$
where $X_p$ is a transitive $G$-set of size $p$. Then tensoring by $J_{X_p}$,  
 we obtain
$$0\to B_p\oplus J_{X_p}\to \Z[X_p]\otimes J_{X_p}\to \Fp J_{X_p}\to 0$$
We need to show that $\Fp J_{X_p}\otimes \Fp I_{G/N}$ is $G$-cohomologically trivial.
Tensoring the following $G$-exact sequence by $\Fp I_{G/N}$:
$$0\to \Fp\to \Fp[X_p]\to \Fp J_{X_p}\to 0$$
we obtain
$$0\to \F_pI_{G/N}\to \Fp[X_p]\otimes I_{G/N}\to \Fp J_{X_p}\otimes I_{G/N}\to 0$$
Now $\Fp[X_p]$ is $\Fp[G]$-projective by Lemma~\ref{lem:fpproj}. So $\Fp[Y]\otimes I_{G/N}$ is also 
$\Fp G$-projective and so cohomologically trivial as an $\Fp G$ module.  We have 
already seen that $\F_pI_{G/N}$ is cohomologically trivial as an
$\Fp G$-module. This shows that $(\Fp J_{X_p})\otimes (\Fp I_{G/N})$
is cohomologically trivial as an $\Fp G$-module and hence also as a $G$-module as $\Fp G$ is cohomologically trivial as a $G$-module.

Note that we have constructed an exact sequence of the required
form for $B_p\oplus J_{X_p}$.  But since $B_p$ is $G$-stably 
permutation, the fact that $B_p\oplus J_{X_p}\sim \Ind^G_N\Res^G_N(B_Y\oplus J_y)
=\Ind^G_N\Res^G_N(B_p)\oplus \Ind^G_N\Res^G_N(J_{X_p})$ shows that $J_{X_p}\sim
\Ind^G_N\Res^G_N(J_{X_p})$ since $\Ind^G_N\Res^G_N$ preserve stably permutation
lattices.

For $G=S_p$, $p$ an odd prime, and $N=N_{S_p}(C_p)=C_p\rtimes C_{p-1}$, we see that
$\Res^G_N(J_{X_p})=J_{N/C_{p-1}}$ and 
so $J_{X_p}\sim \Ind^{S_p}_{N}J_{N/C_{p-1}}$ as $S_p$-lattices.

For $G=A_5, G\cap S_4=A_4$, $N=N_G(C_5)=D_{10}$ and $N\cap S_4=C_2$,
so 
$$J_{X_5}=J_{A_5/A_4}\sim \Ind^{A_5}_{D_{10}}(J_{D_{10}/C_2})$$
as $A_5$-lattices as required.
\end{proof}

\begin{prop} For an odd prime $p$ and the transitive $S_p$ set $X_p$ of size $p$,
a flasque resolution of $J_{X_p}$ is given
by 
$$0\to J_{X_p}\to \Z[S_p/S_{p-2}]\to J_{X_p}^{\otimes 2}\to 0$$
In fact, the $S_p$-lattice $J_{X_p}^{\otimes 2}$ is invertible.
\end{prop}

\begin{proof}
This is well-known and was proven in~\cite{BL91} using somewhat different
language.
As we require this result, we give a quick self-contained proof.
Tensoring the $S_p$-exact sequence 
$$0\to \Z\to \Z[X_p]\to J_{X_p}\to 0$$
by $J_{X_p}$ and noting that $J_{X_p}\otimes \Z[X_p]\cong (I_{X_p}\otimes \Z[X_p])^*\cong \Z[S_p/S_{p-2}]$ by *,
we obtain the $S_p$-exact sequence of the statement.

It suffices to show that $J_{X_p}^{\otimes 2}$ is invertible when restricted to Sylow $q$-subgroups of $S_p$. Note that  any Sylow $q$ subgroup $Q$ of $S_p$
for $q\ne p$, must fix $\e_i$ for some $i=1,\dots,p$, where
$\Z[X_p]$ has $\Z$-basis $\e_i,i=1,\dots,p$.  Then $\Res^{S_p}_Q(I_{X_p})$
permutation $Q$-lattice with $\Z$-basis
$\e_i-\e_j, j\ne i$.  So its dual $J_{X_p}$ and $J_{X_p}^{\otimes 2}$
must also be $Q$-permutation  lattices.
As for the Sylow $p$-subgroup $C_p$, a result of Endo-Miyata shows that 
for cyclic $p$-groups, every flasque lattice is invertible.
Since $(J_{X_p}^{\otimes 2})^*=I_{X_p}^{\otimes 2}$, we need only 
check that $\Res^{S_p}_{C_p}I_{X_p}^{\otimes 2}=I_{C_p}^{\otimes 2}$ is coflasque.
Dualising the exact sequence of the statement and restricting to $C_p$, we obtain the $C_p$-exact sequence
$$0\to (I_{C_p})^{\otimes 2}\to \Z[C_p]^{p-1}\to I_{C_p}\to 0$$
But then since $(I_{C_p})^{C_p}=0$, and $H^1(C_p,\Z[C_p]^{p-1})=0$,
we see that $H^1(C_p,(I_{C_p})^{\otimes 2})=0$ as required.
\end{proof}

Recall the following useful lemma from Florence~\cite{Flo1}.  The original was
stated for lattices for a profinite group.  The proof for $G$-lattices
follows immediately.

\begin{lemma} Let $A_i,B_i,C_i, i=1,2$ be $G$-lattices fitting into 
two exact sequences 
$$0\to A_i\stackrel{j_i}{\to}B_i\stackrel{\pi_i}{\to}C_i\to 0$$
Assume we are given ax $G$ module map $s_i:C_i\to B_i$, and $d_1,d_2$ two coprime integers, such that $\pi\circ s_i=d_i\id$, $i=1,2$. Let  $A_3=A_1\otimes A_2$,
$$B_3=(B_1\otimes B_2)\oplus (C_1\otimes C_2),\qquad
C_3=(C_1\otimes B_2) \oplus (B_1\otimes C_2).$$
Then there is an exact sequence
$$0\to A_3\stackrel{j_3}{\to} B_3\stackrel{\pi_3}{\to} C_3\to 0$$
together with a $G$ module map $s_3:C_3\to B_3$ such that 
$\pi_3\circ s_3=d_1d_2\id$.
\end{lemma}

\begin{remark} Observe that this lemma is very handy for showing that 
in certain circumstances, the tensor product of two quasi-permutation lattices
is again quasi-permutation.  Indeed, with the hypotheses of the Lemma
and the additional assumption that $B_i,C_i,i=1,2$ are all permutation lattices,
then $B_3,C_3$ are also permutation, since the tensor product of permutation
lattices is permutation and the direct sum of permutation lattices is 
permutation.  Indeed, in Theorem 2.2 of the same paper, he 
shows that the tensor product of augmentation ideals of $G$ sets of 
pairwise relatively prime order is quasipermutation as a consequence of this
Lemma.  He then goes on to give a simple proof of Klyachko's result that 
a $k$-torus with character lattice isomorphic to the tensor product 
of two augmentation ideals for $G$ sets of relatively 
prime order is rational.

We will apply this Lemma to prove that the $A_5\times C_2$ lattice 
$J_{X_5}\otimes \Z^-_{A_5}$ is quasi-permutation.
\end{remark}

\begin{remark}
Suppose 
$$0\to A\to B\stackrel{\pi}{\to} C\to 0$$
is an exact sequence of $G$-lattices and $s:C\to B$ is a $G$-equivariant map
such that $\pi\circ s=n\id$.
Let $D$ be another $G$-lattice.
Then for the $G$-exact sequence 
$$0\to (D\otimes A)\to (D\otimes B)\stackrel{\hat{\pi}}{\to} (D\otimes C)\to 0$$
where $\hat{\pi}=(\id_D\otimes \pi)$,
there exists a $G$-equivariant map 
$$\hat{s}=(\id_D\otimes s): (D\otimes C)\to 
(D\otimes B)$$
such that $\hat{\pi}\circ \hat{s}=n\id$.

In particular, this remark applies to the natural exact sequence for the $G$-lattice $J_{X_n}$ where $X_n$ is a $G$-set of size $n$.  
Then for the $G$-exact sequence 
$$0\to \Z\to \Z[X_n]\stackrel{\pi}{\to} J_{X_n}\to 0$$
there is a natural $G$-equivariant map $s:J_{X_n}\to \Z[X_n]$
given by $s(\pi(x))=nx-\sum_{y\in X}y$ which satisfies $\pi\circ s=n\id$.
Tensoring this sequence with $J_{X_n}$,
we obtain a $G$-exact sequence
$$0\to J_{X_n}\to J_{X_n}\otimes \Z[X_n]\stackrel{\hat{\pi}}\to J_{X_n}^{\otimes 2}\to 0$$
For this sequence, there exists a $G$-equivariant map $\hat{s}:J_{X_n}^{\otimes 2}\to J_{X_n}\otimes \Z[X_n]$ such that $\hat{\pi}\circ \hat{s}=n\id$.
\end{remark}

Recall that:

\begin{theorem}~\cite{Leb95,Flo2} Let $K/k$ be a separable extension of prime degree $p$, $p\ge 5$ with Galois closure $L/k$ having Galois group $\Gal(L/k)=C_p\rtimes C_{p-1}$
and $H=\Gal(L/K)=C_{p-1}$.  Then $R^{(1)}_{K/k}(\Gm)$ is not a stably 
rational variety.
\end{theorem}

Note that this is equivalent to $J_{X_p}$ is not stably permutation as an $C_p\rtimes C_{p-1}$ lattice if $p\ge 5$ is prime.  This applies in particular to the $F_{20}=C_5\rtimes C_4$-lattice $J_{X_5}$.

\begin{prop}
For a prime $p\ge 5$, the $S_p\times C_2$-lattice $J_{X_p}\otimes \Z^-_{S_p}$ is $S_p\times C_2$-quasi-invertible but not $S_p\times C_2$-quasipermutation.

The $A_5\times C_2$-lattice $ J_{X_5}\otimes \Z^-_{A_5}$  is
$A_5\times C_2$-quasipermutation.
\end{prop}

\begin{proof} Let $p\ge 5$ be a prime.
A flasque resolution for $J_{X_p}$ for the transitive $S_p$-set $X_p$ can be given by
$$0\to J_{X_p}\to J_{X_p}\otimes \Z[X_p]\stackrel{\pi_1}{\to} J_{X_p}^{\otimes 2}\to 0$$
As we have seen, $J_{X_p}\otimes \Z[X_p]$ is permutation as the dual of 
$I_{X_p}\otimes \Z[X_p]$.  We have shown that $J_{X_p}^{\otimes 2}$ is $S_p$-quinvertible.  
We see from the remark that there exists an $S_p$-equivariant map
$$s_1:J_{X_p}^{\otimes 2}\oplus P\to \Z[X_p]\otimes J_{X_p}\oplus P$$
such that $\pi_1\circ s_1=p\id$.
Inflating this sequence from $S_p$ to $S_p\times C_2$ gives us the 
same statements for the $S_p\times C_2$-set $X_p$.
That is, the above sequence can be considered also a flasque resolution for
the $S_p\times C_2$-lattice $J_{X_p}$ with all maps considered above 
$S_p\times C_2$-equivariant. 
We remark that the sign lattice $\Z^{-}_{S_p}$ for $S_p\times C_2$ is in
fact the augmentation ideal $I_{Y_2}$ for $Y_2=(S_p\times C_2)/S_p$
and so its augmentation sequence
$$0\to I_{Y_2}\to \Z[Y_2]\stackrel{\pi_2}{\to}\Z\to 0$$
admits an $S_p\times C_2$-equivariant map $s_2:\Z\to \Z[Y_2]$ such that 
$\pi_2\circ s_2=2\id$.
We may apply Florence's Lemma to the 2 exact sequences above to obtain an
$S_p\times C_2$-exact sequence:
\begin{equation} \label{eq:signedweightseq}
0\to J_{X_p}\otimes I_{Y_2}\to M\to N\to 0
\end{equation}
where 
$$M=J_{X_p}\otimes \Z[X_p]\otimes \Z[Y_2]\oplus J_{X_p}^{\otimes 2}\otimes \Z$$
$$N=J_{X_p}\otimes \Z[X_p]\otimes \Z\oplus J_{X_p}^{\otimes 2}\otimes \Z[Y_2]$$
Since $J_{X_p}\otimes \Z[X_p]$ is permutation and $J_{X_p}^{\otimes 2}$ is 
invertible as $S_p\otimes C_2$-lattices, we see that the same holds for $M$ and $N$.  
We may then find an appropriate $S_p\times C_2$-lattice $L$ such that $M\oplus L\cong P$ is permutation.  Note that $L$ is also invertible.
Then the $S_p\times C_2$-exact sequence
$$0\to J_{X_p}\otimes I_{Y_2}\to P\to N\oplus L\to 0$$
shows that the $S_p\times C_2$-lattice $J_{X_p}\otimes \Z^-_{S_p}=J_{X_p}\otimes I_{Y_2}$ is quasi-invertible as required.

Note though, that the $S_p\times C_2$-lattice $J_{X_p}\otimes \Z^-_{S_p}$ is not quasi-permutation as its restriction to $C_p\rtimes C_{p-1}$ is the lattice 
$J_{X_p}$ which is not quasi-permutation.

Now setting $p=5$, we restrict the above sequence (\ref{eq:signedweightseq}) to $A_5\times C_2$.
We have an $A_5\times C_2$-exact sequence
$$0\to J_{X_5}\otimes I_{Y_2}\to M\to N\to 0$$
where
$$M=J_{X_5}\otimes \Z[X_5]\otimes \Z[Y_2]\oplus J_{X_5}^{\otimes 2}\otimes \Z$$
$$N=J_{X_5}\otimes \Z[X_5]\otimes \Z\oplus J_{X_5}^{\otimes 2}\otimes \Z[Y_2]$$
Since the $A_5$-lattice $J_{X_5}$ is quasi-permutation, its inflation, the $A_5\times C_2$-lattice $J_{X_5}$ is also quasi-permutation.
But then the its flasque lattice, the $A_5\times C_2$-lattice $J_{X_5}^{\otimes 2}$, is $A_5\times C_2$-stably permutation.  This shows that as  $A_5\times C_2$-lattices,  $M$ 
and $N$ are stably permutation.

We may then find an appropriate $A_5\times C_2$-permutation lattice $Q$ such that $M\oplus Q\cong P$ is permutation.  
Then the $A_5\times C_2$-exact sequence
$$0\to J_{X_5}\otimes I_{Y_2}\to P\to N\oplus Q\to O$$
shows that the $A_5\times C_2$-lattice $J_{X_5}\otimes \Z^-_{S_5}=J_{X_5}\otimes I_{Y_2}$ is quasi-permutation as required.
\end{proof}

\begin{lemma} $\Z[F_{20}/D_{10}]\otimes J_{F_{20}/C_4}^{\otimes 2}$ is 
$F_{20}$-stably permutation.  
$\Z[S_5/A_5]\otimes J_{S_5/S_4}^{\otimes 2}$ is $S_5$-stably permutation.  
\end{lemma}

\begin{proof}
For any $G$-lattice $M$ and a subgroup $H$ of $G$,
$$\Z[G/H]\otimes M\cong \Ind^G_H\Res^G_H(M)$$
Since we have seen that 
$$0\to J_{X_5}\to J_{X_5}\otimes \Z[X_5]\to J_{X_5}^{\otimes 2}\to 0$$
is a flasque resolution for the $S_5$-lattice $J_{X_5}$,
it is also a flasque resolution for its restrictions to any subgroup.
Note that $J_{X_5}\cong J_{S_5/S_4}$ as an $S_5$-lattice,
$J_{X_5}\cong J_{A_5/A_4}$ as an $A_5$-lattice,
$J_{X_5}\cong J_{F_{20}/C_4}$ as an $F_{20}$-lattice and 
$J_{X_5}\cong J_{D_{10}/C_2}$ as an $D_{10}$-lattice.
Since the $D_{10}$ lattice $J_{X_5}\cong J_{D_{10}/C_2}$ is
quasi-permutation, $J_{X_5}^{\otimes 2}$ must be stably permutation as  $D_{10}$-lattice.
Then 
$$\Z[F_{20}/D_{10}]\otimes J_{X_5}^{\otimes 2}\cong \Ind^{F_{20}}_{D_{10}}
\Res^{F_{20}}_{D_{10}}J_{X_5}$$
shows that $\Z[F_{20}/D_{10}]\otimes J_{X_5}^{\otimes 2}$
is $F_{20}$-stably permutation.
Similarly, since the 
$A_{5}$-lattice $J_{X_5}\cong J_{A_{5}/A_4}$ is
quasi-permutation, $J_{X_5}^{\otimes 2}$ must be stably permutation as 
an  $A_{5}$-lattice.
Then 
$$\Z[S_5/A_5]\otimes J_{X_5}^{\otimes 2}\cong \Ind^{S_5}_{A_5}
\Res^{S_5}_{A_5}J_{X_5}$$
shows that $\Z[S_5/A_5]\otimes J_{X_5}^{\otimes 2}$
is $S_5$-stably permutation.
\end{proof}

\begin{prop} For the $S_5$-lattices $J_{X_5}$ and 
$J_{X_5}\otimes \Z^-_{A_5}$, we have 
$$\rho_{S_5}(J_{X_5}\otimes \Z^-_{A_5})=-\rho_{S_5}(J_{X_5})\ne 0$$
The GAP IDs of the corresponding groups are $[4,31,4,2]$ and $[4,31,5,2]$.

For the $F_{20}$-lattices $J_{X_5}$ and $J_{X_5}\otimes \Z^-_{D_{10}}$,
$$\rho_{F_{20}}(J_{X_5}\otimes \Z^-_{D_{10}})=-\rho_{F_{20}}(J_{X_5})\ne 0$$
The GAP IDs of the corresponding groups are $[4,31,1,3]$ and $[4,31,1,4]$.
\end{prop}

\begin{proof}
From the remark, we have observed that the $S_5$-flasque resolution
for $J_{X_5}$ given by 
$$0\to J_{X_5}\to J_{X_5}\otimes \Z[X_5]\stackrel{\pi_1}{\to} J_X^{\otimes 2}\to 0$$
admits an equivariant map $s_1:J_{X_5}^{\otimes 2}\to J_{X_5}\otimes \Z[X_5]$
such that $\pi_1\circ s_1=5\id$.
For the second sequence, we will take instead the $S_5$-set $S_5/A_5$,
and the quasi-permutation resolution for $I_{S_5/A_5}=\Z^-_{A_5}$
given by 
$$0\to I_{S_5/A_5}\to \Z[S_5/A_5]\stackrel{\pi_2}{\to} \Z\to 0$$
which admits an equivariant map $s_2:\Z\to \Z[S_5/A_5]$ such that 
$\pi_2\circ s_2=2\id$.
So the 2 sequences satisfy the hypotheses for Florence's Lemma
and we obtain an $S_5$-exact sequence
$$0\to  J_{X_5}\otimes I_{S_5/A_5}\to M\to Q\to 0$$
where 
$$M=J_{X_5}\otimes \Z[X_5]\otimes \Z[S_5/A_5]\oplus J_{X_5}^{\otimes 2}\otimes \Z$$
$$Q=J_{X_5}\otimes \Z[X_5]\otimes \Z\oplus J_{X_5}^{\otimes 2}\otimes \Z[S_5/A_5]$$

Since $J_{X_5}\otimes \Z[X_5]$ is $S_5$-permutation and 
$J_{X_5}^{\otimes 2}\otimes \Z[S_5/A_5]$ is $S_5$-permutation by the Lemma,
we see that $Q$ is permutation.  Since also $J_{X_5}^{\otimes 2}\otimes \Z\cong J_{X_5}^{\otimes 2}$ is $S_5$-invertible,
we see that $M$ is also $S_5$-invertible. In fact if $L$ is an $S_5$-lattice 
such that $J_{X_5}^{\otimes 2}\oplus L$ is $S_5$-permutation, we may
adjust this $S_5$-exact sequence to 
$$0\to J_{X_5}\otimes I_{S_5/A_5}\to M\oplus L\to Q\oplus L\to 0$$
Then $Q\oplus L$ is $S_5$-invertible, $M\oplus L$ is $S_5$-permutation,
and the $S_5$-flasque class of $J_{X_5}\otimes I_{S_5/A_5}$ is $[Q\oplus L]=[L]$
since $Q$ is permutation. But the $S_5$-flasque class of $J_{X_5}$
is $[J_{X_5}^{\otimes 2}]$ and since $J_{X_5}^{\otimes 2}\oplus L$ is permutation, we see that for the $S_5$-lattices $J_{X_5}\otimes \Z^-_{A_5}$
$$\rho_{S_5}(J_{X_5}\otimes \Z^-_{A_5})=-\rho_{S_5}(J_{X_5})$$
as required.  We also know that $\rho_{S_5}(J_{X_5})\ne 0$ as the lattice $J_{X_5}$
restricted to $F_{20}$ is not quasi-permutation and so neither is $J_{X_5}$
as an $S_5$-lattice.

The above argument restricted to the subgroup $F_{20}$ gives the other statements, where we note that $S_5$-set $S_5/A_5$ restricted to $F_{20}$ is 
$F_{20}/D_{10}$.

The GAP ID identifications will be discussed below.
\end{proof}

Note that the Corollary was obtained computationally in~\cite{HY12}.

By results of Hoshi and Yamasaki in~\cite{HY12}, there are 7 conjugacy classes of finite subgroups of $\GL_4(\Z)$ which correspond to 
retract but not stably rational algebraic tori.  It turns out that the DadeGroup(4,7) with GAP ID [4,31,7,2] and corresponding lattice
$(\Lambda(A_4),\Aut(A_4))$ is one such example.  All but one of the other such 
examples can be seen to be subgroups (not just conjugate subgroups) of this group.  Their lattices correspond to restrictions of $\Lambda(A_4)$ to the appropriate subgroup. The justifications are
similar to identifying the lattice corresponding to DadeGroup(4,7) since they correspond to subgroups.  For this reason, they are omitted. 

\begin{prop}
The following subgroups  of \textup{DadeGroup(4,7)} with GAP ID \textup{[4,31,7,2]}
and lattice $$(\Lambda(A_4),\Aut(A_4))=(J_{X_5}\otimes \Z^-_{S_5},S_5\times C_2)$$ have GAP IDs and lattices given by:
\begin{enumerate}
\item \textup{[4,31,1,3]:}  $(J_{X_5},F_{20})$.
\item \textup{[4,31,1,4]:} $(J_{X_5}\otimes \Z^-_{D_{10}},F_{20})$
\item \textup{[4,31,2,2]:} $(J_{X_5}\otimes \Z^-_{F_{20}},F_{20}\times C_2)$.
\item \textup{[4,31,4,2]:} $(J_{X_5},S_5)$.
\item \textup{[4,31,5,2]:} $(J_{X_5}\otimes \Z^-_{A_5},S_5)$.
\end{enumerate}
All of these correspond to retract rational tori which are not stably rational.
\end{prop}

\begin{proof}
Note that these observations for [4,31,1,3] and [4,31,4,2] were made in Hoshi
Yamasaki.  It was shown above  that 
$(J_{X_p}\otimes \Z^-_{S_p},S_p\times C_2)$ is quasi-invertible for $p\ge 5$ prime (or equivalently the corresponding norm one torus is retract
rational). Restrictions of quasi-invertible
lattices are quasi-invertible.  So all the lattices on the list above are 
quasi-invertible with corresponding tori retract rational.

By the result above, we know that 
$(J_{X_5},F_{20})$ is not quasi-permutation. So $\rho(J_{X_5},F_{20})\ne 0$.   Since every lattice in the list above
contains either $(J_{X_5},F_{20})$ or $(J_{X_5}\otimes \Z^-_{D_{10}},F_{20})$,
we see that the result will follow since
$$\rho_{F_{20}}(J_{X_5}\otimes \Z^-_{D_{10}})=-\rho_{F_{20}}(J_{X_5})\ne 0$$
\end{proof}

We will now address the  remaining case of a retract but not stably rational 
algebraic $k$-torus of dimension 4.

\begin{remark} For any cyclic group $C_n$ and given a primitive
$n$th root of unity $\omega_n$, there is a natural
$C_n$ lattice given by $\Z[\omega_n]$, which is the ring of integers of the 
field $\Q(\omega_n)$.  The generator of $C_n$, $\sigma_n$ then acts
as multiplication by $\omega_n$ on $\Z[\omega_n]$.
As a ring, $\Z[\omega_n]\cong \Z[X]/(\Phi_n(X))$ where $\Phi_n(X)$ is the 
$n$th cyclotomic polynomial of degree $\varphi(n)$, the Euler $\phi$ function of $n$.
With respect to the basis $1,\omega_n,\dots,\omega_n^{\varphi(n)-1}$,
the matrix of $\sigma_n$ acting on $\Z[\omega_n]$ is the 
companion matrix of $\Phi_n(X)$.
\end{remark}

\begin{prop} The lattice corresponding to the group with GAP ID \textup{[4,33,2,1]} is quasi-invertible
but not quasi-permutation and hence the associated algebraic 
torus is retract but not stably rational.
\end{prop}

\begin{proof}
The generators of the group with GAP ID
[4,33,2,1] given by GAP 
are 
$$A=\left[\begin{array}{rrrr}0&0&-1&0\\-1&0&0&0\\1&1&1&-2\\0&1&0&-1\end{array}\right],B=\left[\begin{array}{rrrr}0&1&0&-1\\0&0&-1&1\\-1&0&0&1\\0&1&0&0\end{array}\right]$$
It is easy to check that $A$ has order 8 and $B$ has order 12.
One could then replace the generators by $A$ and $C=B^4$
and check that $ACA^{-1}=C^{-1}$.  This implies that for the semi-direct product 
$$G=C_3\rtimes C_8=\langle \sigma,\tau: \sigma^3=1,\tau^8=1,\tau\sigma\tau^{-1}=\sigma^{-1}\rangle$$
$\rho:G\to \GL(4,\Z)$ with $\rho(\sigma)=C$ and $\rho(\tau)=A$
gives a faithful representation of the group.

Note that $G$ has centre
$\Z(G)=\langle \tau^2\rangle\cong C_4$ and $H_3=\langle \sigma\rangle\cong C_3$
is normal.  
The group $H_8=\langle \tau\rangle\cong C_8$ is a Sylow 2-subgroup with
3 distinct conjugates $H_8=\langle \tau\rangle, (H_8)^{\sigma}=\langle \tau\sigma\rangle,(H_8)^{\sigma^{2}}=\langle \tau\sigma^2\rangle$ which pairwise intersect
in $\langle \tau^2\rangle$.
There are then $3\varphi(8)=12$ elements of order 8.
Since $\tau^2$ is central, $\tau^2\sigma$ is an element of order 12.
As the group has order 24, the elements of 
$H_{12}=\langle \tau^2\sigma\rangle\cong C_{12}$ account for the remaining 12 elements of the group.
This shows that the only proper subgroups of $G$ are cyclic.
There is exactly one subgroup $H_d$ of $G$ of order $d|12$ and each is a cyclic subgroup of $H_{12}\cong C_{12}$, and is normal in $G$.
The only non-normal subgroups are the 3 conjugates of $H_8$ which 
are cyclic of order 8.  

Let $L$ be the rank 4 lattice with the action of $G$ induced by $\rho$.
We will show that $L_{H_3}\cong (I_{C_3})^2$ and $L_{H_8}\cong \Z[\omega_8]$.

For the restriction to $H_3=\langle \sigma\rangle$,
 since the minimal polynomial of $C=\rho(\sigma)$ is $x^2+x+1$,
and $C\in \GL(4,\Z)$, it must have invariant factors $x^2+x+1,x^2+x+1$.
This is sufficient to show that $\Q L\cong \Q[\omega_3]^2\cong \Q I_{C_3}^2$ as $\Q H_3$-modules.  We need to check that $L\cong \Z[\omega_3]^2\cong I_{C_3}^2$ as $H_3\cong C_3$-lattices.
In fact, one can show that the matrix of $\rho(\sigma)=C$ with respect to 
the $\Z$-basis 
$$\{\e_1,C\e_1,\e_3,C\e_3\}=\{\e_1,\e_2-\e_4,\e_3,-\e_1+\e_4\}$$
of $L$ is $\diag(C_{x^2+x+1},C_{x^2+x+1})$
where 
$$C_{x^2+x+1}=\left[\begin{array}{rr}0&1\\-1&-1\end{array}\right]$$
is the companion matrix of $x^2+x+1$.  We may then conclude that $L_{H_3}\cong (I_{C_3})^2$.

For the restriction to $H_8$, since $A=\rho(\tau)$ is a matrix of order 8 in $\GL_4(\Z)$, it must have minimal polynomial $x^4+1$.
Again this is sufficient to conclude that the $\Q H_8$-module $\Q L$ is congruent to $\Q[\omega_8]$.
In fact, one can show that the matrix of $A$ with respect to the 
$\Z$-basis 
$$\{\e_4,A\e_4,A^2\e_4,A^3\e_4\}=\{\e_4,\e_2-\e_4,-\e_1-\e_2+\e_4,\e_1+\e_2+\e_3-\e_4\}$$
of $L$ is the companion matrix 
$C_{x^4+1}$ of $x^4+1$ given by 
$$\left[\begin{array}{rrrr}0&1&0&0\\0&0&1&0\\0&0&0&1\\-1&0&0&0\end{array}\right]$$
So  $L$ restricted to $H_8=\langle \tau\rangle\cong C_8$ must be isomorphic to $\Z[\omega_8]$.  We see that this lattice  is a sign-permutation lattice
isomorphic to $\Ind^{C_8}_{C_2}\Z_-$.

Note that $L$ restricted to $H_3$ is quasi-permutation as it is isomorphic to 
$I_{C_3}^2$, the direct sum of 2 quasi-permutation lattices
and $L_{H_8}$ is quasi-permutation as it is sign-permutation.
This implies that $L$ is quasi-invertible, since it is quasi-permutation on restriction to its Sylow subgroups.

Now, we need to construct a flasque resolution of $L$.
Note that $L\cong L^*$ as a $G$-lattice. This means we can construct a coflasque resolution of $L$
and dualise.
It turns out that this is easy.  Any non-trivial subgroup $H$ of $G$ contains
$H_2=\langle \tau^4\rangle \cong C_2$ or $H_3=\langle \sigma\rangle\cong C_3$.
Since $\rho(\tau^4)=-I_4$ shows that $L^{\langle \tau^4\rangle}=0$
and $L^{H_3}=(IC_3\oplus IC_3)^{C_3}=0$, we see that for any non-trivial subgroup 
$H$ of $G$, we have $L^H=0$.
So to find a coflasque resolution of $L$, we need only find 
a permutation lattice which surjects onto $L$.
Since $L$ is an irreducible $G$-lattice,
there is a surjection $\pi:\Z G\to L$. If $K=\ker(\pi)$,
then 
$$0\to K\to \Z G\to L\to 0$$
is a coflasque resolution and its dual is 
a flasque resolution of $L\cong L^*$.
Let $F=K^*$.
Then we have the flasque resolution
$$0\to L\to \Z G\to F\to 0$$
Since $\Z G$ is free as a $G$-lattice, we have $\hat{H}^0(H,F)\cong \hat{H}^1(H,L)$ for any subgroup $H$ of 
$G$.
Now, $\hat{H}^0(H_3,F)\cong\hat{H}^1(H_3,L)\cong \hat{H}^1(C_3,IC_3)^2\cong (\Z/3\Z)^2$.

Since $L_{H_8}\cong \Ind^{C_8}_{C_2}\Z_-$ is a sign-permutation lattice
we see that $$\hat{H}^0(H_8,F)=\hat{H}^1(H_8,L)=\hat{H}^1(H_8,\Ind^{H_8}_{C_2}\Z_-)\cong 
\hat{H}^1(C_2,\Z_-)=\Z/2\Z.$$

Suppose $F\oplus P\cong Q$ for some permutation $G$-lattices $P$
and $Q$.
Setting $H_{24}=G$, representatives of the conjugacy classes of subgroups of $G$ can be given by $\{H_d, d|24\}$ as described earlier.
We may then write 
$P=\oplus_{d|24}a_d\Z[G/H_d]$, $Q=\oplus_{d|24}b_d\Z[G/H_d]$ for some 
$a_d,b_d\in \Z$.  We will obtain a contradiction by applying  $\hat{H}^0$
to the equation $F\oplus P\cong Q$ for the group $G$ and the subgroup $H_8$.

We observe that 
$\hat{H}^0(G,F)\cong \hat{H}^1(G,L)=0$.
Indeed, since $L\cong L^*$ as $G$-lattices, $\hat{H}^{1}(H,L)\cong
\hat{H}^{-1}(H,L^*)=\ker_L(N_H)/I_H(L)$ for any subgroup $H$ of $G$.
Since $H_3\cong C_3$ is a normal subgroup and is generated by $\sigma$
with image having minimal polynomial $1+x+x^2$, it is clear that
$\ker_L(N_{H_3})=L$.  As $H_3$ is normal in $G$, $N_G=\sum_{g\in G/{H_3}}gN_{H_3}$, and 
so $\ker_L(N_G)=L$.  Similarly, since $H_8\cong C_8$ is generated by $\tau$ 
with image having minimal polynomial $x^4+1$, we see that $N_{H_8}=\sum_{i=0}^7\tau^i
=\sum_{i=0}^3\tau^i(\tau^4+1)$ has $\ker_L(N_{H_8})=L$.  But we have seen that $\hat{H}^{-1}(H_8,L)\cong \hat{H}^1(H_8,L)\cong \Z/2\Z$ and $\hat{H}^{-1}(H_3,L)\cong \hat{H}^1(H_3,L)\cong (\Z/3\Z)$.  Since $\hat{H}^{-1}(H,L)=\ker_L(N_H)/I_H(L)$, we see that 
$3L\subseteq I_{H_3}(L)\subseteq I_G(L)$ and $2L\subseteq I_{H_8}(L)\subseteq I_G(L)$. Then we see that $I_G(L)=L$ and so $\hat{H}^1(G,L)=0$.

Applying $\hat{H}^0(G,\cdot)$ to $F\oplus P\cong Q$, and observing that $\hat{H}^0(G,\Z[G/H_i])\cong \Z/|H_i|\Z$,
we see that we have
\begin{equation}\label{eq:contradiction}
\sum_{d|24}x_d\Z/d\Z=0
\end{equation}
where $x_d=a_d-b_d\in \Z$.
Tensoring (\ref{eq:contradiction}) by $\Z/8\Z$ and equating coefficients of $\Z/2^i\Z$ for $i=1,2,3$, we obtain
$$x_2+x_6=0, x_4+x_{12}=0,x_8+x_{24}=0$$

Now restrict the isomorphism $F\oplus P\cong Q$
to $H_8$.
Note that $\hat{H}^0(H_8,F)=\Z/2\Z$.
By Mackey's Theorem,
$$\Res^G_{H_8}\Ind^G_{H_d}\Z=\oplus_{H_8xH_d\in H_8\setminus G/H_d}\Z[H_8/H_8\cap H_d^x]$$
For all $d\ne 8$, $H_d$ is a normal subgroup,
and so $H_8xH_d=H_8H_dx$ is a coset of $H_8H_d$ in $G$.  Then
$$\Res^G_{H_8}\Z[G/H_d]\cong \Z[H_8/H_8\cap H_d]^{[G:H_8H_d]}, d\ne 8$$
In particular, if $d=1,2,4$, then $H_d\le H_8$, $H_8\cap H_d=H_d$ and $H_8H_d=H_8$ and 
so 
$$\Res^G_{H_8}\Z[G/H_d]\cong \Z[H_8/H_d]^3, d=1,2,4$$
If $d=3(2^k), k=0,1,2,3$, then $H_d\cap H_8=H_{2^k}$, $H_8H_d=G$ and so 
$$\Res^G_{H_8}\Z[G/H_d]\cong \Z[H_8/H_{2^k}], d=2^k3, k=0,1,2,3$$

For $H_8$, we see that $G=H_81H_8\cup H_8\sigma H_8=H_8\cup H_8\sigma H_8$
and $H_8\cap H_8^{\sigma}=H_4$.
So 
$$\Res^G_{H_8}\Z[G/H_8]\cong\Z\oplus \Z[H_8/H_4]$$
Taking $\hat{H}^0(H_8,\cdot)$ of $F\oplus P\cong Q$ and comparing coefficients
of $\Z/2^i\Z$, $i=1,2,3$, we obtain
$$3x_2+x_6=1, 3x_4+x_8+x_{12}=0, x_{8}+x_{24}=0$$
Together with our previous set of equations, we obtain a contradiction
since $x_2+x_6=0$ and $3x_2+x_6=1$
implies $2x_2=1$ but $x_2\in \Z$.
The contradiction implies that the flasque lattice $F$ cannot be $G$-stably
permutation and so our lattice $L$ cannot be $G$-quasi-permutation.
\end{proof}
\section{Stable rationality of the torus corresponding to [4,25,8,5]}

In this section we present a non-computational proof that the torus
corresponding to [4,25,8,5] is stably rational.  In an earlier version, we
erroneously concluded that we could prove that this algebraic torus was
rational.  It is not clear to this author at this point whether this torus
is rational or not.  However, this example may shed some further light on the question of whether or not stably rational tori are all rational.  We provide a lot of information about this case for this reason.

Let $G$= MatGroupZClass(4,25,8,5).
This group is generated by the following elements
$$g_1=\left[\begin{array}{rrrr}-1&0&0&0\\0&-1&0&0\\0&0&0&-1\\0&0&-1&0\end{array}\right],
g_2=\left[\begin{array}{rrrr}1&0&1&0\\0&0&1&0\\0&0&0&1\\0&-1&0&0\end{array}\right],
g_3=\left[\begin{array}{rrrr}1&-1&1&0\\0&-1&0&0\\0&0&-1&0\\0&0&0&1\end{array}\right],
g_4=\left[\begin{array}{rrrr}1&0&1&1\\0&1&0&0\\0&0&-1&0\\0&0&0&-1\end{array}\right]$$
Let $M_4=\oplus_{i=1}^4\Z\e_i$. $G$ stabilises the sublattice $M_3=\oplus_{i=2}^4\Z\e_i$.
In fact, $M_3\cong \Z B_3$ as a faithful $G\cong W(B_3)$-lattice.  Under the restriction to $M_3$, $g_1$ acts as $\tau_1\tau_2\tau_3(23)$,
$g_2$ acts as $\tau_1(123)$, $g_3$ acts as $\tau_1\tau_2$ and $g_4$ acts as $\tau_2\tau_3$.
Then $(g_2)^3$ acts as $\tau_1\tau_2\tau_3$, so that it may be seen that the restriction generates the same group as $\langle \tau_1,\tau_2,\tau_3,(23),(123)\rangle=W(B_3)\cong C_2\times S_4$.
It can  then be checked that $G$ acts on $M_4/M_3\cong \Z_N^{-}$ where $N=\langle g_2,g_3,g_4\rangle=\langle g_2^3\rangle \times \langle g_3,g_4\rangle \rtimes \langle g_2^2\rangle\cong C_2\times A_4$.
Since $G$ acts faithfully on $M_3$, the restriction map 
$\GL(M_4)\to \GL(M_3)$ restricts to an isomorphism on $G$.
Then as in its restriction to $M_3$, we see that $\langle g_2^3\rangle\cong C_2$ is the centre of $G$, $\langle g_2^2\rangle$ normalizes the subgroup 
$\langle g_3,g_4\rangle \cong C_2\times C_2$ so that $\langle g_2^3,g_3,g_4\rangle\cong A_4$ and $\langle g_2,g_3,g_4\rangle =\langle g_2^3\rangle\times
\langle g_2^2,g_3,g_4\rangle\cong C_2\times A_4$. 

Since $N\cdot \e_1$ contains 
$$\langle g_2\rangle \cdot\e_1=\{\e_1,\e_1+\e_3,
\e_1+\e_3+\e_4,\e_1-\e_2+\e_3+\e_4,\e_1-\e_2+\e_4,\e_1-\e_2\},$$
we see that $|N\cdot \e_1|\ge 6$.  Since $g_2^2(\e_1)=g_4(\e_1)$, and $g_4$ has order 2, then $C_3\cong \langle g_4g_2^2\rangle \le N_{\e_1}$.  So $3| |
N_{\e_1}|$.  But then $|N_{\e_1}|=|N|/|N\cdot \e_1|\le 4$ shows that
$N_{\e_1}=\langle g_4g_2^2\rangle \cong C_3$ and so $|N\cdot \e_1|=8$.
We need only 2 other orbit elements in addition to the above and they 
can be seen to be $g_3\e_1=\e_1-\e_2+\e_3$ and $g_2g_3\e_1=\e_1+\e_4$.
The orbit $N\cdot \e_1$ determines the orbit $N\cdot t_1$.
Clearly $\frac{1}{N_{t_1}}\sum_{n\in N}nt_1\in K(M_4)^N$.
From the computation of the $N$ orbit of $\e_1$, we see that 
$$z=\frac{1}{N_{t_1}}\sum_{n\in N}nt_1=
t_1+t_1t_3+t_1t_3t_4+t_1t_2^{-1}t_3t_4+t_1t_2^{-1}t_4+t_1t_2^{-1}+t_1t_4+t_1t_2^{-1}t_3=t_1(1+t_2^{-1})(1+t_3)(1+t_4)$$
Using the generators of $N$, it is not difficult to double check that the
element $z=t_1(1+t_2^{-1})(1+t_3)(1+t_4)$ is indeed fixed by $N$.
We will write $\beta=(1+t_2^{-1})(1+t_3)(1+t_4)\in K(M_3)$ so that $z=t_1\beta$.
Then $K(M_4)=K(M_3)(t_1)=K(M_3)(z)$.  Since $z$ is fixed by $N$, we have
$K(M_4)^G=(K(M_3)(z))^G=(K(M_3)(z))^N)^{G/N}=(K(M_3)^N(z))^{G/N}=L(z)^{\sigma}$
where $\sigma=g_1N$ and $L=K(M_3)^N$
Now $\sigma(z)=z^{-1}\beta\sigma(\beta)$.
It is well known that $L(z)^\sigma$ is rational over $L^\sigma=K(M_3)^G$
if and only if $\beta\sigma(\beta)$ is a norm of the extension $L^N/L^G$
or equivalently if there exists $\alpha\in L^N$ such that $\beta\sigma(\beta)=
\alpha\sigma(\alpha)$~\cite[p. 160]{Se79}.  We will rephrase this condition for our particular situation.

\begin{defn} Let $G$ act on a field $L$ and let $M$ be a $G$ lattice.
Then $G$ acts quasi-monomially on the group ring $L[M]$ and hence its quotient field $L(M)$
if its action extends the action on the units $L[M]^*=L^*\oplus M$
from its given action on $L^*$ and $L[M]^*/L^*\cong M$
so that 
$$1\to L^*\to L[M]^*\to M\to 0$$ is a short exact sequence of $G$ modules.
\end{defn}

\begin{lemma} Let $L$ be a field on which $G$ acts faithfully.
Let $N$ be a normal subgroup of $G$ of index 2.
Let $G$ act on $L(\Z_N^-)=L(t)$ by a quasi monomial action, extending its action on $L$.  Then $L(\Z_N^-)^G$ is rational over $L^G$ if and only if
the action of $G$ on $L(\Z_N^-)$ can be extended to a quasi-monomial
action on $L(\Z[G/N])$ if and only if $b=gN(t)t\in N_{L^N/L^G}(L^N)^{\times})$.
\end{lemma}

\begin{proof}
Suppose the quasi-monomial action of $G$  on $L(\Z_N^-)$ can be extended to a 
quasi-monomial action on $L(\Z[G/N])$. We write $L(\Z[G/N])=L(X,Y)$ where
$X,Y$ are the multiplicative generators of $\Z[G/N]$ in $L[\Z[G/N]]^{\times}$ corresponding to a permutation basis.  Then if $\sigma=gN$ generates $G/N\cong C_2$, $\sigma(X)=\alpha Y$ for some $\alpha\in L^{\times}$.
Since $\Z[G/N]$ is fixed by $N$, we see that $\alpha\in (L^N)^{\times}$.
Then note that  $\sigma^2=1$ implies that $\sigma(\alpha)\sigma(Y)=X$ and so 
$\sigma(Y)=\frac{1}{\sigma(\alpha)}X$.  We assumed that the quasi-monomial action of $G$ on $L(\Z_N^{-})=L(t)$ can be extended to that on $L(\Z[G/N])=L(X,Y)$.
Then $t$ would be mapped to $X/Y$ and so $\sigma(t)t=
\frac{\sigma(X)X}{\sigma(Y)Y}=\sigma(\alpha)\alpha$ as required.
Then $L(\Z_N^-)^G\subset L(\Z[G/N])^G$.  The latter field is rational over $L^G$ and so $L(\Z_N^-)^G$ is unirational over $L^G$.  Since $L(\Z_N^-)^G$ is of transcendence degree 1 over $L^G$, this means it is rational over $L^G$.

Conversely, if there exists $\alpha\in (L^N)^{\times}$ such that $\sigma(t)t=
\sigma(\alpha)\alpha$, we may extend the quasi-monomial action of $G$ on $L(\Z_N^-)=L(t)$ to one on $L(\Z[G/N])=L(X,Y)$ by the same rule.
That is, if $\sigma$ is a generator of $G/N$, then set $\sigma(X)=\alpha Y$.
This implies that $\sigma(Y)=\frac{1}{\sigma(\alpha)}X$ and the 
inclusion mapping $t\to X/Y$ can be seen to be $G$ invariant.
Note that $N$ fixes $t,X,Y$ and so it is only necessary to determine the action of $G/N$.
\end{proof}

\begin{remark}
We could apply the lemma to our situation.  Here $L=K(M_3)$ and 
$G$ acts quasi-monomially on $K(M_4)=K(M_3)(\Z_N^-)=K(M_3)(z)$.
Then $K(M_4)^G=(K(M_3)^N(z))^{G/N}$ would be rational over $K(M_3)^G$ 
if and only if $\sigma(z)z=\prod_{i=2}^4(1+t_i)(1+t_i^{-1})$ can be written as a norm from the quadratic extension $K(M_3)^N/K(M_3)^G$.
Unfortunately, as we do not have an explicit description of the field $K(M_3)^N$, it is not obvious to the author how to solve such a norm equation.
If this were the case, then $K(M_4)^N$ would be rational over $K(M_3)^G$ since 
we already know that $K(M_3)^G$ is rational over $K^G$ as $M_3$ is a sign permutation lattice.  However, if it is not the case, which may well be true,
this does not solve the rationality problem for $K(M_4)^G$.
Finding a subfield $E$ such that $K(M_4)^G/E$ is non-rational but $E/K^G$ is rational does not rule out the rationality of $K(M_4)^G/K^G$.
\end{remark}

We can however give an explicit proof of the stably rationality of $K(M_4)^G$ over $K^G$.

We first construct a coflasque resolution of $M_4^*$.
Let $\e_i^*$,$i=1,\dots,4$ be the dual basis of the standard basis $\e_1,\dots \e_4$ and recall that the matrix of the
action of each group element on a dual lattice with respect to the dual basis is the inverse transpose of its matrix with respect to the original basis.

We need to determine a permutation $G$-lattice which surjects onto $M_4^*$.
Note that $G\cdot \e_1^*=\{\pm \e_1^*\}$.  It is clear that $N=\langle g_2,g_3,g_4\rangle\subset G_{\e_1^*}$.
We have seen that $N\cong C_2\times A_4$, so $N$ is of index 2 in $G$ and so 
$G_{\e_1^*}=N$.
Note also that $\langle g_2\rangle$ has orbit $\{\e_2^*,-\e_4^*,-\e_3^*,-\e_1^*-\e_2^*,-\e_1^*+\e_4^*,-\e_1^*-\e_3^*\}$.
Since $g_1\e_2^*=-\e_2^*$, it is clear that 
$$\gamma=\{\pm \e_2^*,\pm \e_3^*,\pm \e_4^*, \pm (\e_1^*+\e_2^*), \pm (\e_1^*-\e_3^*), \pm (\e_1^*-\e_4^*)\}$$
 is contained in $G\cdot \e_2^*$.
Since $g_4\e_2^*=\e_2^*$, $g_3(\e_2^*)=g_2^3(\e_3^*)=-\e_1^*-\e_2^*$, and 
$g_3^2=g_4^2=1$, we see that $G_{\e_2^*}$ contains $\langle g_3g_2^3,g_4\rangle\cong C_2\times C_2$.
Then since $|G\cdot \e_2^*|\ge 12$ implies $|G_{\e_2^*}|\le 4$,
the orbit calculation shows that we must have $G_{\e_2^*}=\langle g_3g_2^3,g_4\rangle$  and $G\cdot \e_2^*=\gamma$.

Since $\Z G\cdot\e_1^*+\Z G\cdot\e_2^*$ contains $\{\e_1^*,\e_2^*,\e_3^*,\e_4^*\}$, we see that the map of $G$-lattices
$\pi:\Z[G/G_{\e_1^*}]\oplus \Z[G/G_{\e_2^*}]\to M_4^*$
sending $gG_{\e_1*}\to g\cdot \e_1^*$ and $gG_{\e_2^*}\to g\cdot \e_2^*$
is surjective.  

Note for a finite group $G$ acting on a set $X$,   a permutation $\Z$-basis of a transitive permutation $G$-lattice
$\Z[G/G_x]$ for some $x\in X$, is in bijection with the elements of the orbit $G\cdot x$.
We will write $p(g\cdot x)$ in place of $gG_x$.
Then this basis of $\Z[G/G_{\e_1^*}]\oplus \Z[G/G_{\e_2^*}]$ will be written as 
$$\{p(x): x\in G\cdot \e_1^*\}\cup \{p(y): y\in G\cdot \e_2^*\}$$
The map $\pi$ then sends $p(x)\to x$ for all $x\in G\cdot \e_1^*\cup G\cdot \e_2^*$.
As $x\in G\cdot \e_i^*, i=1,2$ if and only if $-x\in G\cdot \e_i^*$,
it is then easy to see that $p(x)+p(-x)$ is in $\ker(\pi)$ for all
$x\in G\cdot \e_1^*\cup G\cdot \e_2^*$ and that the set
$$\{p(x)+p(-x): x\in G\cdot \e_1^*\cup G\cdot \e_2^*\}$$
is $G$-stable.
So we see that $\ker(\pi)$ has a permutation $G$-sublattice
$$Q=\Z[G/G_{p(\e_1^*)+p(-\e_1^*)}]\oplus \Z[G/G_{p(\e_2^*)+p(-\e_2^*)}]\cong \Z\oplus \Z[G/D_8]$$
where we note that $G/G_{p(\e_2^*)+p(-\e_2^*)}\cong \langle g_1,g_3g_2^3,g_4\rangle\cong D_8$ since $g_1(g_3g_2^3)g_1^{-1}=g_4$.
From the calculation of the $G$-orbits $G\cdot \e_1^*$ and $G\cdot \e_2^*$
above, we see that a permutation basis of $Q$ is
given by 
$$\beta=\{p(x)+p(-x): x\in \{\e_i^*:i=1,\dots,4\}\cup \{\e_1^*+\e_2^*,\e_1^*-\e_3^*,\e_1^*-\e_4^*\}\}$$

A straightforward calculation 
 shows that 
we may extend the permutation basis 
$\beta$
for $Q$ to 
a $\Z$-basis for $C=\ker(\pi)$ by adding the vectors
$\{r_{1i}:i=2,3,4\}$
where 
$$r_{12}=p(\e_1^*+\e_2^*)-p(\e_1^*)-p(\e_2^*),
r_{13}=p(\e_1^*-\e_3^*)-p(\e_1^*)+p(\e_3^*),
r_{14}=p(\e_1^*-\e_4^*)-p(\e_1^*)+p(\e_4^*)$$



As we have seen, $G$ acts by permutations on the sublattice $Q\cong \Z\oplus \Z[G/D_8]$ of $C$.  We will be interested in the action of $G$ on the remaining basis vectors and hence on $C/Q$.
As our calculations will involve the restrictions to the Sylow subgroups of $G$, we will replace the generators with 
$\{g_1,g_2^3,g_2^3g_3,g_4,g_2^4\}$.  We note that a Sylow 3-subgroup of $G$ is $\langle g_2^4\rangle\cong C_3$ and a Sylow 2-subgroup of $G$ is $\langle g_2^3,g_1,g_2^3g_3,g_4\rangle
\cong C_2\times D_8$.
Note that $\langle g_2^3\rangle\cong C_2$ is central and $g_1$ conjugates $g_2^3g_3$ to $g_4$.

We will write the basis vectors for $Q$ as $s_i=p(\e_i^*)+p(-\e_i^*)$, $i=1,\dots,4$,  $$s_{12}=p(\e_1^*+\e_2^*)+p(-(\e_1^*+\e_2^*)),
s_{13}=p(\e_1^*-\e_3^*)+p(-(\e_1^*-\e_3^*)),
s_{14}=p(\e_1^*-\e_4^*)+p(-(\e_1^*-\e_4^*))$$
 We will compute the action of the generators of $G$ on $C$:

$g_1$ swaps $\e_1^*$ and $-\e_1^*$, $\e_2^*$ and $-\e_2^*$,
$\e_3^*$ and $-\e_4^*$, 
This shows that 
$g_1$ fixes $s_1,s_2,s_{12}$ and swaps $s_3$ and $s_4$, $s_{13}$ and $s_{14}$.
$g_1(r_{12})
=s_{12}-s_1-s_2-r_{12}$, 
$g_1(r_{13})
=s_{14}-s_1+s_4-r_{14}$, 
$g_1(r_{14})
=s_{13}-s_1+s_3-r_{13}$.

$g_2^3$ fixes $\e_1^*$,
and swaps $\e_2^*$ and $-(\e_1^*+\e_2^*)$, 
$\e_3^*$ and $\e_1^*-\e_3^*$, 
$\e_4^*$ and $\e_1^*-\e_4^*$.
This shows that 
$g_2^3$ fixes $s_1$, and swaps $s_2$ and $s_{12}$; $s_3$ and $s_{13}$; $s_4$ and $s_{14}$.
$g_2^3(r_{12})
=-s_{12}+s_2+r_{12}$
and $g_2^3$ fixes $r_{13}$ and $r_{14}$.

$g_2^3g_3$ fixes $\e_1^*,\e_2^*,\e_3^*$, and swaps $\e_4^*$ with $\e_1^*-\e_4^*$.
Then 
$g_2^3g_3$ fixes $s_1,s_2,s_3,s_{12},s_{13}$ and swaps $s_4$ and $s_{14}$.
This shows also that $g_2^3g_3$ fixes $r_{12},r_{13},r_{14}$.

$g_4$ fixes $\e_1^*,\e_2^*,\e_4^*$, and swaps $\e_3^*$ with $\e_1^*-\e_3^*$.
Then 
$g_4$ fixes $s_1,s_2,s_4,s_{12},s_{14}$ and swaps $s_3$ and $s_{13}$.
This shows also that $g_4$ fixes $r_{12},r_{13},r_{14}$.

$g_2^4$ fixes $\e_1^*$ and sends $\e_2^*$ to $-\e_3^*$, $\e_3^*$ to $\e_1^*-\e_4^*$, $\e_4^*$ to $\e_1^*+\e_2^*$.
Then $g_2^4$ fixes $s_1$ and sends $s_2$ to $s_3$, $s_3$ to $s_{14}$, $s_4$ to $s_{12}$, $s_{12}$ to $s_{13}$, $s_{13}$ to $s_4$ and $s_{14}$ to $s_2$.
Then $g_2^4(r_{12})
=r_{13}-s_3$, $g_2^4(r_{13})
=r_{14}$, $g_2^4(r_{14})
=r_{12}+s_2$.

From this calculation, it is clear that $G$ acts on $C/Q$ as a sign permutation lattice.  In fact, $\langle g_2^3,g_2^3g_3,g_4\rangle\cong C_2^3$ fix $C/Q$
pointwise, $g_1(r_{12}+Q)=-r_{12}+Q$, $g_1(r_{13}+Q)=-r_{14}+Q$ and 
$g_1(r_{14}+Q)=-r_{13}+Q$, $g_2^4$ acts as $r_{12}+Q\to r_{13}+Q\to r_{14}+Q\to
r_{12}+Q$.

We will show that $C$ is invertible which in turn implies that it is coflasque and flasque.  
According to ~\cite{CW90}, to check that $C$ is invertible it suffices to check that $\F_pC$ is $\F_p\Syl_p$ permutation for each
Sylow $p$-subgroup $\Syl_p$ of $G$ and for any Sylow $2$ subgroup $\Syl_2$, additionally
we must check that $\dim_{\Q}(\Q C)^{\Syl_2}=\dim_{\F_2}(\F_2 C)^{\Syl_2}$.

For $p=3$, we note that a Sylow $p$-subgroup $\langle g_2^4\rangle\cong C_3$ 
acts by permutations on $C/Q$ and so $0\to Q\to C\to C/Q\to 0$ splits 
as a $C_3=\langle g_2^4\rangle$ lattice.  Then $C$ restricted to $C_3$
is isomorphic to $Q\oplus C/Q$ restricted to $C_3$ and so is permutation.
In fact $C_{C_3}\cong \Z C_3^3\oplus \Z$.  Then it is clear that $\F_3 C$
is also $\F_3 C_3$ permutation.

For $p=2$, from the earlier description of the action of $G$ on $C$, one can check that the following is a permutation basis for $\F_2C$ as an $\F_2\Syl_2$
module where $\Syl_2=\langle g_1,g_2^3,g_2^3g_3,g_4\rangle$:
$$\{s_2,s_3,s_4,s_{12},s_{13},s_{14},r_{12}+s_1+s_2,r_{12}+s_{12},r_{14},
r_{13}+s_{13}+s_1+s_3\}$$




A calculation also shows  that the orbits of this permutation basis under $\Syl_2$ 
are $\{s_2,s_{12}\}$, $\{s_3,s_4,s_{13},s_{14}\}$, $\{r_{12}+s_1+s_2,r_{12}+s_{12}\}$ and $\{r_{14},r_{13}+s_{13}+s_1+s_3\}$
and so $\dim_{\F_2}(\F_2C)^{\Syl_2}=4$.

On the other hand, from the exact sequence $0\to Q\to C\to C/Q\to 0$,
we see that 
$$0\to Q^{\Syl_2}\to C^{\Syl_2}\to (C/Q)^{\Syl_2}\to 0$$
since $Q$ is permutation and so is coflasque.
Since the permutation lattice $Q$ has orbits
$\{s_1\}\cup \{s_2,s_{12}\}\cup \{s_3,s_4,s_{13},s_{14}\}$ under $\Syl_2$, we see that
$\rk(Q^{\Syl_2})=3$.  Since $C/Q$ is sign permutation and is fixed by $\langle g_2^3,g_2^3g_3,g_4\rangle$,  then $(C/Q)^{\Syl_2}=(C/Q)^{\langle g_1\rangle}=\Z (r_{13}-r_{14})$.
So
$\dim_{\Q}(\Q C)^{\Syl_2}=\rk(C^{\Syl_2})=\rk(Q^{\Syl_2})+\rk((C/Q)^{\Syl_2})
=3+1=4$.  Since this agrees with $\dim_{\F_2}(\F_2C)^{\Syl_2}$ and $\F_pC$
is $\F_p\Syl_p$-permutation for all Sylow $p$-subgroups $\Syl_p$ of $G$, then 
$C$ is invertible as required.

[Note that although
$$0\to \F_2Q\to \F_2C\to \F_2(C/Q)\to 0$$
is an exact sequence of $\F_2\Syl_2$ modules with 2 end modules permutation,
this sequence does not split.   This can already be seen when restricting the exact sequence above to the subgroup $\langle g_1\rangle$.
This means that one cannot use the same argument as in the $p=3$ where the 
sequence splits integrally.]

So we now have a coflasque resolution for $M_4^*$ for which the coflasque lattice is invertible.  This gives us a flasque resolution for $M_4$ for which the 
flasque lattice $C^*$ is invertible. We need to show that $C$ is stably permutation, which would then imply the same for $C^*$.

Since $C/Q$ is sign-permutation, it is quasi-permutation and self-dual, so that
there must exist a sequence $0\to P_0\to P_1\to C/Q\to 0$ where $P_0,P_1$ are 
permutation.  Explicitly, $C/Q\cong \Ind_{\Syl_2}^G\Z_{C_2^3}^-$,
Indeed,  for any transitive sign permutation lattice $S$ for a group $G$ 
with $v$ an element of the sign permutation basis for $S$, if $G_v$ is the stabilizer subgroup of $v$ and $G_v^{\pm}=\{g\in G:gv=\pm v\}$,
we may see that $S\cong \Ind_{G_v^{\pm}}^G\Z_{G_v}^-$, via the map 
$gv\mapsto g\otimes z, g\in G$ where $z$ is a basis of the rank 1 sign lattice $\Z_{G_v}^-$ of 
$G_v^{\pm}$ with kernel $G_v$.
In our case $C/Q$ is a transitive sign permutation lattice with sign permutation basis $\{\overline{r_{12}},\overline{r_{13}},\overline{r_{14}}\}$.
It is easy to see that $C_2^3$ fixes $\overline{r_{12}}$ pointwise and that the
orbit of $\overline{r_{12}}$ under $G$ is 
$\{\pm\overline{r_{12}},\pm\overline{r_{13}},\pm\overline{r_{14}}\}$.
So the stabilizer subgroup of $\overline{r_{12}}$ is $C_2^3$.
It is similarly easy to check that $G_{\overline{r_{12}}}^{\pm}$ is $\Syl_2$.
Then the exact sequence of $\Syl_2$ lattices
$$0\to Z\to \Z[\Syl_2/C_2^3]\to \Z_{C_2^3}^-\to 0$$
induces an exact sequence of $G$ lattices
$$0\to \Z[G/\Syl_2]\to \Z[G/C_2^3]\to \Ind_{\Syl_2}^G\Z_{C_2^3}^-\to 0$$
Since the last term of this exact sequence is isomorphic to $C/Q$,
we can produce a pullback diagram:
$$
\xymatrix{  
&\Z[G/\Syl_2] \ar[r]^{=}\ar@{>->}[d] &\Z[G/\Syl_2]\ar@{>->}[d]  \\
  Q \ar@{>->}[r]\ar[d]^{=} & \mbox{pull-back} \ar@{->>}[r]\ar@{->>}[d] 
&\ar@{->>}[d]\Z[G/C_2^3]\\  
  Q \ar@{>->}[r]& C \ar@{->>}[r] &\Ind_{\Syl_2}^G\Z_{C_2^3}^- 
}
$$

Since $Q$ and $\Z[G/C_2^3]$ are permutation, the middle row splits.
Also, since we have shown that $C$ is invertible and $\Z[G/\Syl_2]$ is permutation, we also have $\Ext^1_G(C,\Z[G/\Syl_2])=0$ since $C$ is the direct summand of a permutation lattice.  Thus, also the middle column splits.
We then see that 
$$C\oplus \Z[G/\Syl_2]\cong Q\oplus \Z[G/C_2^3]=\Z\oplus \Z[G/D_8]\oplus \Z[G/C_2^3]$$
is stably permutation as required.
This shows that $M_4$ is $G$ quasi-permutation as it then satisfies a $G$ exact sequence:
$$0\to M_4\to \Z[G/N]\oplus \Z[G/C_2^2]\oplus \Z[G/\Syl_2]\to 
\Z\oplus \Z[G/D_8]\oplus \Z[G/C_2^3]\to 0$$

\bibliographystyle{alpha}

\begin{thebibliography}{BCTSSD85}

\bibitem[BBN{\etalchar{+}}78]{BBNWZ78}
Harold Brown, Rolf B{\"u}low, Joachim Neub{\"u}ser, Hans Wondratschek, and Hans
  Zassenhaus.
\newblock {\em Crystallographic groups of four-dimensional space}.
\newblock Wiley-Interscience [John Wiley \& Sons], New
  York-Chichester-Brisbane, 1978.
\newblock Wiley Monographs in Crystallography.

\bibitem[BCTSSD85]{BCTSSD85}
Arnaud Beauville, Jean-Louis Colliot-Th{\'e}l{\`e}ne, Jean-Jacques Sansuc, and
  Peter Swinnerton-Dyer.
\newblock Vari\'et\'es stablement rationnelles non rationnelles.
\newblock {\em Ann. of Math. (2)}, 121(2):283--318, 1985.

\bibitem[Ben98]{Ben98}
Esther Beneish.
\newblock Induction theorems on the stable rationality of the center of the
  ring of generic matrices.
\newblock {\em Trans. Amer. Math. Soc.}, 350(9):3571--3585, 1998.

\bibitem[BLB91]{BL91}
Christine Bessenrodt and Lieven Le~Bruyn.
\newblock Stable rationality of certain {${\rm PGL}_n$}-quotients.
\newblock {\em Invent. Math.}, 104(1):179--199, 1991.

\bibitem[BN70]{BN71}
R.~B{\"u}low and J.~Neub{\"u}ser.
\newblock On some applications of group-theoretical programmes to the
  derivation of the crystal classes of {$R_{4}$}.
\newblock In {\em Computational {P}roblems in {A}bstract {A}lgebra ({P}roc.
  {C}onf., {O}xford, 1967)}, pages 131--135. Pergamon, Oxford, 1970.

\bibitem[Bro82]{Bro82}
Kenneth~S. Brown.
\newblock {\em Cohomology of groups}, volume~87 of {\em Graduate Texts in
  Mathematics}.
\newblock Springer-Verlag, New York-Berlin, 1982.

\bibitem[CK00]{CK00}
Anne Cortella and Boris Kunyavski{\u\i}.
\newblock Rationality problem for generic tori in simple groups.
\newblock {\em J. Algebra}, 225(2):771--793, 2000.

\bibitem[CR87]{CR87}
Charles~W. Curtis and Irving Reiner.
\newblock {\em Methods of representation theory. {V}ol. {II}}.
\newblock Pure and Applied Mathematics (New York). John Wiley \& Sons, Inc.,
  New York, 1987.
\newblock With applications to finite groups and orders, A Wiley-Interscience
  Publication.

\bibitem[CR90]{CR90}
Charles~W. Curtis and Irving Reiner.
\newblock {\em Methods of representation theory. {V}ol. {I}}.
\newblock Wiley Classics Library. John Wiley \& Sons, Inc., New York, 1990.
\newblock With applications to finite groups and orders, Reprint of the 1981
  original, A Wiley-Interscience Publication.

\bibitem[CTS77]{CTS77}
Jean-Louis Colliot-Th{\'e}l{\`e}ne and Jean-Jacques Sansuc.
\newblock La {$R$}-\'equivalence sur les tores.
\newblock {\em Ann. Sci. \'Ecole Norm. Sup. (4)}, 10(2):175--229, 1977.

\bibitem[CTS87]{CTS87}
Jean-Louis Colliot-Th{\'e}l{\`e}ne and Jean-Jacques Sansuc.
\newblock Principal homogeneous spaces under flasque tori: applications.
\newblock {\em J. Algebra}, 106(1):148--205, 1987.

\bibitem[CW90]{CW90}
Gerald Cliff and Alfred Weiss.
\newblock Summands of permutation lattices for finite groups.
\newblock {\em Proc. Amer. Math. Soc.}, 110(1):17--20, 1990.

\bibitem[Dad65]{Dad65}
E.~C. Dade.
\newblock The maximal finite groups of {$4\times 4$} integral matrices.
\newblock {\em Illinois J. Math.}, 9:99--122, 1965.

\bibitem[EM73a]{EM73}
Shizuo End{\^o} and Takehiko Miyata.
\newblock Invariants of finite abelian groups.
\newblock {\em J. Math. Soc. Japan}, 25:7--26, 1973.

\bibitem[EM73b]{EM73p1}
Shizuo End{\^o} and Takehiko Miyata.
\newblock Quasi-permutation modules over finite groups.
\newblock {\em J. Math. Soc. Japan}, 25:397--421, 1973.

\bibitem[EM74]{EM74}
Shizuo End{\^o} and Takehiko Miyata.
\newblock Quasi-permutation modules over finite groups. {II}.
\newblock {\em J. Math. Soc. Japan}, 26:698--714, 1974.

\bibitem[EM75]{EM75}
Shizuo End{\^o} and Takehiko Miyata.
\newblock On a classification of the function fields of algebraic tori.
\newblock {\em Nagoya Math. J.}, 56:85--104, 1975.

\bibitem[EM76]{EM76}
Shizuo End{\'o} and Takehiko Miyata.
\newblock On the projective class group of finite groups.
\newblock {\em Osaka J. Math.}, 13(1):109--122, 1976.

\bibitem[End11]{End11}
Shizuo Endo.
\newblock The rationality problem for norm one tori.
\newblock {\em Nagoya Math. J.}, 202:83--106, 2011.

\bibitem[Flo]{Flo1}
Mathieu Florence.
\newblock A short proof of klyachko's theorem about rational algebraic tori.
\newblock Available on author's website.

\bibitem[Flo06]{Flo2}
Mathieu Florence.
\newblock Non rationality of some norm-one tori.
\newblock Available on author's website, 2006.

\bibitem[GAP15]{GAP4}
The GAP~Group.
\newblock {\em {GAP -- Groups, Algorithms, and Programming, Version 4.7.8}},
  2015.

\bibitem[Hum72]{Hum72}
James~E. Humphreys.
\newblock {\em Introduction to {L}ie algebras and representation theory}.
\newblock Springer-Verlag, New York-Berlin, 1972.
\newblock Graduate Texts in Mathematics, Vol. 9.

\bibitem[HY12]{HY12}
Akinari Hoshi and Aiichi Yamasaki.
\newblock Rationality problem for algebraic tori, 2012.

\bibitem[Kly88]{Kly88}
A.~A. Klyachko.
\newblock On the rationality of tori with cyclic splitting field.
\newblock In {\em Arithmetic and geometry of manifolds ({R}ussian)}, pages
  73--78, 112. Ku\u\i byshev. Gos. Univ., Kuybyshev, 1988.

\bibitem[Kun87]{Kun90}
B.~{\`E}. Kunyavski{\u\i}.
\newblock Three-dimensional algebraic tori.
\newblock In {\em Investigations in number theory ({R}ussian)}, pages 90--111.
  Saratov. Gos. Univ., Saratov, 1987.
\newblock Translated in Selecta Math. Soviet. {{\bf{9}}} (1990), no. 1, 1--21.

\bibitem[Kun07]{Kun07}
Boris Kunyavskii.
\newblock Algebraic tori - thirty years after, 2007.

\bibitem[LB95]{Leb95}
Lieven Le~Bruyn.
\newblock Generic norm one tori.
\newblock {\em Nieuw Arch. Wisk. (4)}, 13(3):401--407, 1995.

\bibitem[Len74]{Len74}
H.~W. Lenstra, Jr.
\newblock Rational functions invariant under a finite abelian group.
\newblock {\em Invent. Math.}, 25:299--325, 1974.

\bibitem[LL00]{LL00}
Nicole Lemire and Martin Lorenz.
\newblock On certain lattices associated with generic division algebras.
\newblock {\em J. Group Theory}, 3(4):385--405, 2000.

\bibitem[Lor05]{Lor05}
Martin Lorenz.
\newblock {\em Multiplicative invariant theory}, volume 135 of {\em
  Encyclopaedia of Mathematical Sciences}.
\newblock Springer-Verlag, Berlin, 2005.
\newblock Invariant Theory and Algebraic Transformation Groups, VI.

\bibitem[Sal84]{Sa84a}
David~J. Saltman.
\newblock Noether's problem over an algebraically closed field.
\newblock {\em Invent. Math.}, 77(1):71--84, 1984.

\bibitem[Ser79]{Se79}
Jean-Pierre Serre.
\newblock {\em Local fields}, volume~67 of {\em Graduate Texts in Mathematics}.
\newblock Springer-Verlag, New York-Berlin, 1979.
\newblock Translated from the French by Marvin Jay Greenberg.

\bibitem[Tah71]{Tah71}
Ken~ichi Tahara.
\newblock On the finite subgroups of {${\rm GL}(3,\,Z)$}.
\newblock {\em Nagoya Math. J.}, 41:169--209, 1971.

\bibitem[Vos67]{Vos67}
V.~E. Voskresenski{\u\i}.
\newblock On two-dimensional algebraic tori. {II}.
\newblock {\em Izv. Akad. Nauk SSSR Ser. Mat.}, 31:711--716, 1967.

\bibitem[Vos70]{Vos70}
V.~E. Voskresenski{\u\i}.
\newblock Birational properties of linear algebraic groups.
\newblock {\em Izv. Akad. Nauk SSSR Ser. Mat.}, 34:3--19, 1970.

\bibitem[Vos74]{Vos74}
V.~E. Voskresenski{\u\i}.
\newblock Stable equivalence of algebraic tori.
\newblock {\em Izv. Akad. Nauk SSSR Ser. Mat.}, 38:3--10, 1974.

\bibitem[Vos82]{Vos83}
V.~E. Voskresenski{\u\i}.
\newblock Projective invariant {D}emazure models.
\newblock {\em Izv. Akad. Nauk SSSR Ser. Mat.}, 46(2):195--210, 431, 1982.

\bibitem[Vos98]{Vos98}
V.~E. Voskresenski{\u\i}.
\newblock {\em Algebraic groups and their birational invariants}, volume 179 of
  {\em Translations of Mathematical Monographs}.
\newblock American Mathematical Society, Providence, RI, 1998.
\newblock Translated from the Russian manuscript by Boris Kunyavski [Boris
  {\`E}. Kunyavski{\u\i}].

\end{thebibliography}
\newcommand{\etalchar}[1]{$^{#1}$}

\end{document}